\documentclass[11pt,article]{amsart}
\usepackage{amsmath,verbatim,tikz,amscd,amssymb}
\usetikzlibrary{arrows,patterns}
\allowdisplaybreaks

\usepackage{xcolor}
\newcommand{\red}{\color{red}}

\newtheorem{theorem}{Theorem}[section]
\newtheorem{lemma}[theorem]{Lemma}
\newtheorem{proposition}[theorem]{Proposition}
\newtheorem{corollary}[theorem]{Corollary}
\newtheorem{question}[theorem]{Question}

\theoremstyle{definition}
\newtheorem{definition}[theorem]{Definition}
\newtheorem{example}[theorem]{Example}

\newtheorem{remark}[theorem]{Remark}

\newenvironment{proofof}[1]{\noindent{\it Proof of
#1.}}{\hfill$\square$\\\mbox{}}

\def\cocoa{{\hbox{\rm C\kern-.13em o\kern-.07em C\kern-.13em o\kern-.15em A}}}

\tikzstyle{new style 0}=[fill=white, draw=black, shape=circle, scale=0.5]
\tikzstyle{new style 1}=[fill=white, draw={rgb,255: red,191; green,0; blue,64}, shape=circle, scale=0.5]
\tikzstyle{new style 2}=[fill=white, draw=black, shape=rectangle]
\tikzstyle{small_fill}=[fill=black, draw=black, shape=circle, scale=0.2]
\tikzstyle{none}=[fill=none, draw=none, shape=circle, scale=0.5]
\tikzstyle{new edge style 0}=[->]
\tikzstyle{new edge style 1}=[-, draw={rgb,255: red,191; green,0; blue,64}]
\tikzstyle{up_curve}=[-, out=15, in=165]
\tikzstyle{new edge style 2}=[->, draw={rgb,255: red,0; green,0; blue,255}, thick, >=stealth]

\begin{document} 

\title{Low dimensional flow polytopes and their toric ideals}
\author{M. Domokos 
\and D. Jo\'o} 
\thanks{Partially supported by  National Research, Development and Innovation Office,  NKFIH K 119934, K 132002 and PD 121410.  }



\date{}
\address{Alfr\'ed R\'enyi Institute of Mathematics\\
Re\'altanoda u. 13-15, 1053 Budapest, Hungary} 
\email{domokos.matyas@renyi.hu} \quad \email{joo.daniel@renyi.hu }
\maketitle

\begin{abstract} The toric ideal of a $d$-dimensional flow polytope has an initial ideal 
generated by square-free monomials of degree at most $d$. 
The toric ideal of a flow polytope of dimension at most four has 
an initial ideal generated by square-free monomials of degree at most two,  
with the only exception of the four-dimensional Birkhoff polytope, whose toric ideal has an initial ideal generated by a
square-free cubic monomial. 
The proof is based on a method to classify certain compressed flow polytopes, and 
a construction of a 
quadratic pulling  triangulation of them.   
Along the way compressed flow polytopes are classified up to dimension four, 
and their Ehrhart polynomials are computed. 
\end{abstract}

\section{Introduction} 

In this paper we study flow polytopes and their toric ideals. 
Given a lattice polytope $\nabla$, one associates with it an ideal $\mathcal{I}_{\nabla}$ generated by 
binomials in a polynomial ring whose variables are labeled by the lattice points in $\nabla$. 
It was conjectured in \cite{diaconis-eriksson} that the ideals of the Birkhoff polytopes are generated in degree $3$. 
The conjecture was proved in \cite{yamaguchi-ogawa-takemura} (dealing in fact with a class of transportation polytopes, that includes the Birkhoff polytopes). Using this result, the following was proved   
in \cite{domokos-joo} (it is stated there for $K=\mathbb{C}$, but the proof obviously holds for any base field $K$): 

\begin{theorem}\label{thm:domokos-joo}  (\cite[Theorem 9.3]{domokos-joo}) 
For any flow polytope $\nabla$ the ideal $\mathcal{I}_{\nabla}$ is generated by binomials of degree at most $3$. 
\end{theorem}  

We note also that it follows from known general arguments that the toric ideal of any  $d$-dimensional flow polytope has a Gr\"obner basis generated by elements of degree at most $d$ (see 
Proposition~\ref{prop:general} 
for  explanation).  In view of this and Theorem~\ref{thm:domokos-joo},  
the following questions naturally arise: 

\begin{question}\label{question:cubic Groebner} 
\begin{itemize}  
\item[(i)] Does $\mathcal{I}_{\nabla}$ have a cubic Gr\"obner basis for each flow polytope $\nabla$? 
If not, give a good degree bound for the generators of an appropriate initial ideal of  $\mathcal{I}_{\nabla}$. 
\item[(ii)] What are the flow polytopes $\nabla$ for which $\mathcal{I}_{\nabla}$ is quadratically generated? 
\item[(iii)]  What are the flow polytopes $\nabla$ for which $\mathcal{I}_{\nabla}$ 
has a quadratic Gr\"obner basis? 
\end{itemize}
\end{question} 
Question~\ref{question:cubic Groebner} 
was answered for $3\times 3$ transportation polytopes (a special class of flow polytopes of dimension at most $4$) 
in \cite[Theorem 1]{haase-paffenholz}. Generalizing this result, in this paper we consider \emph{all} flow polytopes of dimension at most $4$, 
and in Corollary~\ref{cor:cubic groebner} we give an answer for Question~\ref{question:cubic Groebner}, showing that up to dimension $4$ a flow polytope always has a quadratic Gr\"obner basis unless it is equivalent to the Birkhoff polytope, which has a cubic Gr\"obner basis. We note that the more general context of flow polytopes (as opposed to transportation polytopes considered in \cite{haase-paffenholz}) allowed us to sharpen \cite[Theorem 1]{haase-paffenholz}, as the methods of the present paper settle   the proper multiples of the 4-dimensional Birkhoff polytope $B_3$ as well.  
The fact that the toric ideal of $nB_3$ 
has a quadratic Gr\"obner basis for $n>1$ and has a square-free quadratic initial ideal for $n$ divisible by $2$ or $3$ is known, see 
\cite{ohsugi-hibi_2010}. 

We use a combination of the method of  Haase and Paffenholz \cite{haase-paffenholz} and results from our paper \cite{domokos-joo}. 
A flow polytope has a natural hyperplane subdivision, whose cells are compressed flow polytopes. 
The results of \cite{domokos-joo} offer a transparent way to classify in low dimension the relevant  compressed flow polytopes (see 
Corollary~\ref{cor:3-regular}, Proposition~\ref{prop:P(Gamma)}, Proposition~\ref{prop:all compressed}), and to construct appropriate triangulations for them. The existence of these triangulations  
yields the desired statements on Gr\"obner bases by well-known results from \cite{sturmfels}. 
Along the way we classify all compressed flow polytopes of dimension at most $4$ 
(see Proposition~\ref{prop:all 3-dim compressed} and Proposition~\ref{prop:all 4-dim compressed}), and find a pulling triangulation for each prime compressed flow polytope of dimension $3$ or $4$ 
(see Proposition~\ref{prop:triangulation of 4-cells} and Proposition~\ref{prop:triangulation of new 4-cells}).  Finally we use this information to compute the Ehrhart polynomials of the prime compressed $3$ and $4$-dimensional flow polytopes (see Proposition~\ref{prop:ehrhart}).  

We close this introduction by mentioning sources of interest in flow polytopes. The notion arises naturally in combinatorial optimization, in the study of integer network flows (see \cite{schrijver}). A notable subclass is the class of transportation polytopes, including the Birkhoff polytopes. Toric ideals of Birkhoff  polytopes were applied in algebraic statistics in 
\cite{diaconis-sturmfels} and \cite{diaconis-eriksson}. 
Ehrhart polynomials of certain flow polytopes 
were intensively studied in recent years 
(see for example \cite{beck-pixton},  \cite{baldonisilva-deloera-vergne}, \cite{baldoni-vergne}, \cite{meszaros-morales}), motivated partly by the observation that in certain cases they count 
the dimension of weight spaces of representations of compact Lie groups. The toric varieties associated to flow polytopes are special cases of moduli spaces of quiver representations 
(see \cite{hille}, \cite{hille:canada}, \cite{altmann-straten}, 
\cite{domokos-joo}). 
The investigation of the toric indeals of 
$3\times 3$ transportation polytopes in \cite{haase-paffenholz} was inspired by the so-called B\"ogvad conjecture.

\section{Preliminaries on graphs and  flow polytopes}\label{sec:prel-flow-polytopes} 

\subsection{Graphs.} \label{sec:graphs} 
In this paper a \emph{graph} $\Gamma$ means an undirected graph, having possibly multiple edges, but no loops. Denote by $\Gamma_0$ the set of vertices and $\Gamma_1$ the set of edges. 
Write $\chi_0(\Gamma)$ for the number of connected components of $\Gamma$, and set 
\[\chi(\Gamma)=|\Gamma_1|-|\Gamma_0|+\chi_0(\Gamma).\] 
For a vertex $v\in \Gamma_0$ we call the \emph{valency} of $v$ (denoted by 
$\mathrm{valency}_{\Gamma}(v)$) the number of edges in $\Gamma_1$ adjacent to $v$. 

Let us introduce the following ad hoc terminology: we call a graph $\Gamma$ \emph{prime} if it is connected, 
has at least one edge, 
and $\Gamma$ remains connected on removing any of its vertices and the adjacent edges. 
For a positive integer $d\ge 2$ denote by $\mathcal{L}_d$ the set of prime 
graphs $\Gamma$ such that $\chi(\Gamma)=d$ and each vertex of $\Gamma$ has valency at least $3$. 
Write $\mathcal{L}_1$ for the one-element set consisting of the graph with two vertices, connected by two edges. 

Define the \emph{chassis} $\mathcal{C}(\Gamma)$ of a  prime graph  $\Gamma$ with $\chi(\Gamma)>0$ 
as follows. First take the graph $\Gamma'$ obtained by contracting all the edges in $\Gamma$ that are not contained in a 
cycle of $\Gamma$ (contracting an edge means that we remove it and identify its end vertices). 
Note that all vertices of $\Gamma'$ have valency at least $2$. 
Set $V=\{v\in \Gamma'_0\mid \mathrm{valency}_{\Gamma'}(v)\ge 3\}$. Note that $|V|\ge 2$ or $V=\emptyset$.  If $|V|\ge 2$, the chassis  $\mathcal{C}(\Gamma)$ is obtained from $\Gamma'$ by 
replacing each path in $\Gamma'$ that connects distinct vertices $v_1,v_2\in V$   and does not meet any 
other vertex in $V$ by an edge in $\mathcal{C}(\Gamma)$ connecting $v_1$ and $v_2$. 
If $V=\emptyset$, then $\Gamma'$ is a cycle of length at least $2$; in this case $\mathcal{C}(\Gamma)$ is defined to be the cycle of length $2$. 
Note that the chassis of a prime graph $\Gamma$ with $\chi(\Gamma)>0$ belongs to $\mathcal{L}_{\chi(\Gamma)}$, 
and if $\Gamma$ belongs to $\bigcup_{d=1}^{\infty}\mathcal{L}_d$, then $\mathcal{C}(\Gamma)=\Gamma$. 

If $e$ is the only edge in a graph $\Gamma$ that connects the vertices $v$ and $v'$, then the \emph{contraction of 
 $\Gamma$ at $e$} is the graph $\Gamma'$ which is obtained from $\Gamma$ by removing $e$ and identifying $v$ and $v'$. 
We say that the graph $\Gamma'$ is a \emph{contracted descendant} of $\Gamma$, if $\Gamma'$ is obtained from $\Gamma$ by successively contracting edges. 
This gives a partial ordering on the set of isomorphism classes of graphs. 

A graph is \emph{$3$-regular} if all its vertices have valency $3$. For $d\ge 2$ denote by 
$\mathcal{L}_d^{3-\mathrm{reg}}$ the subset of $3$-regular graphs in $\mathcal{L}_d$. 
Note that for $\Gamma\in  \mathcal{L}_d^{3-\mathrm{reg}}$ we have 
$|\Gamma_0|=2(d-1)$ and $|\Gamma_1|=3(d-1)$. 

It is easy to see that for $d\ge 2$, for any $\Gamma\in \mathcal{L}_d$ there exists a 
$\Gamma'\in \mathcal{L}_d^{3-\mathrm{reg}}$ such that $\Gamma$ is a contracted descendant of $\Gamma'$. 
Indeed, if for some $v\in \Gamma_0$ we have $\mathrm{valency}_{\Gamma}(v)\ge 4$, then the set of edges adjacent to 
$v$ can be partitioned as the disjoint union $E_1\bigsqcup E_2$ with $|E_1|\ge 2$ and $|E_2|\ge 2$. 
Now replace the vertex $v$ by a new edge $e$ with end vertices $v_1$ and $v_2$, such that 
for an edge in $E_i$ the end vertex $v$ is replaced by $v_i$ ($i=1,2$). 
Clearly $\Gamma$ is recovered from this new graph by contracting the edge $e$. 
Iterating this step we arrive at the desired $3$-regular graph.

\subsection{Lattice polytopes.} 
A \emph{polytope} is the convex hull of a finite set of points in $\mathbb{R}^d$. It is a \emph{lattice polytope} if its vertices belong to $\mathbb{Z}^d$. 
Denote by $\mathrm{AffSpan}(\nabla)$ the affine subspace of $\mathbb{R}^d$ generated by 
$\nabla$. 
The lattice polytopes $\nabla\subset \mathbb{R}^d$ and $\nabla'\subset \mathbb{R}^{d'}$ are \emph{equivalent} 
(notation: $\nabla\cong\nabla'$) if there exists an affine linear isomorphism  $\varphi:\mathrm{AffSpan}(\nabla)\to\mathrm{AffSpan}(\nabla')$ of affine subspaces such that 
\begin{itemize}
\item[(i)] $\varphi$ maps $\mathrm{AffSpan}(\nabla)\cap \mathbb{Z}^d$ onto $\mathrm{AffSpan}(\nabla')\cap \mathbb{Z}^{d'}$; 
\item[(ii)] $\varphi$ maps  $\nabla$ \ onto $\nabla'$. 
\end{itemize} 

\subsection{Quivers and flow polytopes.} 
Following the literature on representation theory of finite dimensional algebras, by a \emph{quiver} $Q$ we mean a finite directed graph with vertex set $Q_0$, arrow set $Q_1$; 
for an arrow $a\in Q_1$ we denote by $a^-\in Q_0$ its tail and by $a^+\in Q_0$ its head. 
Multiple arrows are allowed, 
but we shall assume that there are no loops in $Q_1$, i.e. $a^-\neq a^+$ for all $a\in Q_1$. 
We shall denote by $\Gamma(Q)$ the underlying graph of $Q$, i.e. the graph obtained by forgetting the orientation of the arrows in $Q$, set 
$\chi(Q):=\chi(\Gamma(Q))$, and for $v\in Q_0$ we write 
$\mathrm{valency}_Q(v)=\mathrm{valency}_{\Gamma(Q)}(v)$.   
If $a\in Q_1$ is the only arrow from $a^-$ to $a^+$, the \emph{contraction of $Q$ at $a$} is the quiver $Q'$ 
obtained from $Q$ by removing the arrow $a$ and identifying its end vertices $a^-$ and $a^+$. 
Clearly, in this case $\Gamma(Q')$ is the contraction of $\Gamma(Q)$ at the edge corresponding to $a$. 

We shall restrict our attention to \emph{flow polytopes} $\nabla(Q,\theta,\underline{\ell},\underline{u})$ associated to a quiver $Q$, a \emph{weight} $\theta\in \mathbb{Z}^{Q_0}$,  and non-negative integer vectors $\underline{\ell},\underline{u}\in \mathbb{N}_0^{Q_1}$ as follows. First consider the affine subspace 
\[\mathcal{A}(Q,\theta)=\{x\in \mathbb{R}^{Q_1}\mid \forall v\in Q_0: 
\sum_{a^+=v}x(a)-\sum_{a^-=v}x(a)=\theta(v)\}\] 
in $\mathbb{R}^{Q_1}$. The polytope 
\[\nabla(Q,\theta,\underline{\ell},\underline{u})=\{x\in \mathcal{A}(Q,\theta)\mid \forall a\in Q_1: \underline{\ell}(a)\le 
x(a)\le \underline{u}(a)\}\] 
is a lattice polytope in $\mathbb{R}^{Q_1}$ by the generalized 
Birkhoff-von Neumann Theorem (cf. \cite[Theorem 13.11]{schrijver}). 
We shall also use the notation 
\[\nabla(Q,\theta)=\{x\in \mathcal{A}(Q,\theta)\mid \forall a\in Q_1: x(a)\ge 0\}.\] 
This is a polyhedron (the intersection of finitely many closed half-spaces in $\mathbb{R}^{Q_1}$). 
When the quiver $Q$ is acyclic, $\nabla(Q,\theta)=\nabla(Q,\theta,\underline{0},\underline{u})$ for $\underline{u}$ large enough, so it is a lattice polytope.  In fact it is well known that for 
any flow polytope $\nabla(Q,\theta,\underline{\ell},\underline{u})$ there exists an acyclic quiver $Q'$ and a weight $\theta'$ such that the lattice polytope $\nabla(Q,\theta,\underline{\ell},\underline{u})$ is equivalent to $\nabla(Q',\theta')$ 
(see for example \cite[Proposition 2.2]{domokos-joo} for a proof).  Note that 
$\dim(\nabla(Q,\theta,\underline{\ell},\underline{u}))\le \chi(Q)=\chi(Q')$. 

\noindent{\bf Convention.} Later in the text when we say that \emph{$\nabla$ is a flow polytope}, we shall mean that we specified $\nabla$ up to equivalence, but did not necessarily fix a concrete pair $(Q,\theta)$ such that 
$\nabla(Q,\theta)$ belongs to the equivalence class of $\nabla$. However, sometimes we shall use the notation 
$\nabla(Q,\theta)$ to denote a flow polytope defined up to equivalence only 
(see for example the statements of Proposition~\ref{prop:3-cells} and Proposition~\ref{prop:4-cells}). 

\subsection{Tightness.} \label{sec:tightness}
This notion was first introduced in \cite[Definition 12]{altmann-straten}.  
The definition given below is equivalent to 
\cite[Definition 4.1]{domokos-joo} by 
\cite[Proposition 4.2]{domokos-joo}. 
\begin{definition} 
The pair $(Q,\theta)$ of an acyclic quiver $Q$ and a weight $\theta\in \mathbb{Z}^{Q_0}$ is 
\emph{tight} if there is no arrow $a\in Q_1$ satisfying condition (R) or (C) below: 
\begin{itemize}
\item[({\bf{R)}}] $\nabla(Q,\theta)\subseteq \{x\in \mathbb{R}^{Q_1}\mid x(a)=0\}$; 
\item[({\bf{C}})] $\{x\in \mathcal{A}(Q,\theta)\mid x(a)\ge 0\}\supseteq 
\{x\in \mathcal{A}(Q,\theta)\mid \forall b\in Q_1\setminus \{a\}: x(b)\ge 0\}$. 
\end{itemize} 
When (R) holds for $a\in Q_1$, we say that \emph{$a$ is removable for $\theta$}. 
When (C) holds for $a\in Q_1$, we say that \emph{$a$ is contractible for $\theta$}. 
\end{definition}
This terminology and the role of tightness are explained by the following statement, 
summarizing \cite[Corollary 4.5]{domokos-joo},  
\cite[Proposition 4.2]{domokos-joo}, \cite[Lemma 4.4]{domokos-joo},  
 and \cite[Proposition 4.21]{domokos-joo}: 

\begin{proposition}\label{prop:tightness} 
\begin{itemize} 
\item[(i)] If $(Q,\theta)$ is tight, the assignment $a\mapsto \{x\in \nabla(Q,\theta)\mid x(a)=0\}$ 
gives a bijection between $Q_1$ and the set of facets (codinension $1$ faces) of the polytope $\nabla(Q,\theta)$.  . \item[(ii)] If $(Q,\theta)$ is tight, then $\dim(\nabla(Q,\theta))=\chi(Q)$. 
\item[(iii)] If $a\in Q_1$ is removable for $\theta$, then $\nabla(Q,\theta)$ is equivalent to $\nabla(Q',\theta)$, where $Q'$ is obtained by removing $a$ from $Q_1$. 
\item[(iv)] If $a\in Q_1$ is contractible for $\theta$, then $\nabla(Q,\theta)$ is equivalent to $\nabla(Q',\theta')$, where $Q'$ is obtained by contracting the arrow $a$, and $\theta'$ is obtained from $\theta$ by adding the values of $\theta$ on the end vertices  of $a$.  
\item[(v)] For any quiver $Q'$ and weight $\theta'\in\mathbb{Z}^{Q'_0}$, there exists a 
tight pair $(Q,\theta)$ such that the lattice polytopes $\nabla(Q',\theta')$ and 
$\nabla(Q,\theta)$ are equivalent; in fact $(Q,\theta)$ is obtained by successively removing or contracting arrows from $Q'$, until there is no contractible or removable arrow. 
\item[(vi)] If $Q$ is obtained from $Q'$ by contracting an arrow, then for any weight $\theta\in \mathbb{Z}^{Q_0}$ there exists a weight $\theta'\in \mathbb{Z}^{Q'_0}$ such that the lattice polytopes 
$\nabla(Q,\theta)$ and $\nabla(Q',\theta')$ are equivalent.  
\end{itemize}
\end{proposition}

\subsection{Products.} 
The \emph{product} of the polytopes $\nabla\subset \mathbb{R}^d$ and $\nabla'\subset \mathbb{R}^{d'}$ is the polytope $\nabla\times\nabla'\subset \mathbb{R}^d\times \mathbb{R}^{d'}=\mathbb{R}^{d+d'}$. If both $\nabla$ and $\nabla'$ are lattice polytopes, then 
$\nabla\times \nabla'$ is a lattice polytope as well. We call a positive dimensional lattice polytope \emph{prime} if it is not equivalent to a product of strictly smaller dimensional lattice polytopes.  Obviously, any positive dimensional lattice polytope is equivalent to a product of finitely many prime lattice polytopes. 

\begin{proposition}\label{prop:prime} Let $Q$ be an acyclic quiver with $\chi(Q)>0$, and let $(Q,\theta)$ be a tight pair. Then the following conditions are equivalent: 
\begin{itemize}
\item[(i)] The lattice polytope $\nabla(Q,\theta)$ is prime. 
\item[(ii)] $\nabla(Q,\theta)$ is not equivalent to a product of strictly smaller dimensional flow polytopes. 
\item[(iii)] The graph $\Gamma(Q)$ is prime. 
\end{itemize} 
\end{proposition} 

\begin{proof} The implication (i)$\Rightarrow$(ii) is trivial. The implication (ii)$\Rightarrow$(iii) is \cite[Proposition 4.10 (i)]{domokos-joo}. 
Assume now that (iii) holds. By \cite[Theorem 4.12 (ii)]{domokos-joo}, the projective toric variety 
corresponding to the polytope $\nabla(Q,\theta)$ is not isomorphic to a product of strictly smaller dimensional projective toric varieties. This implies that $\nabla(Q,\theta)$ is prime, 
because the projective toric variety associated to the product $\nabla\times \nabla'$ of polytopes is isomorphic  to the product of the projective toric varieties associated to $\nabla$ and $\nabla'$ 
(see e.g. \cite{cox-little-schenck} for the correspondence between lattice polytopes and projective toric varieties). 
\end{proof} 

\subsection{A quiver construction.} 
The following construction of quivers plays essential role in our studies. 
Given a graph $\Gamma$, we denote by $\Gamma^*$ the quiver obtained 
by \emph{putting a sink on each edge} of $\Gamma$. In more details,  
$\Gamma^*_0=\Gamma_0\bigcup \{v_e\mid e\in \Gamma_1\}$ (so we introduce a new vertex for each edge in $\Gamma$),  and  
replace each edge $e$ of $\Gamma$ (denote by $e^-$, $e^+\in \Gamma_0$ its end vertices)   by a new vertex $v_e$ and two arrows $a_1,a_2$ with $a_1^+=v_e=a_2^+$, 
$a_1^-=e^-$, and $a_2^-=e^+$, as illustrated on the figure below. 
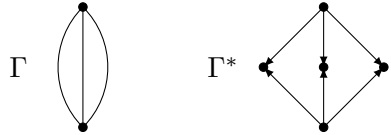
\[
\begin{tikzpicture}[>=latex,scale=0.8] 
\foreach \x in {(0,0),(0,2)} \filldraw \x circle (2pt);
\draw  [-]  (0,0)--(0,2);
\draw  [-]  (0,0) to [out=45,in=-45] (0,2);
\draw  [-]  (0,0) to [out=135,in=-135] (0,2);
\node[left] at (-0.75,1) {$\Gamma$};
\foreach \x in {(4,0),(4,1),(3,1),(5,1),(4,2)} \filldraw \x circle (2pt);
\node[left] at (2.75,1) {$\Gamma^*$};
\draw  [->]  (4,0)--(4,1);
\draw  [->]  (4,2)--(4,1);
\draw  [->]  (4,0)--(3,1);
\draw  [->]  (4,2)--(3,1);
\draw  [->]  (4,0)--(5,1);
\draw  [->]  (4,2)--(5,1);
\end{tikzpicture}
\]
 Note that $\chi(\Gamma^*)=\chi(\Gamma)$, 
 and $\mathcal{C}(\Gamma)=\mathcal{C}(\Gamma(\Gamma^*))$.

\begin{lemma}\label{lemma:reduction} 
Suppose that $Q$ has a vertex $v$ which is a valency $2$ sink, $a\neq b$ are the arrows with $a^+=v=b^+$, $a^-$ is a source in $Q$ (i.e. there are now arrows pointing to $a^-$), 
$a^-\neq b^-$, and $\theta$ is a weight with 
$\theta(a^-)=-1$, $\theta(v)=1$. 
Then the arrow $b$ is contractible for $(Q,\theta)$, and hence $\nabla(Q,\theta)$ is equivalent to 
$\nabla(Q',\theta')$ where $(Q',\theta')$ is obtained from $(Q,\theta)$ after contracting $b$: 
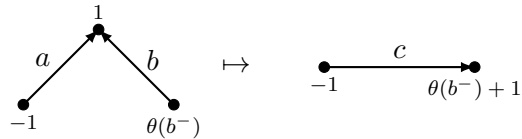
\[\begin{tikzpicture}[>=latex] 
\node [right] at (2.5,0.5) {$\mapsto$}; 
\node [below] at (0,0) {\scriptsize{$-1$}}; \node [below] at (2,0) {\scriptsize{$\theta(b^-)$}}; 
\foreach \x in {(0,0),(1,1),(2,0),(4,0.5),(6,0.5)} \filldraw \x circle (2pt); 
\node [above] at (1,1) {\scriptsize{$1$}};  \node [above, left] at (0.5,0.6) {$a$}; \node[above,right] at (1.5,0.6) {$b$}; 
\draw [->, thick] (0,0)--(1,1); \draw [<-, thick] (1,1)--(2,0);  \draw [->, thick] (4,0.5)--(6,0.5);
\node at (5,0.7) {$c$};
\node [below] at (4,0.5) {\scriptsize{$-1$}}; \node [below] at (6,0.5) {\scriptsize{$\theta(b^-)+1$}};
\end{tikzpicture} \]
\end{lemma} 

\begin{proof} 
Since $a^-$ is a source vertex in $Q$, $\theta(a^-)=-1$ implies that 
$x(a)\le 1$ for all $x\in H:= \mathcal{A}(Q,\theta)\cap \{y\in \mathbb{R}^{Q_1}\mid  y(d)\ge 0 \mbox{ if }d^-=a^-\}$. 
Note that $x(b)=\theta(v)-x(a)=1-x(a)$ for any $x\in \mathcal{A}(Q,\theta)$. Therefore for $x\in H$ 
we have $x(b)\ge 0$. 
It follows that the arrow $b$ is contractible, and  so 
$\nabla(Q,\theta)$ is equivalent to $\nabla(Q',\theta')$ by Proposition~\ref{prop:tightness} (iv).  
\end{proof} 

\subsection{First steps towards a classification.}  
We summarize below consequences of some basic  reduction steps from \cite{domokos-joo} 
(similar reductions in the special case of the canonical weight were obtained before in  \cite{altmann-nill-schwentner-wiercinska}) 
that form the basis of the classification of 
flow polytopes up to equivalence. 

\begin{proposition} \label{prop:valency 2}
Let $Q$ be a connected acyclic quiver with $Q_1\neq\emptyset$, and let $\theta\in \mathbb{Z}^{Q_0}$ be a weight. 
\begin{itemize}
\item[(i)] If there is no contractible arrow in $Q_1$ for $\theta$,
then each vertex of $Q$ has valency at least $2$. 
Moreover,  each valency $2$ vertex is a sink or a source, and no valency $2$ vertices in $Q$ are 
connected by an arrow, 
unless $Q$ is the Kronecker quiver 
(consisting of two vertices and two arrows from one of these vertices to the other). 
\item[(ii)] Suppose that $v$ is a valency $2$ source in $Q$, 
$a,b\in Q_1$ with $a^-=v=b^-$, and assume that $a^+\neq b^+$. 
Denote by $Q'$ the quiver obtained by reversing 
the two arrows $a,b$. Then for any weight $\theta\in \mathbb{Z}^{Q_0}$ we have that 
$\nabla(Q,\theta)\cong \nabla(Q',\theta')$, where $\theta'(v)=-\theta(v)$, $\theta'(a^-)=\theta(a^-)+\theta(v)$, 
$\theta'(b^-)=\theta(b^-)+\theta(v)$, and $\theta'$ agrees with $\theta$ on the remaining vertices. 
\end{itemize} 
\end{proposition} 
\begin{proof} 
Obviously, if $\mathrm{valency}_Q(a^-)=1$ or $\mathrm{valency}_Q(a^+)=1$, then $a$ is contractible for any weight $\theta$, so  
(i) is an immediate  consequence of \cite[Corollary 4.7]{domokos-joo}, whereas (ii) is \cite[Proposition 4.8]{domokos-joo}. 
\end{proof} 

\begin{proposition}\label{prop:Gamma^*}
Let $Q$ be an acyclic quiver with $\Gamma(Q)$ prime and $\chi(Q)\ge 1$.  
 If $\Gamma$ is any graph such that $\mathcal{C}(\Gamma(Q))$ is a contracted descendant of $\Gamma$, then for any weight 
 $\theta\in \mathbb{Z}^{Q_0}$ there is a weight $\sigma\in\mathbb{Z}^{\Gamma^*_0}$ such that 
 $\nabla(Q,\theta)\cong \nabla(\Gamma^*,\sigma)$. 
\end{proposition} 

\begin{proof} 
Assume first that $\chi(Q)\ge 2$. 
By Proposition~\ref{prop:valency 2} (i), (ii) contracting successively contractible arrows adjacent to valency $\le 2$ vertices and reversing arrows at valency $2$ sources we can get from $(Q,\theta)$ to $(Q',\theta')$ such that 
$\nabla(Q,\theta)\cong \nabla(Q',\theta')$, $\mathcal{C}(\Gamma(Q'))= \mathcal{C}(\Gamma(Q))$, all vertices of $Q'$ have valency at least $2$, no valency $2$ vertices are connected by an arrow, and all valency $2$ vertices of $Q'$ are sinks. In particular, $Q'$ is obtained from $\mathcal{C}(\Gamma(Q'))^*$ by contracting arrows. 
Recall that 
$\mathcal{C}(\Gamma(Q'))= \mathcal{C}(\Gamma(Q))$ is a contracted descendant 
of $\Gamma$ by assumption. Therefore $Q'$ is obtained from $\Gamma^*$ by successively contracting arrows. 
Thus by Proposition~\ref{prop:tightness} (vi), there exists a weight $\sigma\in \mathbb{Z}^{(\Gamma^*)_0}$ with 
$\nabla(Q',\theta')\cong \nabla(\Gamma^*,\sigma)$. 

The case $\chi(Q)=1$ is similar. 
Then $\mathcal{C}(\Gamma(Q))$ is a 
$2$-cycle, and by Proposition~\ref{prop:valency 2} (i),
$\nabla(Q,\theta)\cong \nabla(Q',\theta')$ where $Q'$ is the Kronecker quiver. 
\end{proof}

\begin{proposition}\label{prop:3-regular} 
Let $\nabla$ be a $d$-dimensional prime flow polytope where $d\ge 1$. 
\begin{itemize}
\item[(i)] There exists a graph $\Gamma\in \mathcal{L}_d$, a quiver $Q$ obtained by putting a sink on \emph{some} of the edges of $\Gamma$ and orienting somehow the remaining edges (in particular, $\mathcal{C}(\Gamma(Q))=\Gamma$), and a weight $\theta\in \mathbb{Z}^{Q_0}$ such that the pair $(Q,\theta)$ is tight, and $\nabla\cong \nabla(Q,\theta)$. 
 \item[(ii)]  If $d\ge 2$, then there exists a graph $\Gamma\in\mathcal{L}_d^{\mathrm{3-reg}}$ and a weight 
$\sigma\in \mathbb{Z}^{\Gamma^*_0}$ with $\nabla\cong \nabla(\Gamma^*,\sigma)$. 
\end{itemize}
\end{proposition}

\begin{proof} (i) follows from Proposition~\ref{prop:tightness} (v), Proposition~\ref{prop:prime}, and Proposition~\ref{prop:valency 2}, whereas (ii) follows by Proposition~\ref{prop:Gamma^*}  
from (i) and the observation that each graph in $\mathcal{L}_d$ is a contracted descendant of some graph in 
$\mathcal{L}_d^{\mathrm{3-reg}}$. 
\end{proof} 

\begin{remark} The pair $(\Gamma^*,\sigma)$ in Proposition~\ref{prop:3-regular} (ii) may not be tight, unlike the 
pair $(Q,\theta)$ in Proposition~\ref{prop:3-regular} (i). 
\end{remark} 


\section{A regular hyperplane subdivision of flow polytopes} \label{sec:regular subdivision} 

Let $Q$ be a quiver, and $\theta\in\mathbb{Z}^{Q_0}$ a weight  such that the flow polytope 
$\nabla(Q,\theta)$ is positive dimensional; 
set $d=\dim(\nabla(Q,\theta))$. 
We start with the obvious 
equality 
\[\nabla(Q,\theta)=\bigcup_{\underline{k}\in \mathbb{N}_0^{Q_1}}
\nabla(Q,\theta,\underline{k},\underline{k}+\underline{1})\]
where $\underline{1}=(1,1,\dots,1)\in\mathbb{N}_0^{Q_1}$. 
If $\nabla(Q,\theta,\underline{k},\underline{k}+\underline{1})$ in non-empty, 
then there exists a lattice point $x\in \nabla(Q,\theta)$ such that $k(a)\le x(a)$ for all 
$a\in Q_1$. As there are only finitely many lattice points in $\nabla(Q,\theta)$, there are only finitely many non-negative integral vectors $\underline{k}$ with 
$\nabla(Q,\theta,\underline{k},\underline{k}+\underline{1})$ non-empty. 
Denote by $\underline{k}^{(1)},\dots,\underline{k}^{(q)}$ those for which 
$\dim(\nabla(Q,\theta,\underline{k},\underline{k}+\underline{1}))=d$. 
We obtained the subdivision 
\begin{equation} \label{eq:subdivision} \nabla(Q,\theta)=\bigcup_{i=1}^q\nabla(Q,\theta,\underline{k}^{(i)},\underline{k}^{(i)}+\underline{1})
\end{equation} 
of $\nabla(Q,\theta)$ into finitely many $d$-dimensional flow polytopes. To get this subdivision we sliced the polytope 
$\nabla(Q,\theta)$ along a finite set of hyperplanes (namely hyperplanes of the form 
$\{x\in\mathcal{A}(Q,\theta)\mid x(a)=k\}$, where $k\in \mathbb{N}_0$). 
Therefore the  subdivision \eqref{eq:subdivision} is regular, see for example \cite[Section 1.F, p. 35]{bruns-gubeladze}. 

\begin{definition} The flow polytopes  $\nabla(Q,\theta,\underline{k}^{(i)},\underline{k}^{(i)}+\underline{1})$ 
($i=1,\dots,q$) appearing in \eqref{eq:subdivision} are called the \emph{cells of} $(Q,\theta)$. 
They are considered up to equivalence. 
\end{definition}

\begin{remark} 
Although the cells of $(Q,\theta)$ are polytopes considered up to equivalence, they depend on the quiver $Q$ and the weight $\theta$, and not just on the equivalence class of the polytope $\nabla(Q,\theta)$. In other words, up to equivalence a flow polytope can be represented by many 
(quiver,weight) pairs, and  different representations may lead to essentially different cell decompositions. 
We shall make use of this flexibility in proofs later. 
\end{remark} 

\begin{definition}\label{def:P(Gamma)} 
For $\Gamma\in \mathcal{L}_d$  (note that $d=\chi(\Gamma)$) set 
\[\mathcal{P}(\Gamma):=\{\nabla\mid \nabla \text{ is a cell of }(\Gamma^*,\theta) \text{ for some }
\theta\text{ with }\dim(\nabla(\Gamma^*,\theta))=d\}.\] 
\end{definition}  

Our interest in $\mathcal{P}(\Gamma)$ stems from Lemma~\ref{lemma:regular subdivision} below. 

\begin{lemma}\label{lemma:regular subdivision} 
Let $\nabla=\nabla(Q,\theta)$ be a $d$-dimensional prime flow polytope where $\Gamma(Q)$ is prime and $d=\chi(Q)$. 
If $\Gamma\in \mathcal{L}_d$ is a graph such that $\mathcal{C}(\Gamma(Q))$ is a contracted descendant of 
$\Gamma$, then  $\nabla$ has 
a regular hyperplane subdivision 
$\nabla=\nabla_1\cup\dots\cup\nabla_q$  
where each $\nabla_i$ belongs to $\mathcal{P}(\Gamma)$. 
\end{lemma} 

\begin{proof}
By Proposition~\ref{prop:Gamma^*} there exists a weight $\sigma\in \mathbb{Z}^{\Gamma^*_0}$ with $\nabla(\Gamma^*,\sigma)\cong\nabla$, and the subdivison  \eqref{eq:subdivision} of 
$\nabla(\Gamma^*,\sigma)$ gives the claim. 
\end{proof}

\begin{corollary}\label{cor:3-regular} 
For any prime flow polytope $\nabla$ of dimension $d\ge 2$ there exists a graph $\Gamma\in \mathcal{L}_d^{\mathrm{3-reg}}$ such that 
$\nabla$ has 
a regular hyperplane subdivision 
$\nabla=\nabla_1\cup\dots\cup\nabla_q$  
where each $\nabla_i$ belongs to $\mathcal{P}(\Gamma)$. 
\end{corollary} 

\begin{proof}
By Proposition~\ref{prop:3-regular} (ii) there exists a $3$-regular graph $\Gamma$ and a weight for $\Gamma^*$ such that $\nabla\cong \nabla(\Gamma^*,\sigma)$. Now take the subdivision \eqref{eq:subdivision} of 
$\nabla(\Gamma^*,\sigma)$. 
\end{proof}

\begin{definition}\label{def:W(Gamma)}
For $\Gamma\in\mathcal{L}_d$ 
we define a set $\mathcal{W}(\Gamma)$ of weights for $\Gamma^*$ as follows. 
Recall that 
$\Gamma^*_0=\Gamma_0\bigsqcup \{v_e\mid e\in \Gamma_1\}$. 
Now 
\begin{align*}\mathcal{W}(\Gamma)=\{\theta\in \mathbb{Z}^{\Gamma^*_0}\mid &
\theta(v_e)=1\ \forall e\in \Gamma_1, \\  
& 0>\theta(v)>-\mathrm{valency}_{\Gamma}(v) \ \forall v\in \Gamma_0, \\ & \sum_{v\in \Gamma_0}\theta(v)=-|\Gamma_1|\}.
\end{align*}
In particular, if $\Gamma$ is $3$-regular, 
then $\theta(v)=-1$ for $d-1$ of the source vertices $v\in \Gamma_0\subset \Gamma^*_0$,  
and $\theta(v)=-2$ for the remaining  $d-1$ of the source vertices $v\in \Gamma_0\subset \Gamma^*_0$. 
\end{definition} 

\begin{proposition}\label{prop:P(Gamma)} 
For $\Gamma\in\mathcal{L}_d$ we have 
\[\mathcal{P}(\Gamma)=\{\nabla(\Gamma^*,\theta)\mid \theta\in \mathcal{W}(\Gamma),\ \dim(\nabla(\Gamma^*,\theta))=d\}.\] 
\end{proposition} 

\begin{proof}  
Take any cell $\nabla(\Gamma^*,\theta,{\underline{k}},{\underline{k}}+{\underline{1}})$ of $(\Gamma^*,\theta)$ where  $\dim(\nabla(\Gamma^*,\theta))=\chi(\Gamma)$. 
It is obviously equivalent to 
$\nabla(\Gamma^*,\sigma,{\underline{0}},{\underline{1}})$, where 
$\sigma(v)=\theta(v)-\sum_{a^+=v}{\underline{k}}(a)+\sum_{a^-=v}{\underline{k}}(a)$ for $v\in \Gamma^*_0$. 
For $x\in\nabla(\Gamma^*,\sigma,{\underline{0}},{\underline{1}})$ and $e\in\Gamma_1$, denoting by $a_1,a_2\in \Gamma^*_1$ the arrows pointing to $v_e\in\Gamma^*_0$, we have $x(a_1)+x(a_2)=\sigma(v_e)$. As $0\le x(a_1),x(a_2)\le 1$, we have 
$\sigma(v_e)\in\{0,1,2\}$. Now $\sigma(v_e)=0$ would imply $x(a_1)=0=x(a_2)$ for all $x\in \nabla(\Gamma^*,\sigma,{\underline{0}},{\underline{1}})$. 
This would mean that $a_1,a_2$ are removable for the pair $(\Gamma^*,\sigma)$, and so 
$\nabla(\Gamma^*,\sigma,{\underline{0}},{\underline{1}})\cong \nabla((\Gamma')^*,\sigma',{\underline{0}},{\underline{1}})$, where $\Gamma'$ is the graph obtained by removing  the edge $e$ of $\Gamma$ (i.e. $(\Gamma')^*$ is the quiver obtained by removing $a_1,a_2$ from $\Gamma^*$). 
Since $\Gamma\in \mathcal{L}_d$, removing an edge from $\Gamma$ we get a connected graph, 
hence $\chi(\Gamma')<\chi(\Gamma)$. 
This leads to the contradiction 
\[d=\dim(\nabla(\Gamma^*,\sigma,{\underline{0}},{\underline{1}}))=\dim(\nabla((\Gamma')^*,\sigma',{\underline{0}},{\underline{1}})\le\chi(\Gamma')<\chi(\Gamma)=d.\] 
Similarly, $\sigma(v_e)=2$ implies $x(a_1)=1=x(a_2)$ for all $x\in \nabla(\Gamma^*,\sigma,{\underline{0}},{\underline{1}})$, hence 
$\nabla(\Gamma^*,\sigma,{\underline{0}},{\underline{1}})\cong\nabla(\Gamma^*,\sigma',{\underline{0}},{\underline{1}})$, where $\sigma'(v_e)=0$, 
$\sigma'(a_1^-)=\sigma(a_1^-)+1$, $\sigma'(a_2^-)=\sigma(a_2^-)+1$, and for all 
$v\in \Gamma^*_0\setminus\{v_e,a_1^-,a_2^-\}$ we have $\sigma'(v)=\sigma(v)$. 
Consequently, $x(a_1)=0=x(a_2)$ for all $x\in \nabla(\Gamma^*,\sigma',{\underline{0}},{\underline{1}})$, 
leading to the contradiction $d=\dim(\nabla(\Gamma^*,\sigma',{\underline{0}},{\underline{1}}))<\chi(\Gamma)=d$ as above. 
Thus we showed that $\sigma(v_e)=1$ for all $e\in \Gamma_1$. 
This implies in particular that 
\[\nabla(\Gamma^*,\sigma,{\underline{0}},{\underline{1}})=\nabla(\Gamma^*,\sigma).\]

Take next any vertex $v\in \Gamma_0\subset\Gamma^*_0$, and set $\ell:=\mathrm{valency}_{\Gamma}(v)$. 
There are $\ell$ arrows $a_1,\dots,a_{\ell}$ 
starting at $v$, and there is no arrow ending at $v$. So $\sigma(v)=-(x(a_1)+\cdots+x(a_{\ell}))$ for any 
$x\in \nabla(\Gamma^*,\sigma)$, implying by $0\le x(a_1),\dots,x(a_{\ell})\le 1$ that $\theta(v)\in \{0,-1,\dots,-\ell\}$.
As above, $\theta(v)=0$ would imply $x(a_1)=\cdots=x(a_{\ell})=0$ for all $x\in \nabla(\Gamma^*,\sigma)$. 
In particular, $a_1\in \Gamma^*_1$ is removable for $\sigma$; 
denote by $Q$ the quiver obtained by removing $a_1$ from $\Gamma^*$. Note that removing an edge from $\Gamma\in \mathcal{L}_d$ the graph remains connected, hence 
$\chi(Q)<\chi(\Gamma^*)=\chi(\Gamma)$. 
This leads to the contradiction $d=\dim(\nabla(\Gamma^*,\sigma))=\dim(\nabla(Q,\sigma))\le \chi(Q)<\chi(\Gamma)=d$. Similarly, 
$\sigma(v)=-\ell$ would imply $x(a_1)=\cdots=x(a_{\ell})=1$ for all $x\in \nabla(\Gamma^*,\sigma)$, leading to the contradiction 
$d=\dim(\nabla(\Gamma^*,\sigma))<d$.  It follows that $\sigma(v)\in \{-1,-2,\dots,-\ell+1\}$. 
Moreover,  $\nabla(\Gamma^*,\sigma)\neq\emptyset$ implies that 
\[0=\sum_{v\in \Gamma^*_0}\sigma(v)=\sum_{v\in \Gamma_0\subset\Gamma^*_0}\sigma(v)+\sum_{e\in \Gamma_1} \sigma(v_e)=\sum_{v\in \Gamma_0\subset\Gamma^*_0}\sigma(v)+|\Gamma_1|.\] 
Thus $\sigma\in \mathcal{W}(\Gamma)$. 

Conversely, $\theta\in \mathcal{W}(\Gamma)$ implies that  
$\nabla(\Gamma^*,\theta)=\nabla(\Gamma^*,\theta,{\underline{0}},{\underline{1}})$, so 
$\nabla(\Gamma^*,\theta)$ is a cell of $(\Gamma^*,\theta)$. 
\end{proof}


\section{Classification of $3$-dimensional compressed flow polytopes} \label{sec:class-compressed-3-dim} 

A lattice polytope is \emph{compressed} if it has width $1$ with respect to any of its facets. 
This means that  for each facet, the polytope is contained in the region between the hyperplane spanned by the given facet and one of the two neighbouring parallel lattice hyperplanes. Equivalent characterizations and information on the significance of this notion can be found in  \cite{haase-etal}, \cite[Theorem 2.4]{sullivant}, \cite[Theorem 2.3]{ohsugi-hibi}, 
\cite[Chapter 9]{lorea-rambau-santos}.  The relevance of this notion for the present study is that 
for any quiver $Q$ and weight $\theta\in \mathbb{Z}^{Q_0}$, the  cells of $(Q,\theta)$ are compressed. 
In this section we describe all compressed flow polytopes up to dimension $3$.

It is obvious that a product $\nabla\times \nabla'$ of lattice polytopes is compressed if and only if both $\nabla$ and $\nabla'$ are compressed. 
Therefore to classify compressed flow polytopes, it is sufficient to deal with the prime compressed flow polytopes. 

Clearly, up to equivalence, the closed interval $[0,1]\subset \mathbb{R}$ is the only $1$-dimensional compressed polytope, and it can be realized as the flow polytope of the 
(quiver,weight) pair 
 \begin{tikzpicture}[>=latex']
\foreach \x in {(0,0),(1,0)} \filldraw \x circle (2pt);   
\draw  [->, thick]  (0,0) to [out=45,in=135] (1,0);
\draw  [->, thick]  (0,0) to [out=-45,in=-135] (1,0);
\node [left] at (0,0) {$-1$};
\node [right] at (1,0) {$1$};
\end{tikzpicture}. 
It is also easy to see that up to equivalence, the only prime $2$-dimensional compressed polytope is the $2$-simplex in 
$\mathbb{R}^2$ with vertices $(0,0)$, $(1,0)$, $(0,1)$, and it can be realized as the 
flow polytope of the (quiver,weight) pair 
 \begin{tikzpicture}[>=latex']
\foreach \x in {(0,0),(1,0)} \filldraw \x circle (2pt);   
\draw  [->, thick]  (0,0) to [out=45,in=135] (1,0);
\draw  [->, thick]  (0,0) to  (1,0);
\draw  [->, thick]  (0,0) to [out=-45,in=-135] (1,0);
\node [left] at (0,0) {$-1$};
\node [right] at (1,0) {$1$};
\end{tikzpicture}. 
 
 \begin{proposition}\label{prop:all compressed} 
Let $d$ be a positive integer.  For any compressed $d$-dimensional prime flow polytope $\nabla$ there exists a 
graph $\Gamma\in \mathcal{L}_d$ such that $\nabla\in\mathcal{P}(\Gamma)$. 
\end{proposition} 

\begin{proof} Let $\nabla$ be a $d$-dimensional compressed prime flow polytope. 
By Proposition~\ref{prop:3-regular} (i) there is a tight pair $(Q,\sigma)$ with $\nabla(Q,\sigma)\cong \nabla$, 
where $Q$ is obtained by putting a sink on some of the edges of 
$\Gamma:=\mathcal{C}(\Gamma(Q))\in \mathcal{L}_d$ and orienting somehow the remaining edges. 
So  $\chi(\Gamma)=\chi(Q)=d$. 
Tightness of $(Q,\sigma)$ implies that the facets of $\nabla(Q,\sigma)$ are 
$\{x\in \nabla(Q,\sigma)\mid x(a)=0\}$, where $a\in Q_1$. Since $\nabla$ is compressed, we conclude that for all $a\in Q_1$ we have that 
$\{x(a)\mid x\in \nabla(Q,\sigma)\}$ is the closed interval $[0,1]$.  
Indeed, by tightness of $(Q,\sigma)$ we can take any $x$ belonging to the facet $F$ corresponding to $a$, and then $x(a)=0$. Again by tightness, $a\in Q_1$ is not removable for $\sigma$, hence there exists a vertex $y$ of $\nabla(Q,\sigma)$ with $y(a)\neq 0$. Since $\nabla$ is compressed, it lies 
between the facet $F$ and the lattice hyperplane 
$\{x\in \mathcal{A}(Q,\sigma)\mid x(a)=1\}$. 
Thus $y(a)=1$. 
In particular, all coordinates of any lattice point of $\nabla(Q,\sigma)$ belong to $\{0,1\}$, and 
for any $a\in Q_1$ there exist $x,y\in \nabla(Q,\sigma)$ such that $x(a)=0$ and $y(a)=1$. 
Consider a valency $2$ sink $v$ in $Q_0$, and denote by $a,b\in Q_1$ the arrows with 
$a^+=v=b^+$. For any $x\in \nabla(Q,\sigma)$ we have $x(a)+x(b)=\sigma(v)$. 
Applying this for a lattice point $x$ with $x(a)=0$, we get that $\sigma(v)=x(b)\in \{0,1\}$. 
On the other hand, the existence of a lattice point $x$ with $x(a)=1$ implies that 
$\sigma(v)\ge1$. Thus $\sigma(v)=1$. 

Now let $a\in Q_1$ be an arrow with $\mathrm{valency}_Q(a^-)>2$ and $\mathrm{valency}_Q(a^+)>2$. Denote by $(Q',\sigma')$ the pair obtained  by putting a sink on $a$ 
(i.e. we replace $a$ by two arrows $a_1$ and $a_2$, where $a_1^-=a^-$, 
$a_2^-=a^+$, and $a_1^+=a_2^+$ is a new vertex in $Q'$), $\sigma'(v)=1$ for the new sink $v$, 
$\sigma'(a^+)=\sigma(a^+)-1$
and $\sigma'(w)=\sigma(w)$ for all other verices. The map $\mathcal{A}(Q,\sigma)\to 
\mathcal{A}(Q',\sigma')$ 
sending $x$ to $y$, where $y(a_1)=x(a)$, $y(a_2)=1-x(a)$, and $y(b)=x(b)$ for all 
$b\in Q'_1\setminus \{a_1,a_2\}=Q_1\setminus \{a\}$ is easily seen to give an equivalence between the lattice polytopes $\nabla(Q,\sigma)$ and $\nabla(Q',\sigma')$. 

A successive application of the above step leads us to the pair $(\Gamma^*,\theta)$ with $\nabla(\Gamma^*,\theta)\cong\nabla$, 
where $\Gamma=\mathcal{C}(\Gamma(Q))$, and 
$\theta(v)=1$ for each valency $2$ sink $v$. 
Moreover, although the pair $(\Gamma^*,\theta)$ may not be tight, it still has the property 
that for any lattice point $x$ and any $a\in Q_1$, we have $x(a)\in \{0,1\}$, 
and for any $a\in Q_1$ there exist $x,y\in \nabla(Q,\theta)$ such that $x(a)=0$ and $y(a)=1$. 
It follows that $\nabla(\Gamma^*,\theta)=\nabla(\Gamma^*,\theta,\underline{0},\underline{1})$, so 
$\nabla$ is a cell of $(\Gamma^*,\theta)$. That is, $\nabla\in \mathcal{P}(\Gamma)$. 
\end{proof}

 \begin{proposition} \label{prop:all 3-dim compressed} 
Up to equivalence there are three prime compressed $3$-dimensional flow polytopes: 
$\nabla_{\mathrm{I.a}}^{(3)}$, $\nabla_{\mathrm{II.c}}^{(3)}$, $\nabla_{\mathrm{I.b}}^{(3)}$. 
For each such $\nabla$  a tight pair $(Q,\theta)$ with $\nabla\cong \nabla(Q,\theta)$ is given below, together with  a system of linear inequalities defining $\nabla$ as a subset of $\mathbb{R}^3$ 
 (namely $L\ge 0$, where $L$ ranges over the degree $1$ polynomials in $x,y,z$ attached to the arrows), and with 
 a list of the lattice points $v_0,v_1,\dots,v_m$ of $\nabla\subset \mathbb{R}^3$. 
 
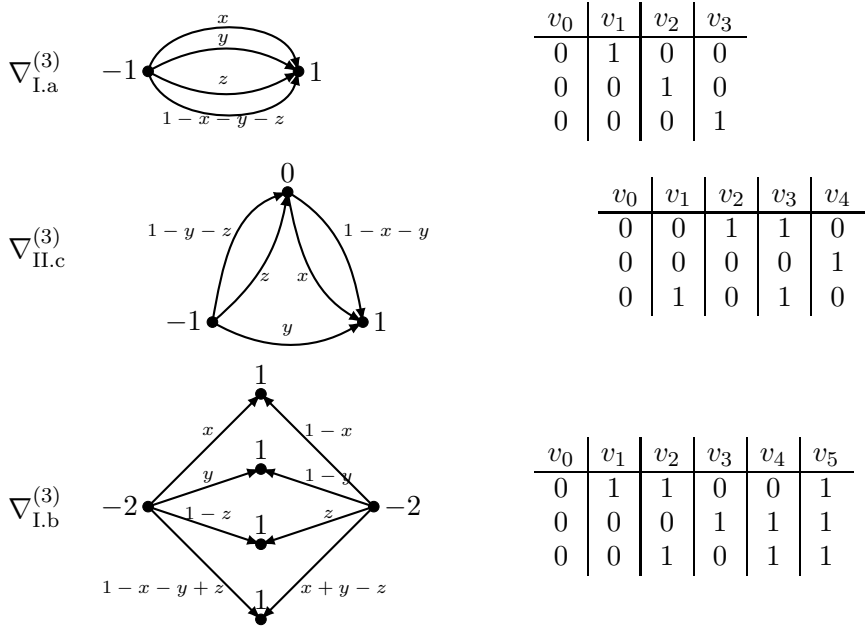
\begin{tikzpicture}[>=latex]
\foreach \x in {(0,0),(2,0)} \filldraw \x circle (2pt); 
\node [left] at (-1,0) {$\nabla_{\mathrm{I.a}}^{(3)}$};  
\draw  [->, thick]  (0,0) to [out=30,in=150] (2,0); 
\node [above] at (1,0.5) {\text{\tiny$x$}}; 
\node [above] at (1,0.2) {\text{\tiny$y$}};
\node [above] at (1,-0.3) {\text{\tiny$z$}}; 
\node [below] at (1,-0.4) {\text{\tiny$1-x-y-z$}};
\draw  [->, thick]  (0,0) to [out=75,in=105] (2,0);
\draw  [->, thick]  (0,0) to [out=-30,in=-150] (2,0);
\draw  [->, thick]  (0,0) to [out=-75,in=-105] (2,0);
\node [left] at (0,0) {$-1$};
\node [right] at (2,0) {$1$};
\node [right] at (5,0) {$\begin{array}{c|c|c|c}
v_0 & v_1 & v_2 & v_3 \\
\hline 
0&1&0&0\\ 0&0&1&0\\ 0&0&0&1
\end{array}$}; 
\end{tikzpicture}
 
\begin{tikzpicture}[>=latex]
\foreach \x in {(0,0),(1,1.73),(2,0)} \filldraw \x circle (2pt); 
\node [left] at (-1.8,1) {$\nabla_{\mathrm{II.c}}^{(3)}$};  
\draw [->, thick] (0,0) to [out=-30,in =-150] (2,0); 
\draw [->, thick] (0,0) to [out=80,in =-160] (1,1.73);
\draw [->, thick] (0,0) to [out=40,in=-100] (1,1.73); 
\draw [->, thick] (1,1.73) to [out=-80,in=150] (2,0); 
\draw [->, thick] (1,1.73) to [out=-30,in=100] (2,0); 
\node [above] at (1,1.73) {$0$}; 
\node [left] at (0,0) {$-1$}; 
\node [right] at (2,0) {$1$};  
\node at (1,-0.1) {\text{\tiny$y$}};
\node at (0.7,0.6) {\text{\tiny$z$}};
\node at (1.2,0.6) {\text{\tiny$x$}};
\node [left] at (0.4,1.2) {\text{\tiny$1-y-z$}};
\node [right] at (1.6,1.2) {\text{\tiny$1-x-y$}};
\node [right] at (5,1) {$\begin{array}{c|c|c|c|c}
v_0 & v_1 & v_2 & v_3 & v_4 \\
\hline 
0&0&1&1&0\\ 0&0&0&0&1\\ 0&1&0&1&0
\end{array}$}; 
\end{tikzpicture}

\begin{tikzpicture}[>=latex]
\foreach \x in {(0,0),(3,0),(1.5,1.5),(1.5,0.5),(1.5,-0.5),(1.5,-1.5)} \filldraw \x circle (2pt); 
\node [left] at (-1,0) {$\nabla_{\mathrm{I.b}}^{(3)}$};  
\draw  [->, thick]  (0,0) to  (1.5,1.5); 
\draw  [->, thick]  (0,0) to  (1.5,0.5);
\draw  [->, thick]  (0,0) to (1.5,-0.5);
\draw  [->, thick]  (0,0) to (1.5,-1.5);
\draw  [->, thick]  (3,0) to  (1.5,1.5);
\draw  [->, thick]  (3,0) to  (1.5,0.5);
\draw  [->, thick]  (3,0) to (1.5,-0.5);
\draw  [->, thick]  (3,0) to (1.5,-1.5);
\node [left] at (0,0) {$-2$};
\node [right] at (3,0) {$-2$};
\node at (0.8,1) {\text{\tiny$x$}}; 
\node at (2.4,1) {\text{\tiny$1-x$}}; 
\node at (0.8,0.4) {\text{\tiny$y$}};
\node at (2.4,0.4) {\text{\tiny$1-y$}};
\node at (0.8,-0.1) {\text{\tiny$1-z$}};
\node at (2.4,-0.1) {\text{\tiny$z$}};
\node at (0.2,-1.1) {\text{\tiny$1-x-y+z$}};
\node at (2.6,-1.1) {\text{\tiny$x+y-z$}};
\node [above] at (1.5,1.5) {$1$};
\node [above] at (1.5,0.5) {$1$}; 
\node [above] at (1.5,-0.5) {$1$}; 
\node [above] at (1.5,-1.5) {$1$}; 
\node [right] at (5,0) {$\begin{array}{c|c|c|c|c|c}
v_0 & v_1 & v_2 & v_3 & v_4 &v_5 \\
\hline 
0&1&1&0&0&1\\ 0&0&0&1&1&1\\ 0&0&1&0&1&1
\end{array}$}; 
\end{tikzpicture}
 \end{proposition} 

\begin{proof} We follow the strategy offered by Proposition~\ref{prop:P(Gamma)} and Proposition~\ref{prop:all compressed}:  
for all pairs $(\Gamma^*,\theta)$, where $\Gamma\in \mathcal{L}_3$ and $\theta\in \mathcal{W}(\Gamma)$ we investigate the polytope $\nabla(\Gamma^*,\theta)$.    
The graphs $\Gamma$ in $\mathcal{L}_3$ and the weights 
in $\mathcal{W}(\Gamma)$ 
are shown  below (the source vertices of $\Gamma^*$ correspond bijectively to the 
vertices of $\Gamma$, and the negative integer written next to a vertex of $\Gamma$ indicates the value of the weight for $\Gamma^*$ at the corresponding vertex; recall that the value of a weight in $\mathcal{W}(\Gamma)$ on any sink vertex is $1$).

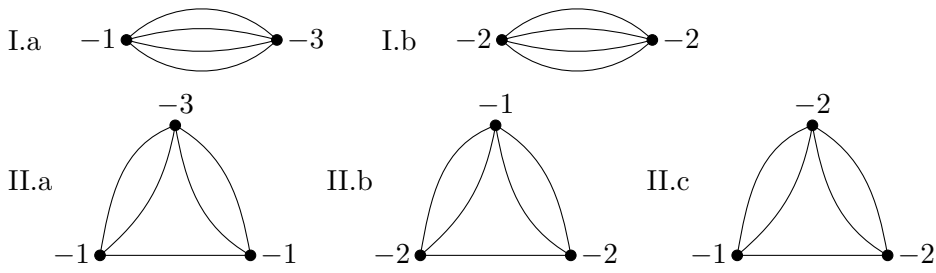
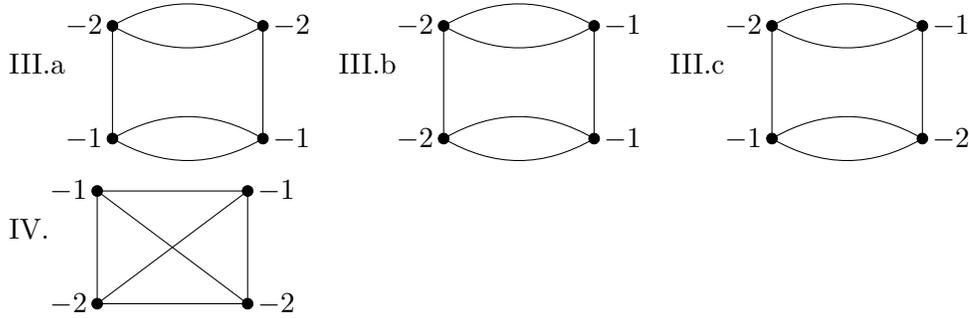
\begin{tikzpicture} 
\foreach \x in {(0,0),(2,0)} \filldraw \x circle (2pt); 
\node [left] at (-1,0) {I.a};  
\draw  [-]  (0,0) to [out=15,in=165] (2,0);
\draw  [-]  (0,0) to [out=45,in=135] (2,0);
\draw  [-]  (0,0) to [out=-15,in=-165] (2,0);
\draw  [-]  (0,0) to [out=-45,in=-135] (2,0);
\node [left] at (0,0) {$-1$};
\node [right] at (2,0) {$-3$};
\end{tikzpicture}
\quad \begin{tikzpicture} 
\foreach \x in {(0,0),(2,0)} \filldraw \x circle (2pt); 
\node [left] at (-1,0) {I.b};  
\draw  [-]  (0,0) to [out=15,in=165] (2,0);
\draw  [-]  (0,0) to [out=45,in=135] (2,0);
\draw  [-]  (0,0) to [out=-15,in=-165] (2,0);
\draw  [-]  (0,0) to [out=-45,in=-135] (2,0);
\node [left] at (0,0) {$-2$};
\node [right] at (2,0) {$-2$};
\end{tikzpicture}

\begin{tikzpicture}
\foreach \x in {(0,0),(1,1.73),(2,0)} \filldraw \x circle (2pt); 
\node [left] at (-0.5,1) {II.a};  
\draw (0,0)--(2,0); 
\draw [-] (0,0) to [out=80,in =-160] (1,1.73);
\draw [-] (0,0) to [out=40,in=-100] (1,1.73); 
\draw [-] (1,1.73) to [out=-80,in=150] (2,0); 
\draw [-] (1,1.73) to [out=-30,in=100] (2,0); 
\node [above] at (1,1.73) {$-3$}; 
\node [left] at (0,0) {$-1$}; 
\node [right] at (2,0) {$-1$};  
\end{tikzpicture}
\begin{tikzpicture}
\foreach \x in {(0,0),(1,1.73),(2,0)} \filldraw \x circle (2pt); 
\node [left] at (-0.5,1) {II.b};  
\draw (0,0)--(2,0); 
\draw [-] (0,0) to [out=80,in =-160] (1,1.73);
\draw [-] (0,0) to [out=40,in=-100] (1,1.73); 
\draw [-] (1,1.73) to [out=-80,in=150] (2,0); 
\draw [-] (1,1.73) to [out=-30,in=100] (2,0); 
\node [above] at (1,1.73) {$-1$}; 
\node [left] at (0,0) {$-2$}; 
\node [right] at (2,0) {$-2$};  
\end{tikzpicture}
\begin{tikzpicture}
\foreach \x in {(0,0),(1,1.73),(2,0)} \filldraw \x circle (2pt); 
\node [left] at (-0.5,1) {II.c};  
\draw (0,0)--(2,0); 
\draw [-] (0,0) to [out=80,in =-160] (1,1.73);
\draw [-] (0,0) to [out=40,in=-100] (1,1.73); 
\draw [-] (1,1.73) to [out=-80,in=150] (2,0); 
\draw [-] (1,1.73) to [out=-30,in=100] (2,0); 
\node [above] at (1,1.73) {$-2$}; 
\node [left] at (0,0) {$-1$}; 
\node [right] at (2,0) {$-2$};  
\end{tikzpicture}

\begin{tikzpicture}
\foreach \x in {(0,0),(0,1.5),(2,0),(2,1.5)} \filldraw \x circle (2pt); 
\draw (0,0)--(0,1.5); \draw (2,0)--(2,1.5);  
\node [left] at (-0.5,1) {III.a};  
\node [left] at (0,0) {$-1$};
\node [left] at (0,1.5) {$-2$};
\node [right] at (2,1.5) {$-2$};
\node [right] at (2,0) {$-1$};
\draw  [-]  (0,0) to [out=30,in=150] (2,0);
\draw  [-]  (0,0) to [out=-30,in=-150] (2,0);
\draw  [-]  (0,1.5) to [out=30,in=150] (2,1.5);
\draw  [-]  (0,1.5) to [out=-30,in=-150] (2,1.5);
\end{tikzpicture}
\begin{tikzpicture}
\foreach \x in {(0,0),(0,1.5),(2,0),(2,1.5)} \filldraw \x circle (2pt); 
\draw (0,0)--(0,1.5); \draw (2,0)--(2,1.5);  
\node [left] at (-0.5,1) {III.b};  
\node [left] at (0,0) {$-2$};
\node [left] at (0,1.5) {$-2$};
\node [right] at (2,1.5) {$-1$};
\node [right] at (2,0) {$-1$};
\draw  [-]  (0,0) to [out=30,in=150] (2,0);
\draw  [-]  (0,0) to [out=-30,in=-150] (2,0);
\draw  [-]  (0,1.5) to [out=30,in=150] (2,1.5);
\draw  [-]  (0,1.5) to [out=-30,in=-150] (2,1.5);
\end{tikzpicture}
\begin{tikzpicture}
\foreach \x in {(0,0),(0,1.5),(2,0),(2,1.5)} \filldraw \x circle (2pt); 
\draw (0,0)--(0,1.5); \draw (2,0)--(2,1.5);  
\node [left] at (-0.5,1) {III.c};  
\node [left] at (0,0) {$-1$};
\node [left] at (0,1.5) {$-2$};
\node [right] at (2,1.5) {$-1$};
\node [right] at (2,0) {$-2$};
\draw  [-]  (0,0) to [out=30,in=150] (2,0);
\draw  [-]  (0,0) to [out=-30,in=-150] (2,0);
\draw  [-]  (0,1.5) to [out=30,in=150] (2,1.5);
\draw  [-]  (0,1.5) to [out=-30,in=-150] (2,1.5);
\end{tikzpicture}

\begin{tikzpicture}  
\foreach \x in {(0,0),(0,1.5),(2,0),(2,1.5)} \filldraw \x circle (2pt); 
\draw (0,0)--(0,1.5)--(2,1.5)--(2,0)--(0,0); \draw (0,0)--(2,1.5); \draw (2,0)--(0,1.5); 
\node [left] at (-0.5,1) {IV.};
\node [left] at (0,0) {$-2$};
\node [left] at (0,1.5) {$-1$};
\node [right] at (2,1.5) {$-1$};
\node [right] at (2,0) {$-2$};
\end{tikzpicture}

Consider first the (quiver,weight) pair corresponding to I.a. 
Applying four times Lemma~\ref{lemma:reduction} we get $(Q,\theta)$ below: 

 
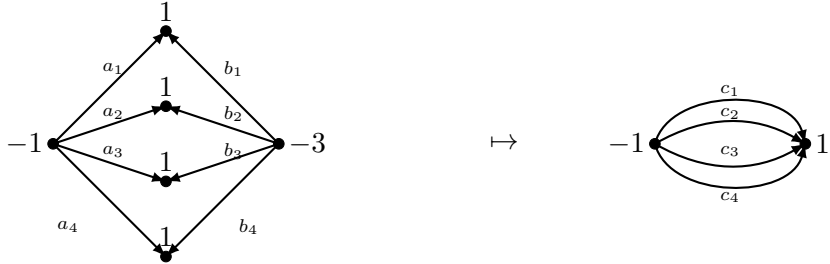
\begin{tikzpicture}[>=latex]
\foreach \x in {(0,0),(3,0),(1.5,1.5),(1.5,0.5),(1.5,-0.5),(1.5,-1.5)} \filldraw \x circle (2pt); 
\draw  [->, thick]  (0,0) to  (1.5,1.5); 
\draw  [->, thick]  (0,0) to  (1.5,0.5);
\draw  [->, thick]  (0,0) to (1.5,-0.5);
\draw  [->, thick]  (0,0) to (1.5,-1.5);
\draw  [->, thick]  (3,0) to  (1.5,1.5);
\draw  [->, thick]  (3,0) to  (1.5,0.5);
\draw  [->, thick]  (3,0) to (1.5,-0.5);
\draw  [->, thick]  (3,0) to (1.5,-1.5);
\node [left] at (0,0) {$-1$};
\node [right] at (3,0) {$-3$};
\node at (0.8,1) {\text{\tiny$a_1$}}; 
\node at (2.4,1) {\text{\tiny$b_1$}}; 
\node at (0.8,0.4) {\text{\tiny$a_2$}};
\node at (2.4,0.4) {\text{\tiny$b_2$}};
\node at (0.8,-0.1) {\text{\tiny$a_3$}};
\node at (2.4,-0.1) {\text{\tiny$b_3$}};
\node at (0.2,-1.1) {\text{\tiny$a_4$}};
\node at (2.6,-1.1) {\text{\tiny$b_4$}};
\node [above] at (1.5,1.5) {$1$};
\node [above] at (1.5,0.5) {$1$}; 
\node [above] at (1.5,-0.5) {$1$}; 
\node [above] at (1.5,-1.5) {$1$}; 
\node at (6,0) {$\mapsto$\quad}; 
\foreach \x in {(8,0),(10,0)} \filldraw \x circle (2pt); 
\draw  [->, thick]  (8,0) to [out=30,in=150] (10,0); 
\node [above] at (9,0.5) {\text{\tiny$c_1$}}; 
\node [above] at (9,0.2) {\text{\tiny$c_2$}};
\node [above] at (9,-0.3) {\text{\tiny$c_3$}}; 
\node [below] at (9,-0.5) {\text{\tiny$c_4$}};
\draw  [->, thick]  (8,0) to [out=75,in=105] (10,0);
\draw  [->, thick]  (8,0) to [out=-30,in=-150] (10,0);
\draw  [->, thick]  (8,0) to [out=-75,in=-105] (10,0);
\node [left] at (8,0) {$-1$};
\node [right] at (10,0) {$1$};
\end{tikzpicture}

The polytope $\nabla(Q,\theta)$ is the 3-simplex with vertices 
$(0,0,0)$, $(1,0,0)$, $(0,1,0)$, $(0,0,1)$, that we denoted by $\nabla_{\mathrm{I.a}}^{(3)}$ in the statement.  
After tightening the (quiver,weight) pairs belonging to each of  II.a, II.b, III.b and IV  
we get the same pair $(Q,\theta)$ as above for I.a. 

After tightening the (quiver,weight) pair belonging to III.a we get the disjoint union of two copies of   \begin{tikzpicture}[>=latex']
\foreach \x in {(0,0),(1,0)} \filldraw \x circle (2pt);   
\draw  [->, thick]  (0,0) to [out=45,in=135] (1,0);
\draw  [->, thick]  (0,0) to [out=-45,in=-135] (1,0);
\node [left] at (0,0) {$-1$};
\node [right] at (1,0) {$1$};
\end{tikzpicture}, hence it gives a $2$-dimensional polytope (the unit square). 

Let us turn to the (quiver,weight) pair belonging to II.c. Applying three times Lemma~\ref{lemma:reduction} we get to the pair $(Q',\theta')$: 

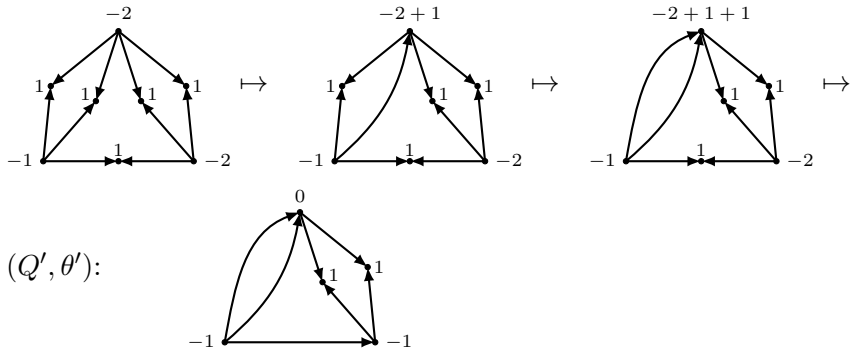
\begin{tikzpicture}[>=latex]
\foreach \x in {(0,0),(1,1.73),(2,0),(1.3,0.8),(1.9,1),(0.7,0.8),(0.1,1),(1,0)} \filldraw \x circle (1pt); 
\node  at (1.45,0.9) {\text{\tiny{$1$}}}; 
\node at (2.05,1) {\text{\tiny{$1$}}}; 
\draw [->, thick] (0,0) to (1,0); 
\draw [->, thick] (2,0) to (1,0); \node at (1,0.15) {\text{\tiny{$1$}}}; 
\draw [->, thick] (0,0) to (0.7,0.8); \node at (0.56,0.9) {\text{\tiny{$1$}}}; 
\draw [->, thick] (1,1.73) to  (0.1,1); 
\draw [->, thick] (1,1.73) to (0.7,0.8);
\draw [->, thick] (0,0) to  (0.1,1); \node at (-0.05,1) {\text{\tiny{$1$}}}; 
\draw [->, thick] (1,1.73) to (1.3,0.8); 
\draw [->, thick] (2,0) to (1.3,0.8); 
\draw [->, thick] (1,1.73) to  (1.9,1); 
\draw [->, thick] (2,0) to  (1.9,1); 
\node [above] at (1,1.73) {\text{\tiny{$-2$}}}; 
\node [left] at (0,0) {\text{\tiny{$-1$}}}; 
\node [right] at (2,0) {\text{\tiny{$-2$}}};  
\node at (2.8,1) {$\mapsto$};
\end{tikzpicture} 
\begin{tikzpicture}[>=latex]
\foreach \x in {(0,0),(1,1.73),(2,0),(1.3,0.8),(1.9,1),(1,0),(0.1,1)} \filldraw \x circle (1pt); 
\node  at (1.45,0.9) {\text{\tiny{$1$}}}; 
\node at (2.05,1) {\text{\tiny{$1$}}}; 
\draw [->, thick] (0,0) to (1,0); 
\draw [->, thick] (2,0) to (1,0); \node at (1,0.15) {\text{\tiny{$1$}}}; 
\draw [->, thick] (0,0) to [out=40,in=-100] (1,1.73); 
\draw [->, thick] (1,1.73) to  (0.1,1); 
\draw [->, thick] (0,0) to  (0.1,1); \node at (-0.05,1) {\text{\tiny{$1$}}}; 
\draw [->, thick] (1,1.73) to (1.3,0.8); 
\draw [->, thick] (2,0) to (1.3,0.8); 
\draw [->, thick] (1,1.73) to  (1.9,1); 
\draw [->, thick] (2,0) to  (1.9,1); 
\node [above] at (1,1.73) {\text{\tiny{$-2+1$}}}; 
\node [left] at (0,0) {\text{\tiny{$-1$}}}; 
\node [right] at (2,0) {\text{\tiny{$-2$}}};  
\node at (2.8,1) {$\mapsto$};
\end{tikzpicture} 
\begin{tikzpicture}[>=latex]
\foreach \x in {(0,0),(1,1.73),(2,0),(1.3,0.8),(1.9,1),(1,0)} \filldraw \x circle (1pt); 
\node  at (1.45,0.9) {\text{\tiny{$1$}}}; 
\node at (2.05,1) {\text{\tiny{$1$}}}; 
\draw [->, thick] (0,0) to (1,0); 
\draw [->, thick] (2,0) to (1,0); \node at (1,0.15) {\text{\tiny{$1$}}}; 
\draw [->, thick] (0,0) to [out=80,in =-160] (1,1.73);
\draw [->, thick] (0,0) to [out=40,in=-100] (1,1.73); 
\draw [->, thick] (1,1.73) to (1.3,0.8); 
\draw [->, thick] (2,0) to (1.3,0.8); 
\draw [->, thick] (1,1.73) to  (1.9,1); 
\draw [->, thick] (2,0) to  (1.9,1); 
\node [above] at (1,1.73) {\text{\tiny{$-2+1+1$}}}; 
\node [left] at (0,0) {\text{\tiny{$-1$}}}; 
\node [right] at (2,0) {\text{\tiny{$-2$}}};  
\node at (2.8,1) {$\mapsto$};
\end{tikzpicture} 

\begin{tikzpicture}[>=latex]
\foreach \x in {(0,0),(1,1.73),(2,0),(1.3,0.8),(1.9,1)} \filldraw \x circle (1pt); 
\node [left] at (-1.5,1) {$(Q',\theta')$:};  
\node  at (1.45,0.9) {\text{\tiny{$1$}}}; 
\node at (2.05,1) {\text{\tiny{$1$}}}; 
\draw [->, thick] (0,0) to (2,0); 
\draw [->, thick] (0,0) to [out=80,in =-160] (1,1.73);
\draw [->, thick] (0,0) to [out=40,in=-100] (1,1.73); 
\draw [->, thick] (1,1.73) to (1.3,0.8); 
\draw [->, thick] (2,0) to (1.3,0.8); 
\draw [->, thick] (1,1.73) to  (1.9,1); 
\draw [->, thick] (2,0) to  (1.9,1); 
\node [above] at (1,1.73) {\text{\tiny{$0$}}}; 
\node [left] at (0,0) {\text{\tiny{$-1$}}}; 
\node [right] at (2,0) {\text{\tiny{$-1$}}};  
\end{tikzpicture} 

Next we choose a  spanning tree in $Q'$: 

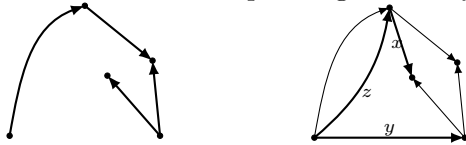
\begin{tikzpicture}[>=latex]
\foreach \x in {(0,0),(1,1.73),(2,0),(1.3,0.8),(1.9,1)} \filldraw \x circle (1pt); 
\draw [->, thick] (0,0) to [out=80,in =-160] (1,1.73);
\draw [->, thick] (2,0) to (1.3,0.8); 
\draw [->, thick] (1,1.73) to  (1.9,1); 
\draw [->, thick] (2,0) to  (1.9,1); 
\node at (2.6,0) {\qquad\qquad\qquad\qquad};
\end{tikzpicture}
\begin{tikzpicture}[>=latex]
\foreach \x in {(0,0),(1,1.73),(2,0),(1.3,0.8),(1.9,1)} \filldraw \x circle (1pt); 
\draw [->, thick] (0,0) to (2,0); 
\draw [->] (0,0) to [out=80,in =-160] (1,1.73);
\draw [->, thick] (0,0) to [out=40,in=-100] (1,1.73); 
\draw [->, thick] (1,1.73) to (1.3,0.8); 
\draw [->] (2,0) to (1.3,0.8); 
\draw [->] (1,1.73) to  (1.9,1); 
\draw [->] (2,0) to  (1.9,1); 
\node at (1,0.1) {\text{\tiny$y$}};
\node at (0.7,0.6) {\text{\tiny$z$}};
\node at (1.1,1.25) {\text{\tiny$x$}};
\end{tikzpicture}

The coordinates $x,y,z$ corresponding 
to the arrows in the complement of the spanning tree are free parameters in the 
affine subspace $\mathcal{A}(Q',\theta')\subset \mathbb{R}^{Q'_1}$; going through the arrows in an appropriate order we express all other coordinates of a general point in $\mathcal{A}(Q',\theta')$ in terms of $x,y,z$:  

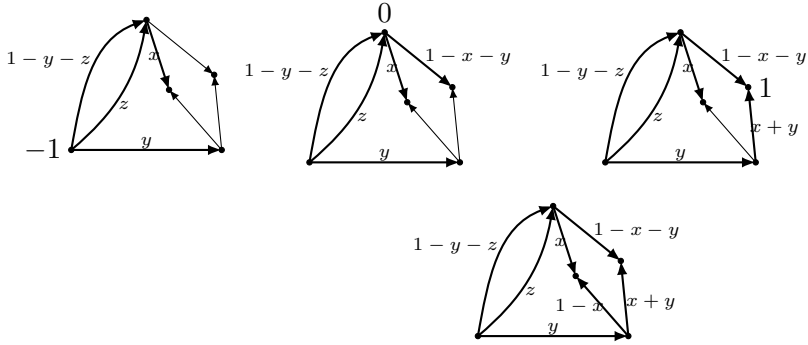
\begin{tikzpicture}[>=latex]
\foreach \x in {(0,0),(1,1.73),(2,0),(1.3,0.8),(1.9,1)} \filldraw \x circle (1pt); 
\draw [->, thick] (0,0) to (2,0); 
\draw [->, thick] (0,0) to [out=80,in =-160] (1,1.73);
\draw [->, thick] (0,0) to [out=40,in=-100] (1,1.73); 
\draw [->, thick] (1,1.73) to (1.3,0.8); 
\draw [->] (2,0) to (1.3,0.8); 
\draw [->] (1,1.73) to  (1.9,1); 
\draw [->] (2,0) to  (1.9,1); 
\node [left] at (0,0) {$-1$}; 
\node at (1,0.1) {\text{\tiny$y$}};
\node at (0.7,0.6) {\text{\tiny$z$}};
\node at (1.1,1.25) {\text{\tiny$x$}};
\node [left] at (0.4,1.2) {\text{\tiny$1-y-z$}};
\end{tikzpicture}
\begin{tikzpicture}[>=latex]
\foreach \x in {(0,0),(1,1.73),(2,0),(1.3,0.8),(1.9,1)} \filldraw \x circle (1pt); 
\draw [->, thick] (0,0) to (2,0); 
\draw [->, thick] (0,0) to [out=80,in =-160] (1,1.73);
\draw [->, thick] (0,0) to [out=40,in=-100] (1,1.73); 
\draw [->, thick] (1,1.73) to (1.3,0.8); 
\draw [->] (2,0) to (1.3,0.8); 
\draw [->, thick] (1,1.73) to  (1.9,1); 
\draw [->] (2,0) to  (1.9,1); 
\node [above] at (1,1.73) {$0$}; 
\node at (1,0.1) {\text{\tiny$y$}};
\node at (0.7,0.6) {\text{\tiny$z$}};
\node at (1.1,1.25) {\text{\tiny$x$}};
\node [left] at (0.4,1.2) {\text{\tiny$1-y-z$}};
\node [right] at (1.4,1.4) {\text{\tiny$1-x-y$}};
\end{tikzpicture}
\begin{tikzpicture}[>=latex]
\foreach \x in {(0,0),(1,1.73),(2,0),(1.3,0.8),(1.9,1)} \filldraw \x circle (1pt); 
\node [right] at (1.9,1) {$1$}; 
\draw [->, thick] (0,0) to (2,0); 
\draw [->, thick] (0,0) to [out=80,in =-160] (1,1.73);
\draw [->, thick] (0,0) to [out=40,in=-100] (1,1.73); 
\draw [->, thick] (1,1.73) to (1.3,0.8); 
\draw [->] (2,0) to (1.3,0.8); 
\draw [->, thick] (1,1.73) to  (1.9,1); 
\draw [->, thick] (2,0) to  (1.9,1); 
\node at (1,0.1) {\text{\tiny$y$}};
\node at (0.7,0.6) {\text{\tiny$z$}};
\node at (1.1,1.25) {\text{\tiny$x$}};
\node at (2.25,0.45) {\text{\tiny$x+y$}};
\node [left] at (0.4,1.2) {\text{\tiny$1-y-z$}};
\node [right] at (1.4,1.4) {\text{\tiny$1-x-y$}};
\end{tikzpicture}

\begin{center}
\begin{tikzpicture}[>=latex]
\foreach \x in {(0,0),(1,1.73),(2,0),(1.3,0.8),(1.9,1)} \filldraw \x circle (1pt); 
\draw [->, thick] (0,0) to (2,0); 
\draw [->, thick] (0,0) to [out=80,in =-160] (1,1.73);
\draw [->, thick] (0,0) to [out=40,in=-100] (1,1.73); 
\draw [->, thick] (1,1.73) to (1.3,0.8); 
\draw [->, thick] (2,0) to (1.3,0.8); 
\draw [->, thick] (1,1.73) to  (1.9,1); 
\draw [->, thick] (2,0) to  (1.9,1); 
\node at (1,0.1) {\text{\tiny$y$}};
\node at (0.7,0.6) {\text{\tiny$z$}};
\node at (1.35,0.4) {\text{\tiny$1-x$}};
\node at (1.1,1.25) {\text{\tiny$x$}};
\node at (2.3,0.45) {\text{\tiny$x+y$}};
\node [left] at (0.4,1.2) {\text{\tiny$1-y-z$}};
\node [right] at (1.4,1.4) {\text{\tiny$1-x-y$}};
\end{tikzpicture}
\end{center} 

This shows that $\nabla(Q',\theta')$ can be viewed as the polytope  in $\mathbb{R}^3$ 
defined by the inequalities  $x\ge0$, 
$y\ge 0$, $z\ge 0$, $1-x\ge 0$, $x+y\ge 0$, 
$1-x-y\ge 0$, $1-y-z\ge 0$. 
The inequalities $x\ge 0$, $y\ge 0$ imply $x+y\ge 0$, hence the arrow 
above labeled with $x+y$ is contractible (see Section~\ref{sec:tightness}). 
Also the inequalities $y\ge 0$, $1-x-y\ge 0$ imply 
$1-x\ge 0$, hence the arrow labeled by $1-x$ is contractible.  
Contracting these two arrows we end up with the (quiver,weight) pair 

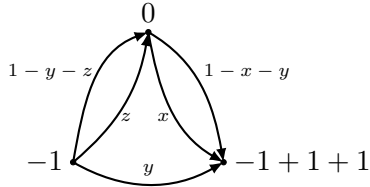
\begin{center}
\begin{tikzpicture}[>=latex]
\foreach \x in {(0,0),(1,1.73),(2,0)} \filldraw \x circle (1pt); 
\draw [->, thick] (0,0) to [out=-30,in =-150] (2,0); 
\draw [->, thick] (0,0) to [out=80,in =-160] (1,1.73);
\draw [->, thick] (0,0) to [out=40,in=-100] (1,1.73); 
\draw [->, thick] (1,1.73) to [out=-80,in=150] (2,0); 
\draw [->, thick] (1,1.73) to [out=-30,in=100] (2,0); 
\node [above] at (1,1.73) {$0$}; 
\node [left] at (0,0) {$-1$}; 
\node [right] at (2,0) {$-1+1+1$};  
\node at (1,-0.1) {\text{\tiny$y$}};
\node at (0.7,0.6) {\text{\tiny$z$}};
\node at (1.2,0.6) {\text{\tiny$x$}};
\node [left] at (0.4,1.2) {\text{\tiny$1-y-z$}};
\node [right] at (1.6,1.2) {\text{\tiny$1-x-y$}};
\end{tikzpicture}
\end{center}

and the inequalities
given in the statement for $\nabla_{\mathrm{II.c}}^{(3)}$. 
It is then an easy matter to list the lattice points in our polytope. 

A similar analysis yields that the (quiver,weight) pair belonging to III.c results in the same polytope as II.c, whereas after tightening the (quiver,weight) pair belonging to I.b we get the 
polytope $\nabla_{\mathrm{I.b}}^{(3)}$ given in the statement. 
\end{proof} 

In the course of the proof of Proposition~\ref{prop:all 3-dim compressed} we showed the following: 

\begin{proposition}\label{prop:3-cells}
\begin{itemize} 
\item[(i)] $\mathcal{P}(\Gamma_{\mathrm{I}})=\{\nabla_{\mathrm{I.a}}^{(3)},\quad \nabla_{\mathrm{I.b}}^{(3)}\}$
\item[(ii)] $\mathcal{P}(\Gamma_{\mathrm{II}})=\{\nabla_{\mathrm{I.a}}^{(3)},\quad \nabla_{\mathrm{II.c}}^{(3)}\}$
\item[(iii)] $\mathcal{P}(\Gamma_{\mathrm{III}})=\{\nabla_{\mathrm{I.a}}^{(3)},\quad \nabla_{\mathrm{II.c}}^{(3)}\}$
\item[(iv)] $\mathcal{P}(\Gamma_{\mathrm{IV}})=\{\nabla_{\mathrm{I.a}}^{(3)}\}$
\end{itemize}
\end{proposition}

\begin{remark}\label{remark:all compressed are flow} A compressed $d$-dimensional flow polytope is equivalent to  a $0-1$-polytope in $\mathbb{R}^d$ 
(this is a consequence for example of  the method of proof of Proposition~\ref{prop:all 3-dim compressed}). 
It is easy to check that all $3$-dimensional compressed $0-1$-polytopes in $\mathbb{R}^3$ are in fact flow polytopes. 
This does not hold for higher dimensions, see Example~\ref{example:4 dim compressed non-flow}.  
\end{remark}


\section{$4$-dimensional cells of quivers with $3$-regular chassis}\label{sec:4-cells}  

In this section we decribe the polytopes in $\mathcal{P}(\Gamma)$ for $\Gamma\in \mathcal{L}_4^{\mathrm{3-reg}}$, 
in order to apply later effectively Corollary~\ref{cor:3-regular} for $4$-dimensional flow polytopes. 
There are five graphs $\Gamma_{\mathrm{I}},\Gamma_{\mathrm{II}},\Gamma_{\mathrm{III}},
\Gamma_{\mathrm{IV}},\Gamma_{\mathrm{V}}$  in $\mathcal{L}_4^{\mathrm{3-reg}}$. These graphs $\Gamma$ together with the  corresponding weights in 
$\mathcal{W}(\Gamma)$ are given below (for $\theta\in \mathcal{W}(\Gamma)$ we indicate the value of $\theta$ at the vertices in $\Gamma_0\subset \Gamma^*_0$; the value of $\theta$ on the remaining vertices of $\Gamma^*$ is $1$):  

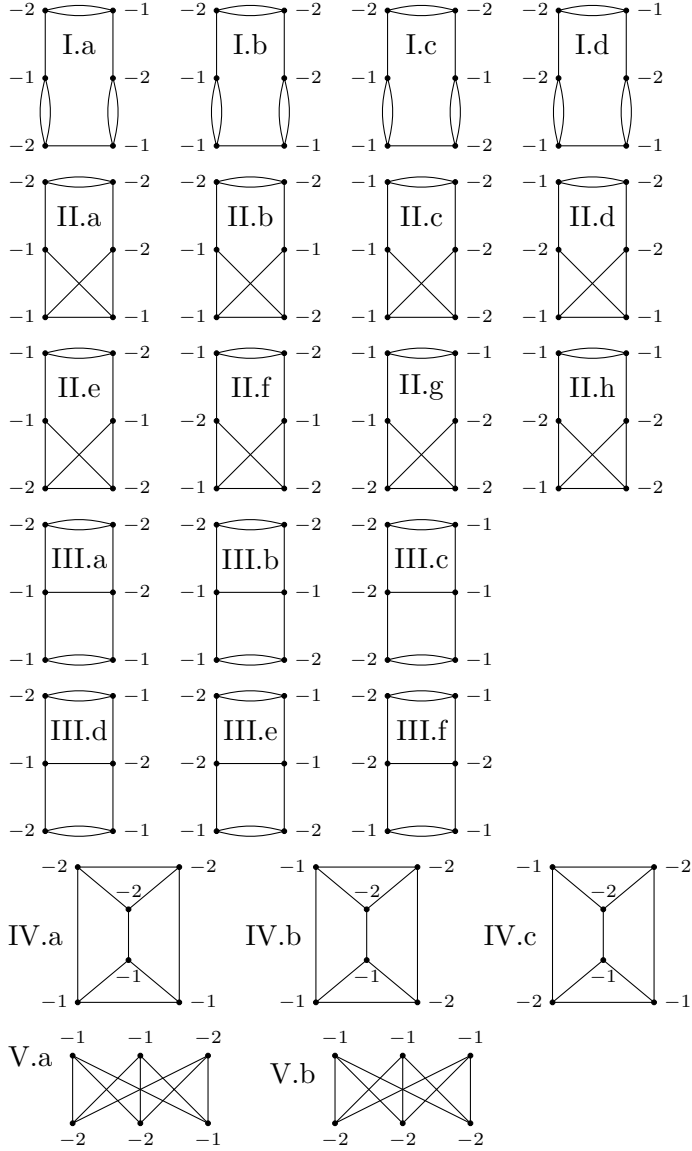
\begin{tikzpicture}[scale=0.45] 
\foreach \x in {(0,0),(2,0),(0,2),(2,2),(0,4),(2,4)} \filldraw \x circle (2pt); 
\node  at (1,3) {I.a};  
\draw  [-]  (0,4) to [out=15,in=165] (2,4);
\draw  [-]  (0,4) to [out=-15,in=-165] (2,4);
\draw [-] (0,4) to (0,2); 
\draw [-] (2,4) to (2,2);
\draw [-]  (0,2) to [out=-105,in=105] (0,0); 
\draw [-]  (0,2) to [out=-75,in=75] (0,0); 
\draw [-]  (2,2) to [out=-105,in=105] (2,0); 
\draw [-]  (2,2) to [out=-75,in=75] (2,0); 
\draw [-] (0,0) to (2,0);
\node [left] at (0,4) {\text{\tiny$-2$}}; \node [right] at (2,4) {\text{\tiny$-1$}}; \node [right] at (2,2) {\text{\tiny$-2$}}; \node [right] at (2,0) {\text{\tiny$-1$}}; 
\node [left] at (0,0) {\text{\tiny$-2$}}; \node [left] at (0,2) {\text{\tiny$-1$}}; 
\end{tikzpicture}   
\begin{tikzpicture}[scale=0.45] 
\foreach \x in {(0,0),(2,0),(0,2),(2,2),(0,4),(2,4)} \filldraw \x circle (2pt); 
\node at (1,3) {I.b};  
\draw  [-]  (0,4) to [out=15,in=165] (2,4);
\draw  [-]  (0,4) to [out=-15,in=-165] (2,4);
\draw [-] (0,4) to (0,2); 
\draw [-] (2,4) to (2,2);
\draw [-]  (0,2) to [out=-105,in=105] (0,0); 
\draw [-]  (0,2) to [out=-75,in=75] (0,0); 
\draw [-]  (2,2) to [out=-105,in=105] (2,0); 
\draw [-]  (2,2) to [out=-75,in=75] (2,0); 
\draw [-] (0,0) to (2,0);
\node [left] at (0,4) {\text{\tiny$-2$}}; \node [right] at (2,4) {\text{\tiny$-2$}}; \node [right] at (2,2) {\text{\tiny$-2$}}; \node [right] at (2,0) {\text{\tiny$-1$}}; 
\node [left] at (0,0) {\text{\tiny$-1$}}; \node [left] at (0,2) {\text{\tiny$-1$}}; 
\end{tikzpicture}   
\begin{tikzpicture}[scale=0.45] 
\foreach \x in {(0,0),(2,0),(0,2),(2,2),(0,4),(2,4)} \filldraw \x circle (2pt); 
\node at (1,3) {I.c};  
\draw  [-]  (0,4) to [out=15,in=165] (2,4);
\draw  [-]  (0,4) to [out=-15,in=-165] (2,4);
\draw [-] (0,4) to (0,2); 
\draw [-] (2,4) to (2,2);
\draw [-]  (0,2) to [out=-105,in=105] (0,0); 
\draw [-]  (0,2) to [out=-75,in=75] (0,0); 
\draw [-]  (2,2) to [out=-105,in=105] (2,0); 
\draw [-]  (2,2) to [out=-75,in=75] (2,0); 
\draw [-] (0,0) to (2,0);
\node [left] at (0,4) {\text{\tiny$-2$}}; \node [right] at (2,4) {\text{\tiny$-2$}}; \node [right] at (2,2) {\text{\tiny$-1$}}; \node [right] at (2,0) {\text{\tiny$-2$}}; 
\node [left] at (0,0) {\text{\tiny$-1$}}; \node [left] at (0,2) {\text{\tiny$-1$}}; 
\end{tikzpicture}   
\begin{tikzpicture}[scale=0.45] 
\foreach \x in {(0,0),(2,0),(0,2),(2,2),(0,4),(2,4)} \filldraw \x circle (2pt); 
\node at (1,3) {I.d};  
\draw  [-]  (0,4) to [out=15,in=165] (2,4);
\draw  [-]  (0,4) to [out=-15,in=-165] (2,4);
\draw [-] (0,4) to (0,2); 
\draw [-] (2,4) to (2,2);
\draw [-]  (0,2) to [out=-105,in=105] (0,0); 
\draw [-]  (0,2) to [out=-75,in=75] (0,0); 
\draw [-]  (2,2) to [out=-105,in=105] (2,0); 
\draw [-]  (2,2) to [out=-75,in=75] (2,0); 
\draw [-] (0,0) to (2,0);
\node [left] at (0,4) {\text{\tiny$-2$}}; \node [right] at (2,4) {\text{\tiny$-1$}}; \node [right] at (2,2) {\text{\tiny$-2$}}; \node [right] at (2,0) {\text{\tiny$-1$}}; 
\node [left] at (0,0) {\text{\tiny$-1$}}; \node [left] at (0,2) {\text{\tiny$-2$}}; 
\end{tikzpicture} 

\begin{tikzpicture}[scale=0.45] 
\foreach \x in {(0,0),(2,0),(0,2),(2,2),(0,4),(2,4)} \filldraw \x circle (2pt); 
\node at (1,3) {II.a};  
\draw  [-]  (0,4) to [out=15,in=165] (2,4);
\draw  [-]  (0,4) to [out=-15,in=-165] (2,4);
\draw [-] (0,4) to (0,2); 
\draw [-] (2,4) to (2,2);
\draw [-]  (0,2) to  (0,0); 
\draw [-]  (2,2) to  (2,0); 
\draw [-] (0,0) to (2,0); 
\draw [-] (0,2) to  (2,0);
\draw [-] (0,0) to  (2,2); 
\node [left] at (0,4) {\text{\tiny$-2$}}; \node [right] at (2,4) {\text{\tiny$-2$}}; \node [right] at (2,2) {\text{\tiny$-2$}}; \node [right] at (2,0) {\text{\tiny$-1$}}; 
\node [left] at (0,0) {\text{\tiny$-1$}}; \node [left] at (0,2) {\text{\tiny$-1$}}; 
\end{tikzpicture} 
\begin{tikzpicture}[scale=0.45] 
\foreach \x in {(0,0),(2,0),(0,2),(2,2),(0,4),(2,4)} \filldraw \x circle (2pt); 
\node at (1,3) {II.b};  
\draw  [-]  (0,4) to [out=15,in=165] (2,4);
\draw  [-]  (0,4) to [out=-15,in=-165] (2,4);
\draw [-] (0,4) to (0,2); 
\draw [-] (2,4) to (2,2);
\draw [-]  (0,2) to  (0,0); 
\draw [-]  (2,2) to  (2,0); 
\draw [-] (0,0) to (2,0); 
\draw [-] (0,2) to  (2,0);
\draw [-] (0,0) to  (2,2); 
\node [left] at (0,4) {\text{\tiny$-2$}}; \node [right] at (2,4) {\text{\tiny$-2$}}; \node [right] at (2,2) {\text{\tiny$-1$}}; \node [right] at (2,0) {\text{\tiny$-2$}}; 
\node [left] at (0,0) {\text{\tiny$-1$}}; \node [left] at (0,2) {\text{\tiny$-1$}}; 
\end{tikzpicture} 
\begin{tikzpicture}[scale=0.45] 
\foreach \x in {(0,0),(2,0),(0,2),(2,2),(0,4),(2,4)} \filldraw \x circle (2pt); 
\node at (1,3) {II.c};  
\draw  [-]  (0,4) to [out=15,in=165] (2,4);
\draw  [-]  (0,4) to [out=-15,in=-165] (2,4);
\draw [-] (0,4) to (0,2); 
\draw [-] (2,4) to (2,2);
\draw [-]  (0,2) to  (0,0); 
\draw [-]  (2,2) to  (2,0); 
\draw [-] (0,0) to (2,0); 
\draw [-] (0,2) to  (2,0);
\draw [-] (0,0) to  (2,2); 
\node [left] at (0,4) {\text{\tiny$-1$}}; \node [right] at (2,4) {\text{\tiny$-2$}}; \node [right] at (2,2) {\text{\tiny$-2$}}; \node [right] at (2,0) {\text{\tiny$-2$}}; 
\node [left] at (0,0) {\text{\tiny$-1$}}; \node [left] at (0,2) {\text{\tiny$-1$}}; 
\end{tikzpicture} 
\begin{tikzpicture}[scale=0.45] 
\foreach \x in {(0,0),(2,0),(0,2),(2,2),(0,4),(2,4)} \filldraw \x circle (2pt); 
\node at (1,3) {II.d};  
\draw  [-]  (0,4) to [out=15,in=165] (2,4);
\draw  [-]  (0,4) to [out=-15,in=-165] (2,4);
\draw [-] (0,4) to (0,2); 
\draw [-] (2,4) to (2,2);
\draw [-]  (0,2) to  (0,0); 
\draw [-]  (2,2) to  (2,0); 
\draw [-] (0,0) to (2,0); 
\draw [-] (0,2) to  (2,0);
\draw [-] (0,0) to  (2,2); 
\node [left] at (0,4) {\text{\tiny$-1$}}; \node [right] at (2,4) {\text{\tiny$-2$}}; \node [right] at (2,2) {\text{\tiny$-2$}}; \node [right] at (2,0) {\text{\tiny$-1$}}; 
\node [left] at (0,0) {\text{\tiny$-1$}}; \node [left] at (0,2) {\text{\tiny$-2$}}; 
\end{tikzpicture} 

\begin{tikzpicture}[scale=0.45] 
\foreach \x in {(0,0),(2,0),(0,2),(2,2),(0,4),(2,4)} \filldraw \x circle (2pt); 
\node at (1,3) {II.e};  
\draw  [-]  (0,4) to [out=15,in=165] (2,4);
\draw  [-]  (0,4) to [out=-15,in=-165] (2,4);
\draw [-] (0,4) to (0,2); 
\draw [-] (2,4) to (2,2);
\draw [-]  (0,2) to  (0,0); 
\draw [-]  (2,2) to  (2,0); 
\draw [-] (0,0) to (2,0); 
\draw [-] (0,2) to  (2,0);
\draw [-] (0,0) to  (2,2); 
\node [left] at (0,4) {\text{\tiny$-1$}}; \node [right] at (2,4) {\text{\tiny$-2$}}; \node [right] at (2,2) {\text{\tiny$-1$}}; \node [right] at (2,0) {\text{\tiny$-2$}}; 
\node [left] at (0,0) {\text{\tiny$-2$}}; \node [left] at (0,2) {\text{\tiny$-1$}}; 
\end{tikzpicture} 
\begin{tikzpicture}[scale=0.45] 
\foreach \x in {(0,0),(2,0),(0,2),(2,2),(0,4),(2,4)} \filldraw \x circle (2pt); 
\node at (1,3) {II.f};  
\draw  [-]  (0,4) to [out=15,in=165] (2,4);
\draw  [-]  (0,4) to [out=-15,in=-165] (2,4);
\draw [-] (0,4) to (0,2); 
\draw [-] (2,4) to (2,2);
\draw [-]  (0,2) to  (0,0); 
\draw [-]  (2,2) to  (2,0); 
\draw [-] (0,0) to (2,0); 
\draw [-] (0,2) to  (2,0);
\draw [-] (0,0) to  (2,2); 
\node [left] at (0,4) {\text{\tiny$-1$}}; \node [right] at (2,4) {\text{\tiny$-2$}}; \node [right] at (2,2) {\text{\tiny$-1$}}; \node [right] at (2,0) {\text{\tiny$-2$}}; 
\node [left] at (0,0) {\text{\tiny$-1$}}; \node [left] at (0,2) {\text{\tiny$-2$}}; 
\end{tikzpicture} 
\begin{tikzpicture}[scale=0.45] 
\foreach \x in {(0,0),(2,0),(0,2),(2,2),(0,4),(2,4)} \filldraw \x circle (2pt); 
\node at (1,3) {II.g};  
\draw  [-]  (0,4) to [out=15,in=165] (2,4);
\draw  [-]  (0,4) to [out=-15,in=-165] (2,4);
\draw [-] (0,4) to (0,2); 
\draw [-] (2,4) to (2,2);
\draw [-]  (0,2) to  (0,0); 
\draw [-]  (2,2) to  (2,0); 
\draw [-] (0,0) to (2,0); 
\draw [-] (0,2) to  (2,0);
\draw [-] (0,0) to  (2,2); 
\node [left] at (0,4) {\text{\tiny$-1$}}; \node [right] at (2,4) {\text{\tiny$-1$}}; \node [right] at (2,2) {\text{\tiny$-2$}}; \node [right] at (2,0) {\text{\tiny$-2$}}; 
\node [left] at (0,0) {\text{\tiny$-2$}}; \node [left] at (0,2) {\text{\tiny$-1$}}; 
\end{tikzpicture} 
\begin{tikzpicture}[scale=0.45] 
\foreach \x in {(0,0),(2,0),(0,2),(2,2),(0,4),(2,4)} \filldraw \x circle (2pt); 
\node at (1,3) {II.h};  
\draw  [-]  (0,4) to [out=15,in=165] (2,4);
\draw  [-]  (0,4) to [out=-15,in=-165] (2,4);
\draw [-] (0,4) to (0,2); 
\draw [-] (2,4) to (2,2);
\draw [-]  (0,2) to  (0,0); 
\draw [-]  (2,2) to  (2,0); 
\draw [-] (0,0) to (2,0); 
\draw [-] (0,2) to  (2,0);
\draw [-] (0,0) to  (2,2); 
\node [left] at (0,4) {\text{\tiny$-1$}}; \node [right] at (2,4) {\text{\tiny$-1$}}; \node [right] at (2,2) {\text{\tiny$-2$}}; \node [right] at (2,0) {\text{\tiny$-2$}}; 
\node [left] at (0,0) {\text{\tiny$-1$}}; \node [left] at (0,2) {\text{\tiny$-2$}}; 
\end{tikzpicture} 

\begin{tikzpicture}[scale=0.45] 
\foreach \x in {(0,0),(2,0),(0,2),(2,2),(0,4),(2,4)} \filldraw \x circle (2pt); 
\node at (1,3) {III.a};  
\draw  [-]  (0,4) to [out=15,in=165] (2,4);
\draw  [-]  (0,4) to [out=-15,in=-165] (2,4);
\draw [-] (0,4) to (0,2); 
\draw [-] (2,4) to (2,2);
\draw [-]  (0,2) to  (0,0); 
\draw [-]  (2,2) to  (2,0); 
\draw [-]  (0,2) to (2,2); 
\draw [-] (0,0) to [out=15,in=165] (2,0); 
\draw [-] (0,0) to [out=-15,in=-165] (2,0); 
\node [left] at (0,4) {\text{\tiny$-2$}}; \node [right] at (2,4) {\text{\tiny$-2$}}; \node [right] at (2,2) {\text{\tiny$-2$}}; \node [right] at (2,0) {\text{\tiny$-1$}}; 
\node [left] at (0,0) {\text{\tiny$-1$}}; \node [left] at (0,2) {\text{\tiny$-1$}}; 
\end{tikzpicture} 
\begin{tikzpicture}[scale=0.45] 
\foreach \x in {(0,0),(2,0),(0,2),(2,2),(0,4),(2,4)} \filldraw \x circle (2pt); 
\node at (1,3) {III.b};  
\draw  [-]  (0,4) to [out=15,in=165] (2,4);
\draw  [-]  (0,4) to [out=-15,in=-165] (2,4);
\draw [-] (0,4) to (0,2); 
\draw [-] (2,4) to (2,2);
\draw [-]  (0,2) to  (0,0); 
\draw [-]  (2,2) to  (2,0); 
\draw [-]  (0,2) to (2,2); 
\draw [-] (0,0) to [out=15,in=165] (2,0); 
\draw [-] (0,0) to [out=-15,in=-165] (2,0); 
\node [left] at (0,4) {\text{\tiny$-2$}}; \node [right] at (2,4) {\text{\tiny$-2$}}; \node [right] at (2,2) {\text{\tiny$-1$}}; \node [right] at (2,0) {\text{\tiny$-2$}}; 
\node [left] at (0,0) {\text{\tiny$-1$}}; \node [left] at (0,2) {\text{\tiny$-1$}}; 
\end{tikzpicture} 
\begin{tikzpicture}[scale=0.45] 
\foreach \x in {(0,0),(2,0),(0,2),(2,2),(0,4),(2,4)} \filldraw \x circle (2pt); 
\node at (1,3) {III.c};  
\draw  [-]  (0,4) to [out=15,in=165] (2,4);
\draw  [-]  (0,4) to [out=-15,in=-165] (2,4);
\draw [-] (0,4) to (0,2); 
\draw [-] (2,4) to (2,2);
\draw [-]  (0,2) to  (0,0); 
\draw [-]  (2,2) to  (2,0); 
\draw [-]  (0,2) to (2,2); 
\draw [-] (0,0) to [out=15,in=165] (2,0); 
\draw [-] (0,0) to [out=-15,in=-165] (2,0); 
\node [left] at (0,4) {\text{\tiny$-2$}}; \node [right] at (2,4) {\text{\tiny$-1$}}; \node [right] at (2,2) {\text{\tiny$-1$}}; \node [right] at (2,0) {\text{\tiny$-1$}}; 
\node [left] at (0,0) {\text{\tiny$-2$}}; \node [left] at (0,2) {\text{\tiny$-2$}}; 
\end{tikzpicture} 

\begin{tikzpicture}[scale=0.45] 
\foreach \x in {(0,0),(2,0),(0,2),(2,2),(0,4),(2,4)} \filldraw \x circle (2pt); 
\node at (1,3) {III.d};  
\draw  [-]  (0,4) to [out=15,in=165] (2,4);
\draw  [-]  (0,4) to [out=-15,in=-165] (2,4);
\draw [-] (0,4) to (0,2); 
\draw [-] (2,4) to (2,2);
\draw [-]  (0,2) to  (0,0); 
\draw [-]  (2,2) to  (2,0); 
\draw [-]  (0,2) to (2,2); 
\draw [-] (0,0) to [out=15,in=165] (2,0); 
\draw [-] (0,0) to [out=-15,in=-165] (2,0); 
\node [left] at (0,4) {\text{\tiny$-2$}}; \node [right] at (2,4) {\text{\tiny$-1$}}; \node [right] at (2,2) {\text{\tiny$-2$}}; \node [right] at (2,0) {\text{\tiny$-1$}}; 
\node [left] at (0,0) {\text{\tiny$-2$}}; \node [left] at (0,2) {\text{\tiny$-1$}}; 
\end{tikzpicture} 
\begin{tikzpicture}[scale=0.45] 
\foreach \x in {(0,0),(2,0),(0,2),(2,2),(0,4),(2,4)} \filldraw \x circle (2pt); 
\node at (1,3) {III.e};  
\draw  [-]  (0,4) to [out=15,in=165] (2,4);
\draw  [-]  (0,4) to [out=-15,in=-165] (2,4);
\draw [-] (0,4) to (0,2); 
\draw [-] (2,4) to (2,2);
\draw [-]  (0,2) to  (0,0); 
\draw [-]  (2,2) to  (2,0); 
\draw [-]  (0,2) to (2,2); 
\draw [-] (0,0) to [out=15,in=165] (2,0); 
\draw [-] (0,0) to [out=-15,in=-165] (2,0); 
\node [left] at (0,4) {\text{\tiny$-2$}}; \node [right] at (2,4) {\text{\tiny$-1$}}; \node [right] at (2,2) {\text{\tiny$-1$}}; \node [right] at (2,0) {\text{\tiny$-2$}}; 
\node [left] at (0,0) {\text{\tiny$-1$}}; \node [left] at (0,2) {\text{\tiny$-2$}}; 
\end{tikzpicture} 
\begin{tikzpicture}[scale=0.45] 
\foreach \x in {(0,0),(2,0),(0,2),(2,2),(0,4),(2,4)} \filldraw \x circle (2pt); 
\node at (1,3) {III.f};  
\draw  [-]  (0,4) to [out=15,in=165] (2,4);
\draw  [-]  (0,4) to [out=-15,in=-165] (2,4);
\draw [-] (0,4) to (0,2); 
\draw [-] (2,4) to (2,2);
\draw [-]  (0,2) to  (0,0); 
\draw [-]  (2,2) to  (2,0); 
\draw [-]  (0,2) to (2,2); 
\draw [-] (0,0) to [out=15,in=165] (2,0); 
\draw [-] (0,0) to [out=-15,in=-165] (2,0); 
\node [left] at (0,4) {\text{\tiny$-2$}}; \node [right] at (2,4) {\text{\tiny$-1$}}; \node [right] at (2,2) {\text{\tiny$-2$}}; \node [right] at (2,0) {\text{\tiny$-1$}}; 
\node [left] at (0,0) {\text{\tiny$-1$}}; \node [left] at (0,2) {\text{\tiny$-2$}}; 
\end{tikzpicture} 

\begin{tikzpicture}[scale=0.45]  
\node at (-1.25,2) {IV.a};
\foreach \x in {(0,0),(3,0),(1.5,1.25),(1.5,2.75),(0,4),(3,4)} \filldraw \x circle (2pt); 
\draw (0,0)--(0,4)--(3,4)--(3,0)--(0,0); \draw (0,0)--(1.5,1.25)--(3,0);\draw (0,4)--(1.5,2.75)--(3,4);
\draw (1.5,1.25)--(1.5,2.75); 
\node [left] at (0,0) {\text{\tiny$-1$}}; 
\node [left] at (0,4) {\text{\tiny$-2$}}; 
\node[right] at (3,0)  {\text{\tiny$-1$}};
\node[right] at (3,4)  {\text{\tiny$-2$}};
\node[below] at (1.5,1.25)  {\text{\tiny$-1$}};
\node[above] at (1.5,2.75)  {\text{\tiny$-2$}};
\end{tikzpicture}
\begin{tikzpicture}[scale=0.45]  
\node at (-1.25,2) {IV.b};
\foreach \x in {(0,0),(3,0),(1.5,1.25),(1.5,2.75),(0,4),(3,4)} \filldraw \x circle (2pt); 
\draw (0,0)--(0,4)--(3,4)--(3,0)--(0,0); \draw (0,0)--(1.5,1.25)--(3,0);\draw (0,4)--(1.5,2.75)--(3,4);
\draw (1.5,1.25)--(1.5,2.75); 
\node [left] at (0,0) {\text{\tiny$-1$}}; 
\node [left] at (0,4) {\text{\tiny$-1$}}; 
\node[right] at (3,0)  {\text{\tiny$-2$}};
\node[right] at (3,4)  {\text{\tiny$-2$}};
\node[below] at (1.5,1.25)  {\text{\tiny$-1$}};
\node[above] at (1.5,2.75)  {\text{\tiny$-2$}};
\end{tikzpicture}
\begin{tikzpicture}[scale=0.45]  
\node at (-1.25,2) {IV.c};
\foreach \x in {(0,0),(3,0),(1.5,1.25),(1.5,2.75),(0,4),(3,4)} \filldraw \x circle (2pt); 
\draw (0,0)--(0,4)--(3,4)--(3,0)--(0,0); \draw (0,0)--(1.5,1.25)--(3,0);\draw (0,4)--(1.5,2.75)--(3,4);
\draw (1.5,1.25)--(1.5,2.75); 
\node [left] at (0,0) {\text{\tiny$-2$}}; 
\node [left] at (0,4) {\text{\tiny$-1$}}; 
\node[right] at (3,0)  {\text{\tiny$-1$}};
\node[right] at (3,4)  {\text{\tiny$-2$}};
\node[below] at (1.5,1.25)  {\text{\tiny$-1$}};
\node[above] at (1.5,2.75)  {\text{\tiny$-2$}};
\end{tikzpicture}

\begin{tikzpicture}[scale=0.45]  
\node at (-1.25,2) {V.a};
\foreach \x in {(0,0),(2,0),(4,0),(0,2),(2,2),(4,2)} \filldraw \x circle (2pt); 
\draw (0,0)--(0,2)--(2,0)--(2,2)--(4,0)--(4,2)--(2,0); \draw (4,2)--(0,0)--(2,2);\draw (4,0)--(0,2);
\node [below] at (0,0) {\text{\tiny$-2$}}; 
\node [below] at (2,0) {\text{\tiny$-2$}}; 
\node[below] at (4,0)  {\text{\tiny$-1$}};
\node[above] at (0,2)  {\text{\tiny$-1$}};
\node[above] at (2,2)  {\text{\tiny$-1$}};
\node[above] at (4,2)  {\text{\tiny$-2$}};
\end{tikzpicture}\quad
\begin{tikzpicture}[scale=0.45]  
\node at (-1.25,1.5) {V.b};
\foreach \x in {(0,0),(2,0),(4,0),(0,2),(2,2),(4,2)} \filldraw \x circle (2pt); 
\draw (0,0)--(0,2)--(2,0)--(2,2)--(4,0)--(4,2)--(2,0); \draw (4,2)--(0,0)--(2,2);\draw (4,0)--(0,2);
\node [below] at (0,0) {\text{\tiny$-2$}}; 
\node [below] at (2,0) {\text{\tiny$-2$}}; 
\node[below] at (4,0)  {\text{\tiny$-2$}};
\node[above] at (0,2)  {\text{\tiny$-1$}};
\node[above] at (2,2)  {\text{\tiny$-1$}};
\node[above] at (4,2)  {\text{\tiny$-1$}};
\end{tikzpicture} 

Up to equivalence, the (quiver,weight) pairs in the above list yield six 
$4$-dimensional flow polytopes. For each such polytope $\nabla$ we present a tight (quiver,weight) pair 
$(Q,\theta)$ with $\nabla\cong \nabla(Q,\theta)$. Moreover, in the figures below to each arrow $a\in Q_1$ we assign a degree $1$ polynomial $L_a$ in the indeterminates $x,y,z,w$ such that 
\begin{equation*}\label{eq:xyzw}\mathcal{A}(Q,\theta)=\{(L_a(x,y,z,w))_{a\in Q_1}\in \mathbb{R}^{Q_1}\mid  x,y,z,w\in\mathbb{R}\}.\end{equation*} 
In other words, the polytope $\nabla$ as a subset of $\mathbb{R}^4\supset \mathbb{Z}^4$ is described as 
\[\nabla=\{(x,y,z,w)\in \mathbb{R}^4\mid \forall a\in Q_1:\ L_a(x,y,z,w)\ge 0\}\] 
(note that the above set of linear inequalities is minimal by tightness of $(Q,\theta)$).  
Furthermore, we list all lattice points of $\nabla\subset \mathbb{R}^4$. 

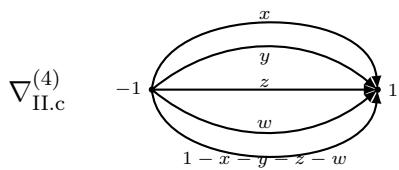
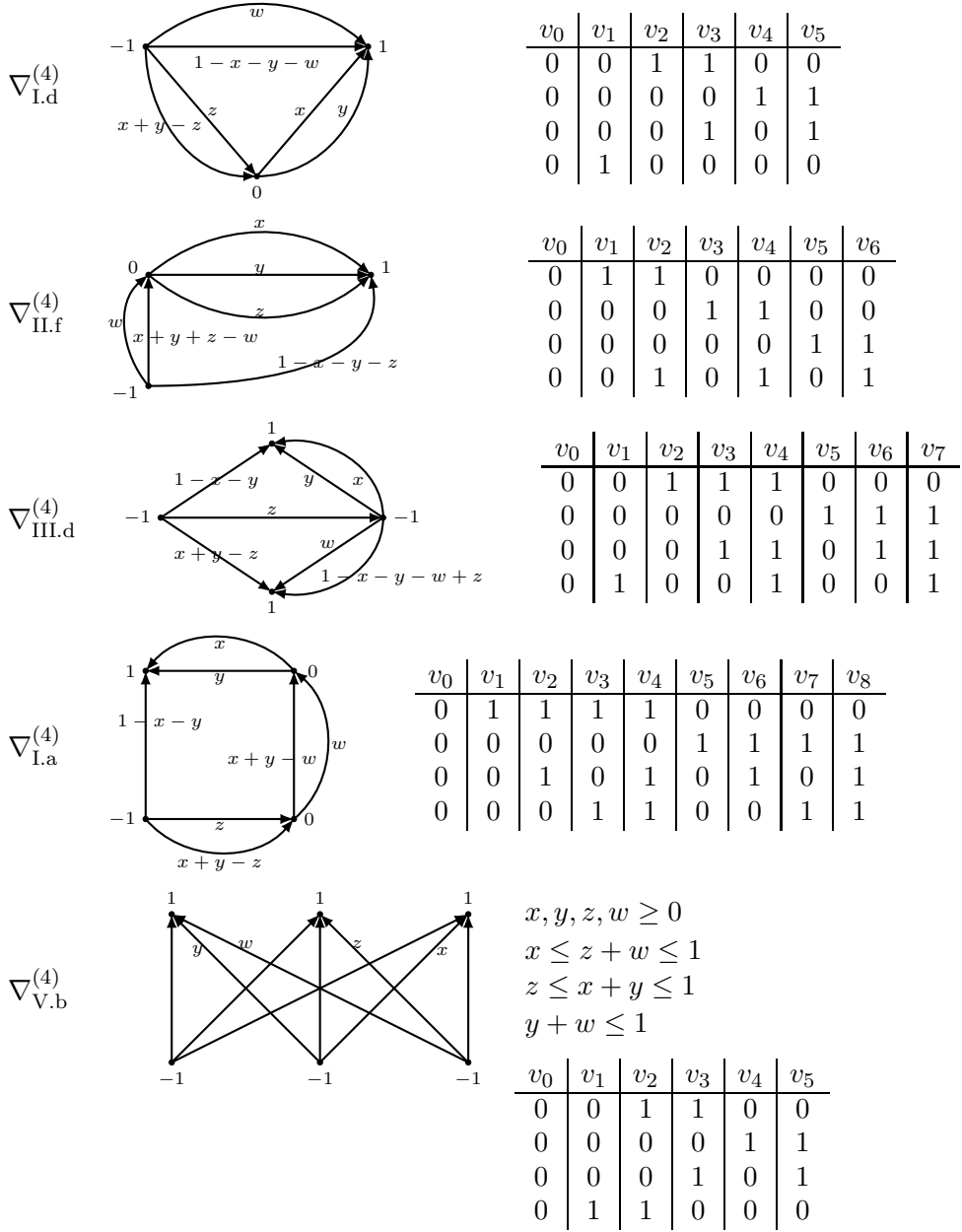
\begin{tikzpicture}[>=latex]
\foreach \x in {(0,0),(3,0)} \filldraw \x circle (1pt); 
\node [left] at (-1,0) {$\nabla_{\mathrm{II.c}}^{(4)}$};  
\draw  [->, thick]  (0,0) to [out=85,in=95] (3,0);
\node [above] at (1.5,0.8) {\text{\tiny$x$}}; 
\draw  [->, thick]  (0,0) to [out=40,in=140] (3,0); 
\node [above] at (1.5,0.2) {\text{\tiny$y$}};
\draw [->, thick] (0,0) to (3,0); 
\node [above] at (1.5,-0.1)   {\text{\tiny$z$}}; 
\draw  [->, thick]  (0,0) to [out=-40,in=-140] (3,0);
\node [above] at (1.5,-0.65) {\text{\tiny$w$}}; 
\draw  [->, thick]  (0,0) to [out=-85,in=-95] (3,0);
\node [below] at (1.5,-0.7) {\text{\tiny$1-x-y-z-w$}};
\node [left] at (0,0) {\text{\tiny{$-1$}}};
\node [right] at (3,0) {\text{\tiny{$1$}}};
\node [right] at (5,0) {$\begin{array}{c|c|c|c|c}
v_0 & v_1 & v_2 & v_3 & v_4\\
\hline 
0&1&0&0&0\\ 0&0&1&0&0\\ 0&0&0&1&0\\ 0&0&0&0&1
\end{array}$}; 
\end{tikzpicture}
 
\begin{tikzpicture}[>=latex]
\foreach \x in {(0,0),(1.5,-1.75),(3,0)} \filldraw \x circle (1pt); 
\node [left] at (-1,-0.5) {$\nabla_{\mathrm{I.d}}^{(4)}$};  
\draw  [->, thick]  (0,0) to  (3,0); \node [below] at (1.5,0) {\text{\tiny$1-x-y-w$}};
\draw  [->, thick]  (0,0) to [out=40,in=140] (3,0); \node [below] at (1.5,0.65) {\text{\tiny$w$}}; 
\draw [->, thick] (0,0) to (1.5,-1.75); \node at (0.9,-0.87) {\text{\tiny$z$}}; 
\draw  [->, thick]  (0,0) to [out=-90,in=180] (1.5,-1.75); \node [left] at (0.9,-1.07) {\text{\tiny$x+y-z$}};
\draw  [->, thick]  (1.5,-1.75) to  (3,0); \node [left] at (2.3,-0.87)   {\text{\tiny$x$}}; 
\draw  [->, thick]  (1.5,-1.75) to [out=0,in=-90] (3,0); \node [right] at (2.45,-0.87) {\text{\tiny$y$}};
\node [left] at (0,0) {\text{\tiny{$-1$}}}; \node [below] at (1.5,-1.75) {\text{\tiny{$0$}}}; \node [right] at (3,0) {\text{\tiny{$1$}}};
\node [right] at (5,-0.7) {$\begin{array}{c|c|c|c|c|c}
v_0 & v_1 & v_2 & v_3 & v_4 & v_5 \\
\hline 
0&0&1&1&0&0\\ 0&0&0&0&1&1\\ 0&0&0&1&0&1 \\ 0&1&0&0&0&0
\end{array}$}; 
\end{tikzpicture}

\begin{tikzpicture}[>=latex]
\foreach \x in {(0,0),(0,1.5),(3,1.5)} \filldraw \x circle (1pt); 
\node [left] at (-1,1) {$\nabla_{\mathrm{II.f}}^{(4)}$};  
\draw  [->, thick]  (0,0) to  (0,1.5); \node [right] at (-0.35,0.65) {\text{\tiny$x+y+z-w$}};
\draw  [->, thick]  (0,0) to [out=130,in=-140] (0,1.5); \node [left] at (-0.2,0.85) {\text{\tiny$w$}}; 
\draw [->, thick] (0,1.5) to (3,1.5); \node [above] at (1.5,2) {\text{\tiny$x$}}; 
\draw  [->, thick]  (0,1.5) to [out=40,in=140] (3,1.5); \node [above] at (1.5,1.35) {\text{\tiny$y$}};
\draw  [->, thick]  (0,1.5) to [out=-40,in=-140] (3,1.5); \node [above] at (1.5,0.8)   {\text{\tiny$z$}}; 
\draw  [->, thick]  (0,0) to [out=0,in=-80] (3,1.5); \node [right] at (1.6,0.3) {\text{\tiny$1-x-y-z$}};
\node [left] at (0,-0.1) {\text{\tiny{$-1$}}}; \node [left] at (0,1.6) {\text{\tiny{$0$}}}; \node [right] at (3,1.6) {\text{\tiny{$1$}}};
\node [right] at (5,1) {$\begin{array}{c|c|c|c|c|c|c}
v_0 & v_1 & v_2 & v_3 & v_4 & v_5 & v_6\\
\hline 
0&1&1&0&0&0&0\\ 0&0&0&1&1&0&0\\ 0&0&0&0&0&1&1\\ 0&0&1&0&1&0&1
\end{array}$}; 
\end{tikzpicture}
 
\begin{tikzpicture}[>=latex]
\foreach \x in {(0,0),(1.5,1),(3,0),(1.5,-1)} \filldraw \x circle (1pt); 
\node [left] at (-1,0) {$\nabla_{\mathrm{III.d}}^{(4)}$};  
\draw  [->, thick]  (0,0) to  (3,0); \node at (1.5,0.1) {\text{\tiny$z$}}; 
\draw [->, thick]  (0,0) to (1.5,1); \node at (0.75,0.5) {\text{\tiny$1-x-y$}};  
\draw [->, thick]  (0,0) to (1.5,-1); \node at (0.75,-0.5) {\text{\tiny$x+y-z$}};  
\draw [->, thick] (3,0) to (1.5,1); \node [right] at (2.45,0.5) {\text{\tiny$x$}}; 
\draw  [->, thick]  (3,0) to [out=90,in=10] (1.5,1); \node  at (2,0.5) {\text{\tiny$y$}};
\draw  [->, thick]  (3,0) to (1.5,-1); \node at (2.25,-0.35) {\text{\tiny$w$}}; 
\draw  [->, thick]  (3,0) to [out=-90,in=-10] (1.5,-1); \node at (3.25,-0.8)  {\text{\tiny$1-x-y-w+z$}}; 
\node [left] at (0,0) {\text{\tiny{$-1$}}}; \node [above] at (1.5,1) {\text{\tiny{$1$}}}; 
\node [below] at (1.5,-1) {\text{\tiny{$1$}}};  \node [right] at (3,0) {\text{\tiny{$-1$}}}; 
\node [right] at (5,0) {$\begin{array}{c|c|c|c|c|c|c|c}
v_0 & v_1 & v_2 & v_3 & v_4 & v_5 & v_6 & v_7\\
\hline 
0&0&1&1&1&0&0&0\\ 0&0&0&0&0&1&1&1\\ 0&0&0&1&1&0&1&1\\ 0&1&0&0&1&0&0&1
\end{array}$}; 
\end{tikzpicture}
 
\begin{tikzpicture}[>=latex]
\foreach \x in {(0,0),(2,0),(0,2),(2,2)} \filldraw \x circle (1pt); 
\node [left] at (-1,1) {$\nabla_{\mathrm{I.a}}^{(4)}$};  
\draw  [->, thick]  (0,0) to  (0,2); \node [left] at (0.9,1.3) {\text{\tiny$1-x-y$}}; 
\draw [->, thick]  (0,0) to (2,0); \node at (1,-0.1) {\text{\tiny$z$}};  
\draw [->, thick]  (0,0) to [out=-50,in=-130] (2,0); \node at (1,-0.6) {\text{\tiny$x+y-z$}};  
\draw [->, thick] (2,0) to (2,2); \node [left] at (2.45,0.8) {\text{\tiny$x+y-w$}}; 
\draw  [->, thick]  (2,0) to  [out=40,in=-40] (2,2); \node [left]  at (2.85,1) {\text{\tiny$w$}};
\draw  [->, thick]  (2,2) to (0,2); \node [below] at (1,2.1) {\text{\tiny$y$}}; 
\draw  [->, thick]  (2,2) to [out=130,in=50] (0,2); \node at (1,2.35)  {\text{\tiny$x$}}; 
\node [left] at (0,0) {\text{\tiny{$-1$}}}; \node [left] at (0,2) {\text{\tiny{$1$}}}; 
\node [right] at (2,0) {\text{\tiny{$0$}}};  \node [right] at (2.05,2) {\text{\tiny{$0$}}}; 
\node [right] at (3.5,1) {$\begin{array}{c|c|c|c|c|c|c|c|c}
v_0 & v_1 & v_2 & v_3 & v_4 & v_5 & v_6 & v_7 & v_8\\
\hline 
0&1&1&1&1&0&0&0&0\\ 0&0&0&0&0&1&1&1&1\\ 0&0&1&0&1&0&1&0&1\\ 0&0&0&1&1&0&0&1&1
\end{array}$}; 
\end{tikzpicture}
 
\begin{tikzpicture}[>=latex];
\node at (-1.8,1) {$\nabla_{\mathrm{V.b}}^{(4)}$};
\foreach \x in {(0,0),(2,0),(4,0),(0,2),(2,2),(4,2)} \filldraw \x circle (1pt); 
\draw [->, thick] (0,0) to (0,2); 
\draw [->, thick] (0,0) to (2,2); 
\draw [->, thick] (0,0) to (4,2); 
\draw [->, thick] (2,0) to (0,2); \node at (0.35,1.5) {\text{\tiny$y$}};
\draw [->, thick] (2,0) to (2,2); 
\draw [->, thick] (2,0) to (4,2);  \node at (3.65,1.5)  {\text{\tiny$x$}};
\draw [->, thick] (4,0) to (0,2);  \node at (1,1.6)  {\text{\tiny$w$}};
\draw [->, thick] (4,0) to (2,2);  \node at (2.5,1.6)  {\text{\tiny$z$}};
\draw [->, thick] (4,0) to (4,2); 
\node [below] at (0,0) {\text{\tiny$-1$}}; 
\node [below] at (2,0) {\text{\tiny$-1$}}; 
\node[below] at (4,0)  {\text{\tiny$-1$}};
\node[above] at (0,2)  {\text{\tiny$1$}};
\node[above] at (2,2)  {\text{\tiny$1$}};
\node[above] at (4,2)  {\text{\tiny$1$}};
\node [right] at (4.5,2) {\ $x,y,z,w\ge 0$}; 
\node [right] at (4.5,1.5){\  $x\le z+w\le 1$};  
\node [right] at (4.5,1) {\ $z\le x+y\le 1$}; 
\node [right] at (4.5,0.5) {\  $y+w\le 1$};
\node [right] at (4.5,-1.1) {$\begin{array}{c|c|c|c|c|c}
v_0 & v_1 & v_2 & v_3 & v_4 & v_5  \\
\hline 
0&0&1&1&0&0\\ 0&0&0&0&1&1\\ 0&0&0&1&0&1\\ 0&1&1&0&0&0
\end{array}$}; 
\end{tikzpicture} 

\begin{remark} 
The polytope $\nabla_{\mathrm{V.b}}^{(4)}$ is the $4$-dimensional \emph{Birkhoff polytope}. 
\end{remark}

\begin{proposition}\label{prop:4-cells}
\begin{itemize} \item[(i)] $\mathcal{P}(\Gamma_{\mathrm{I}})=
\{\nabla_{\mathrm{I.a}}^{(4)}, \quad \nabla_{\mathrm{I.d}}^{(4)}\}$ 
 \item[(ii)] $\mathcal{P}(\Gamma_{\mathrm{II}})=
\{\nabla_{\mathrm{I.d}}^{(4)}, \quad \nabla_{\mathrm{II.c}}^{(4)}, \quad \nabla_{\mathrm{II.f}}^{(4)}\}$ 
 \item[(iii)] $\mathcal{P}(\Gamma_{\mathrm{III}})=
\{\nabla_{\mathrm{II.c}}^{(4)}, \quad \nabla_{\mathrm{II.f}}^{(4)}, \quad 
\nabla_{\mathrm{III.d}}^{(4)}\}$
\item[(iv)] $\mathcal{P}(\Gamma_{\mathrm{IV}})=
\{\nabla_{\mathrm{I.d}}^{(4)}, \quad \nabla_{\mathrm{II.c}}^{(4)}\}$
\item[(v)] $\mathcal{P}(\Gamma_{\mathrm{V}})=
\{\nabla_{\mathrm{II.c}}^{(4)},\quad \nabla_{\mathrm{V.b}}^{(4)}\}$
\end{itemize}
\end{proposition}

\begin{proof}
We go through the graphs $\Gamma\in \{\Gamma_{\mathrm{I}},\Gamma_{\mathrm{II}},\Gamma_{\mathrm{III}},
\Gamma_{\mathrm{IV}},\Gamma_{\mathrm{V}}\}  \in \mathcal{L}_4^{\mathrm{3-reg}}$ and the weights in 
$\mathcal{W}(\Gamma)$  listed at the beginning of Section~\ref{sec:4-cells}, and we tighten each such pair 
$(\Gamma^*,\theta)$. This can be done as follows. 
We first apply Lemma~\ref{lemma:reduction} a couple of times. Thus we get rid of (the majority of) the valency $2$ 
sink vertices, and get a (quiver,weight) pair $(Q,\sigma)$.  Then we find a spanning tree in $Q$, and take the coordinates corresponding to the arrows in the complement of the spanning tree as free coordinates $x,y,z,w$ in the affine space $\mathcal{A}(Q,\sigma)$. Going through the arrows of $Q$ in an appropriate order,  for each arrow $a\in Q_1$ we can write down  a  degree $1$ polynomial $L_a(x,y,z,w)$ such that 
$\mathcal{A}(Q,\sigma)=\{(L_a(x,y,z,w))_{a\in Q_1}\in \mathbb{R}^{Q_1}\mid  x,y,z,w\in\mathbb{R}\}$ 
(the degree one polynomials corresponding to the arrows in the complement of the chosen spanning tree are $x,y,z,w$). 
Then  we have 
\[\nabla(Q,\sigma)\cong \{(x,y,z,w)\in \mathbb{R}^4\mid \forall a\in Q_1:\ L_a(x,y,z,w)\ge 0\}\subset \mathbb{R}^4.\] 
By inspection of the system $L_a\ge 0$ ($a\in Q_1$) of inequalities it is a 
simple matter to sort out the redundant inequalities (in other words, to find the contractible arrows), or to detect if the polytope is contained in a proper affine hyperplane of $\mathbb{R}^4$. In the latter case the pair $(\Gamma^*,\theta)$ we started with does not yield a $4$-dimensional polytope.  Otherwise we contract the contractible arrows, and   
end up with a tight pair $(Q',\sigma')$. Moreover, the lattice points in 
$\nabla(Q',\sigma')$ (viewed as a subset of $\mathbb{R}^4$) are just those $4$-dimensional $0-1$ vectors, whose coordinates satisfy the inequalities $L_a\ge 0$ ($a\in Q'_1$). 
It turns out that this process gives us $6$ tight pairs 
$(Q',\sigma')$ with $\chi(Q')=4$, 
yielding the polytopes 
$\nabla_{\mathrm{II.c}}^{(4)}$, $\nabla_{\mathrm{I.d}}^{(4)}$, $\nabla_{\mathrm{II.f}}^{(4)}$, 
$\nabla_{\mathrm{III.d}}^{(4)}$, $\nabla_{\mathrm{I.a}}^{(4)}$, $\nabla_{\mathrm{V.b}}^{(4)}$. 
In fact we get 
\begin{itemize} 
\item[] $\nabla_{\mathrm{II.c}}^{(4)}$ from 
$\mathrm{II.c}$, $\mathrm{III.c}$, 
$\mathrm{IV.b}$, $\mathrm{V.a}$; 
\item[] $\nabla_{\mathrm{I.d}}^{(4)}$ from 
$\mathrm{I.d}$, $\mathrm{II.d}$, 
$\mathrm{II.e}$, $\mathrm{IV.c}$;
\item[] $\nabla_{\mathrm{II.f}}^{(4)}$ from $\mathrm{II.f}$, 
$\mathrm{III.e}$;
\item[] $\nabla_{\mathrm{III.d}}^{(4)}$ from 
$\mathrm{III.d}$; 
\item[] $\nabla_{\mathrm{I.a}}^{(4)}$ from 
$\mathrm{I.a}$; 
\item[] $\nabla_{\mathrm{V.b}}^{(4)}$ from 
$\mathrm{V.b}$.  
\end{itemize} 
These $6$ polytopes are pairwise inequivalent, because 
the pair  
(number of facets, number of lattice points) distinguishes between them 
(recall that the number of facets of $\nabla(Q,\theta)$ equals the number of arrows in $Q$ when $(Q,\theta)$ is tight).  
We get at most $3$-dimensional polytopes from the remaining  (quiver,weight) pairs labeled by 
$\mathrm{I.b}$, $\mathrm{I.c}$, $\mathrm{II.a}$, 
$\mathrm{II.b}$, $\mathrm{II.g}$, $\mathrm{II.h}$, 
$\mathrm{III.a}$, $\mathrm{III.b}$, $\mathrm{III.f}$, 
$\mathrm{IV.a}$. 

We omit the straightforward and lengthy details of the above process. 
An illustration  on a somewhat smaller example can be found in the proof of Proposition~\ref{prop:all 3-dim compressed} 
(see the analysis of the (quiver,weight) pair labeled by II.c in that proof).  
\end{proof}


\section{Triangulations} \label{sec:triangulation}

We refer to \cite{haase-etal} and \cite{lorea-rambau-santos} for basic notions  about triangulations of lattice polytopes. 
In particular, a \emph{quadratic} triangulation of a lattice polytope is  a regular unimodular triangulation whose  minimal non-faces have two elements. 

In Proposition~\ref{prop:triangulation of 4-cells} below we shall use the following notation: suppose that $v_0,v_1,\dots,v_m$ are the lattice points in a lattice polytope 
$\nabla$. For a subset $\{i_1,\dots,i_k\}\subseteq \{0,1,\dots,m\}$ we write 
$\langle i_1,\dots,i_k\rangle$ for the convex hull of $\{v_{i_1},\dots,v_{i_k}\}$. 
By a \emph{non-face} of a triangulation of a polytope we mean a set of lattice points contained in no face of the triangulation; 
it is a \emph{minimal non-face} if all its proper subsets are contained in some face of the triangulation. 

\begin{proposition} \label{prop:triangulation of 4-cells} 
Each of the polytopes $\nabla_{\mathrm{II.c}}^{(4)}$,  $\nabla_{\mathrm{I.d}}^{(4)}$, $\nabla_{\mathrm{II.f}}^{(4)}$, 
$\nabla_{\mathrm{III.d}}^{(4)}$, $\nabla_{\mathrm{I.a}}^{(4)}$ and the prime compressed $3$-dimensional flow polytopes has a  pulling triangulation whose minimal non-faces have two elements. 
The polytope $\nabla_{\mathrm{V.b}}^{(4)}$ has a pulling triangulation whose minimal non-faces 
have three elements. More concretely, 
$\nabla_{\mathrm{I.a}}^{(3)}$ and $\nabla_{\mathrm{II.c}}^{(4)}$ are unimodular simplices of dimension $3$ and $4$, respectively, 
whereas 
\begin{itemize} 
\item[(i)] pulling at $v_0$ yields the triangulation of $\nabla_{\mathrm{II.c}}^{(3)}$  with maximal cells 
\[\langle 0,2,3,4\rangle,\ \langle 0,1,3,4\rangle \text{ and one minimal non-face }\{v_1,v_2\}.\] 
\item[(ii)] pulling at $v_0$ yields the triangulation of $\nabla_{\mathrm{I.b}}^{(3)}$ with  
maximal cells  
\[\langle 0,1,2,5\rangle, \langle 0,3,4,5\rangle, \langle 0,1,3,5\rangle, \langle 0,2,4,5 \rangle\] 
and  minimal non-faces  

\[\{v_1,v_4\},\ \{v_2,v_3\}.\] 
\item[(iii)] pulling at $v_2$ yields the  triangulation of $\nabla_{\mathrm{I.d}}^{(4)}$ with maximal cells 
\[\langle 0,1,2,3,5\rangle,\ \langle 0,1,2,4,5\rangle \text{ and  one minimal non-face }\{v_3,v_4\}.\] 
\item[(iv)] pulling at $v_2$ and then at $v_3$ yields the  triangulation of $\nabla_{\mathrm{II.f}}^{(4)}$ 
with maximal cells 
\[\langle 0,1,2,3,5\rangle,\  \langle 0,2,3,5,6\rangle, \ \langle 0,2,3,4,6\rangle\] 
and minimal non-faces 
\[\{v_1,v_4\},\ \{v_1,v_6\},\ \{v_4,v_5\}.\]
\item[(v)] pulling at $v_3$, $v_0$ and then at $v_4$ yields the  triangulation of $\nabla_{\mathrm{III.d}}^{(4)}$
with maximal cells 
\[\langle 0,1,2,3,5\rangle, \langle 0,3,5,6,7\rangle, \langle 0,1,3,5,7\rangle, \langle 1,2,3,4,5\rangle, 
\langle 1,3,4,5,7\rangle\]
and minimal non-faces  
\[\{v_0,v_4\},\ \{v_1,v_6\}, \ \{v_2,v_6\},\ \{v_2,v_7\},\ \{v_4,v_6\}.\] 
\item[(vi)] pulling at $v_4$, $v_6$ and then at $v_5$ yields the  triangulation of $\nabla_{\mathrm{I.a}}^{(4)}$ 
with maximal cells 
\[\langle 0,4,5,6,7\rangle, \langle 0,4,6,7,8\rangle,\langle 0,1,2,4,6\rangle, \langle 0,1,4,5,6\rangle, 
\langle 0,1,3,4,5\rangle, \langle 0,3,4,5,7\rangle\] 
and minimal non-faces 
\[\{v_1,v_7\},\{v_1,v_8\},\{v_2,v_3\},\{v_2,v_5\},\{v_2,v_7\},\{v_2,v_8\},\{v_3,v_6\},\{v_3,v_8\},\{v_5,v_8\}.\]
\item[(vii)] pulling at $v_0$ yields the  triangulation of $\nabla_{\mathrm{II.f}}^{(4)}$ 
with maximal cells 
\[\langle 0,2,3,4,5\rangle, \langle 0,1,2,3,5\rangle,\langle 0,1,2,4,5 \rangle 
\text{ and one minimal non-face }\{v_1,v_3,v_4\}.\] 
\end{itemize}
\end{proposition}

\begin{proof}
The list of lattice points and the equations for the facets of the polytopes in question were given in 
Proposition~\ref{prop:all 3-dim compressed} and before 
Proposition~\ref{prop:4-cells}. 
Using this data it is easy to come up with the desired pulling triangulations. 
We give the details for $\nabla_{\mathrm{II.f}}^{(4)}$, the other cases are similar. 
The facets  of  $\nabla_{\mathrm{II.f}}^{(4)}$ are 
\[\langle 0,3,4,5,6\rangle, \langle 0,1,2,5,6\rangle, \langle 0,1,2,3,4\rangle, \langle 0,1,3,5\rangle, 
\langle 1,2,3,4,5,6 \rangle, \langle 0,2,4,6\rangle.\]    
Pulling at $v_2$ we get the subdivision whose $4$-dimensional cells are 
\[\langle 0,1,2,3,5\rangle, \qquad \langle 0,2,3,4,5,6\rangle.\] 
The first cell above is a unimodular simplex, the second is a pyramid 
with apex $v_2$ over the facet $F=\langle 0,3,4,5,6\rangle$ of $\nabla_{\mathrm{II.f}}^{(4)}$. 
To continue we need to determine the facets of this latter pyramid: apart from $F$   
these are the convex hulls of $v_0$ and the facets of $F$. 
The facets of $F$ are 
$\langle 0,5,6\rangle$, $\langle 0,3,4\rangle$, $\langle 0,3,5\rangle$, $\langle 0,4,6\rangle$, $\langle 3,5,6\rangle$. 
Thus the facets of $\langle 0,2,3,4,5,6\rangle$ are 
\[\langle 0,3,4,5,6\rangle, \langle 2,0,5,6\rangle, \langle 2,0,3,4\rangle, \langle 2,0,3,5\rangle, \langle 2,0,4,6\rangle, \langle 2,3,5,6\rangle.\]  
Pulling at $v_3$ we get the full triangulation whose $4$-dimensional simplices are 
\[\langle 0,1,2,3,5\rangle, \langle 0,2,3,5,6\rangle, \langle 0,2,3,4,6\rangle.\]
It is an easy matter to check that the minimal non-faces of the above triangulation are 
$\{v_1,v_4\}$, $\{v_1,v_6\}$, $\{v_4,v_5\}$.  
\end{proof} 

The key result of this section is that, up to equivalence,  $\nabla_{\mathrm{V.b}}^{(4)}$ is the only flow polytope up to dimension $4$ that does not have a quadratic triangulation. In order to show this we first need to understand the set of weights on the quiver $\Gamma_{\mathrm{V}}^*$ that yield polytopes equivalent to  $\nabla_{\mathrm{V.b}}^{(4)}$. 
For a quiver $Q$ and $\underline{k} \in \mathbb{Z}^{Q_1}$ we will denote by 
$\omega_{\underline{k}}\in \mathbb{Z}^{Q_0}$ the weight given by $$\omega_{\underline{k}}(v)=\sum_{a^+=v}\underline{k}(a)-\sum_{a^-=v}\underline{k}(a).$$

\begin{lemma}\label{lemma:translated_quivers}
For a quiver $Q$ and $\underline{k} \in \mathbb{N}_0^{Q_1}$ and a weight $\theta \in \mathbb{Z}^{Q_0}$  we have that, $\nabla(Q, \theta)$ and  $\nabla(Q, \theta + \omega_{\underline{k}})$ are equivalent if and only if  for each $x \in  \nabla(Q, \theta + \omega_{\underline{k}})$ we have $x(a) \geq \underline{k}(a)$ for all $a \in Q_1$.
\end{lemma}
\begin{proof}
It is straightforward to see that for the translation map $\varphi: \mathbb{R}^{Q_1} \rightarrow \mathbb{R}^{Q_1}$ defined as $\varphi(x) = x + \underline{k}$ we have that,
$$\varphi(\nabla(Q, \theta)) =   \{ x \in \nabla(Q, \theta + \omega_{\underline{k}}) \mid \forall a \in Q_1 : x(a) \geq k(a) \} \subseteq \nabla(Q, \theta + \omega_{\underline{k}}),$$
proving the claim.
\end{proof}

Let us denote the arrows and vertices of $\Gamma_{\mathrm{V}}^*$ as in the figure below.
$$
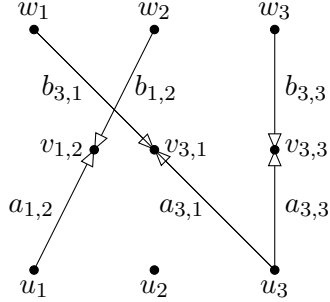
\begin{tikzpicture}[>=open triangle 45,scale=0.8]   
\foreach \x in {(0,0),(2,0),(4,0),(0,4),(2,4),(4,4),(1,2),(4,2),(2,2)} \filldraw \x circle (2pt); 
\draw [->] (0,0)--(1,2); \draw [->] (2,4)--(1,2); \draw (4,0)--(0,4); \draw [->] (4,0)--(4,2); \draw [->] (4,4)--(4,2);
\draw[->] (4,0)--(2,2); \draw [->] (0,4)--(2,2); 
\node[right] at (2,2) {$v_{3,1}$}; \node[left] at (3,1) {$a_{3,1}$}; \node[left] at (1,3) {$b_{3,1}$};
\node[left] at (0.5,1) {$a_{1,2}$}; \node[right] at (1.5,3) {$b_{1,2}$}; 
\node[right] at (4,1) {$a_{3,3}$}; \node[right] at (4,3) {$b_{3,3}$}; 
\node[left] at (1,2) {$v_{1,2}$};
\node[above] at (0,4) {$w_1$};
\node[above] at (2,4) {$w_2$};
\node[above] at (4,4) {$w_3$};
\node[right] at (4,2) {$v_{3,3}$};
\node[below] at (0,0) {$u_1$};
\node[below] at (2,0) {$u_2$};
\node[below] at (4,0) {$u_3$};
\end{tikzpicture}$$ 
Let $\theta_{\mathrm{V.b}}$ denote the weight defined as $\theta_{\mathrm{V.b}}(u_i) = -1$, $\theta_{\mathrm{V.b}}(w_i) =  -2$ and $\theta_{\mathrm{V.b}}(v_{i,j}) = 1$, for all $i,j: 1 \leq i,j \leq 3$.

\begin{lemma}\label{lemma:grobner_weights}
For an integer weight $\theta$, the polytope $\nabla(\Gamma_{\mathrm{V}}^*,\theta)$ is equivalent to the Birkhoff polytope $\nabla_{\mathrm{V.b}}^{(4)}$ if and only if - after possibly reversing the role of the $u_i$ and the $w_i$ - we have that $$\theta = \theta_{\mathrm{V.b}} + \omega_{\underline{k}}$$ for some $\underline{k}\in \mathbb{N}_0^{{(\Gamma_{\mathrm{V}}^*)}_1}$ satisfying $\underline{k}(a_{i,j}) = 0$ for for all $i,j: 1 \leq i,j \leq 3$. 
In particular, $\nabla(\Gamma_{\mathrm{V}}^*,\theta)\cong \nabla_{\mathrm{V.b}}^{(4)}$ implies 
$\theta(u_1)=\theta(u_2)=\theta(u_3)=-1$ or $\theta(w_1)=\theta(w_2)=\theta(w_3)=-1$. 
\end{lemma}
\begin{proof}
First let us assume that $\theta= \theta_{\mathrm{V.b}} + \omega_{\underline{k}}$ for some $\underline{k}\in \mathbb{N}_0^{{(\Gamma_{\mathrm{V}}^*)}_1}$ satisfying $\underline{k}(a_{i,j}) = 0$ for for all $i,j: 1 \leq i,j \leq 3$. 
It follows from $\theta(u_i)=-1$ that for any $x \in \nabla(\Gamma_{\mathrm{V}}^*,\theta)$ we have $x(a_{i,j}) \leq 1$ and $x(b_{i,j}) \geq \theta(v_{i,j}) - 1 = \underline{k}(b_{i,j})$ for all $1 \leq i, j \leq 3$. Let $\underline{k}\in \mathbb{N}_0^{{(\Gamma_{\mathrm{V}}^*)}_1}$ be defined as $\underline{k}(a_{i,j}) = 0$ and $\underline{k}(b_{i,j}) = \theta(v_{i,j}) - 1$. Since $\underline{k}(a_{i,j}) = 0$ we have that $x \geq \underline{k}$ and then by Lemma~\ref{lemma:translated_quivers}: 
$$\nabla(\Gamma_{\mathrm{V}}^*,\theta) \cong \nabla(\Gamma_{\mathrm{V}}^*,\theta_{\mathrm{V.b}}).$$ 

For the other direction let $\theta$ be an integer weight such that $\nabla(\Gamma_{\mathrm{V}}^*,\theta) \cong \nabla_{\mathrm{V.b}}^{(4)}$. Since the Birkhoff polytope has 9 facets, none of which are parallel to an other, we have that for each $i,j$ one of the faces 
\[\{x\in \nabla(\Gamma_{\mathrm{V}}^*,\theta)\mid x(a_{i,j})=0\} \text{ or }\{x\in \nabla(\Gamma_{\mathrm{V}}^*,\theta)\mid x(b_{i,j})=0\}\]  
is a facet. Since for $\nabla(\Gamma_{\mathrm{V}}^*,\theta)$ we have $x(a_{i,j}) = \theta(v_{i,j}) - x(b_{i,j})$ and the Birkhoff polytope is compressed we have that for some non-negative integer vector $\underline{k}$: 
\begin{equation}\label{eq:23/1} \nabla(\Gamma_{\mathrm{V}}^*,\theta) = \nabla(\Gamma_{\mathrm{V}}^*,\theta,{\underline{k}},{\underline{k}}+{\underline{1}})\cong \nabla(\Gamma_{\mathrm{V}}^*,\theta- \omega_{\underline{k}},{\underline{0}},{\underline{1}}).\end{equation} 
By the first equality in \eqref{eq:23/1} and Lemma~\ref{lemma:translated_quivers} we have 
\begin{equation}\label{eq:23/2} 
 \nabla(\Gamma_{\mathrm{V}}^*,\theta)\cong \nabla(\Gamma_{\mathrm{V}}^*,\theta - \omega_{\underline{k}}). 
\end{equation}
From \eqref{eq:23/1} and \eqref{eq:23/2} we conclude 
$$\nabla(\Gamma_{\mathrm{V}}^*,\theta - \omega_{\underline{k}}) = \nabla(\Gamma^*,\theta - \omega_{\underline{k}},{\underline{0}},{\underline{1}}).$$
Arguing the same way as in the proof of Proposition~\ref{prop:P(Gamma)} we conclude that $\theta - \omega_{\underline{k}} \in \mathcal{W}(\Gamma_{\mathrm{V}})$, meaning that $\theta-\omega_{\underline{k}}$ takes value $1$ on each sink and $-1$ on $3$ of the sources and $-2$ on the remaining $3$ sources. By the assumption that $\nabla(\Gamma_{\mathrm{V}}^*,\theta)$ is the Birkhoff polytope and by the proof of Proposition~\ref{prop:4-cells} 
(i.e. the quiver $\Gamma_{\mathrm{V}}^*$ with the weight labeled by $\mathrm{V.a}$ in Section~\ref{sec:4-cells} yields the polytope $\nabla_{\mathrm{II.c}}^{(4)}$, which is not equivalent to the Birkhoff polytope), we conclude that $\theta - \omega_{\underline{k}}= \theta_{\mathrm{V.b}}$ after possibly reversing the two sides of the bipartition. 
So by symmetry  we can assume that 
$$(\theta - \omega_{\underline{k}})(u_1) = (\theta - \omega_{\underline{k}})(u_2) = (\theta - \omega_{\underline{k}})(u_3) = -1$$
and
$$(\theta - \omega_{\underline{k}})(w_1) = (\theta - \omega_{\underline{k}})(w_2) = (\theta - \omega_{\underline{k}})(w_3) = -2.$$
We need to show that ${\underline{k}}(a_{i, j}) = 0$ for for all $i,j: 1 \leq i,j \leq 3$. 
Since $\theta_{\mathrm{V.b}}$ is invariant under permuting the vertices on each side of the bipartition, it is sufficient to show that $\underline{k}(a_{1,1}) = 0$.
Suppose for contradiction that ${\underline{k}}(a_{1, 1}) \ge 1$. 
 
Let $x\in \mathcal{A}(\Gamma_{\mathrm{V}}^*,\theta_{\mathrm{V.b}})$ be defined as follows:
$$x(a_{1,1}) = -1 \quad x(a_{1,2}) = x(a_{1,3}) = x(a_{2,1}) = x(a_{3,1}) = 1$$
$$x(a_{2,2}) = x(a_{2,3}) = x(a_{3,2}) = x(a_{3,3}) = 0$$
$$x(b_{i,j}) = 1 - x(a_{i,j}) \quad \forall i,j: 1 \leq i,j \leq 3.$$
We then have 
$y:=x+\underline{k}\in \nabla(\Gamma_{\mathrm{V}}^*,\theta)$ with $y(a_{1,1})<\underline{k}(a_{1,1})$, hence by Lemma~\ref{lemma:translated_quivers}, $\nabla(\Gamma_{\mathrm{V}}^*,\theta)$ is not equivalent to 
$\nabla(\Gamma_{\mathrm{V}}^*,\theta-\omega_{\underline{k}})$, contradicting to \eqref{eq:23/2}. 
\end{proof}


\begin{theorem}\label{thm:cubic triangulation} 
Let $\nabla$ be a flow polytope of dimension $1\le \dim(\nabla)\le 4$. If $\nabla$ is not equivalent to the Birkhoff polytope then it has a quadratic triangulation. The Birkhoff polytope has a regular unimodular triangulation whose minimal non-faces have at most $3$ elements.
\end{theorem} 

The proof follows the strategy of \cite{haase-paffenholz} that can be summarized in the following Lemma: 

\begin{lemma}\label{lemma:haase-paffenholz} 
Suppose that 
$\nabla=\nabla_1\cup\dots\cup\nabla_q$ 
is a regular hyperplane subdivision of the $d$-dimensional lattice polytope $\nabla$, where $\nabla_i$ is a 
$d$-dimensional lattice polytope ($i=1,\dots,q$).  Suppose that each $\nabla_i$ has a pulling triangulation whose minimal 
non-faces have at most $n$ elements. 
\begin{itemize}
\item[(i)] The triangulations of the $\nabla_i$ together give a regular triangulation of $\nabla$  whose minimal non-faces 
have at most $\max\{n,2\}$ elements. 
\item[(ii)] If each $\nabla_i$ is compressed, then the above triangulation of $\nabla$ is unimodular. 
\end{itemize} 
\end{lemma} 
\begin{proof}  
As pulling refinements of regular subdivisions are regular (see \cite[Lemma 2.1 (1)]{haase-etal}), this triangulation of $\nabla$ is regular. 
Let $N$ be a minimal non-face of this triangulation. If $N\subseteq \nabla_i$ for some $i$, then it is a minimal 
non-face of the triangulation of $\nabla_i$ we started with, hence $|N|\le n$. 
If no $\nabla_i$ contains $N$, then there is a hyperplane $H$ among the hyperplanes that yield the dissection of $\nabla$ as the union of the $\nabla_i$, such that $N$ intersects non-trivially both half-spaces with border $H$. So there exist 
$v,w\in N$ such that no $\nabla_i$ contains $\{v,w\}$. 
It follows that $\{v,w\}$ is a $2$-element non-face contained in $N$, 
implying by minimality of $N$ that $N=\{v,w\}$, so $|N|=2$.  

For the last statement recall that any pulling triangulation of a compressed polytope is unimodular by a result of Santos (see \cite{ohsugi-hibi} and 
\cite{sullivant} for the first published proofs).  
\end{proof} 

\begin{proofof}{Theorem~\ref{thm:cubic triangulation}}
Assume first that $\nabla$ is prime and $3$-dimensional. 
Consider the regular hyperplane subdivision  $\nabla=\nabla_1\cup\dots\cup\nabla_q$ given by Corollary~\ref{cor:3-regular}.  By Proposition~\ref{prop:all 3-dim compressed} and Proposition~\ref{prop:triangulation of 4-cells}, each $\nabla_i$ has a pulling triangulation that is quadratic. 
Thus  by Lemma~\ref{lemma:haase-paffenholz}, $\nabla$ has a quadratic triangulation. 

We pointed out in the paragraph preceding  Proposition~\ref{prop:all compressed} that the prime compressed 
flow polytopes of dimension $1$ or $2$ are unimodular simplices, hence flow polytopes of dimension $\le 2$ have a quadratic triangulation by  similar reasoning as above. 

As products of polytopes having quadratic triangulation also have a quadratic triangulation 
(see \cite[Proposition 2.11]{haase-etal}), we get that 
all flow polytopes of dimension at most $3$ as well as the non-prime flow polytopes of dimension $4$ have a quadratic triangulation. 

Next we treat the case when $\nabla$ is prime and $4$-dimensional. 
By Proposition~\ref{prop:3-regular} $\nabla\cong \nabla(\Gamma^*,\theta)$, where 
$\Gamma\in \{\Gamma_{\mathrm{I}},\Gamma_{\mathrm{II}},\Gamma_{\mathrm{III}},\Gamma_{\mathrm{IV}},\Gamma_{\mathrm{V}}\}$. 
Assume first that $\Gamma\in \{\Gamma_{\mathrm{I}},\Gamma_{\mathrm{II}},\Gamma_{\mathrm{III}},\Gamma_{\mathrm{IV}}\}$. 
By Lemma~\ref{lemma:regular subdivision} $\nabla$ 
 has a regular hyperplane subdivision $\nabla=\nabla_1\cup\dots\cup\nabla_q$ 
where each $\nabla_i$ belongs to  $\mathcal{P}(\Gamma_{\mathrm{I}})$ or 
$\mathcal{P}(\Gamma_{\mathrm{II}})$ or $\mathcal{P}(\Gamma_{\mathrm{III}})$ or $\mathcal{P}(\Gamma_{\mathrm{IV}})$. Therefore by Proposition~\ref{prop:4-cells} and Proposition~\ref{prop:triangulation of 4-cells} each $\nabla_i$ 
has a quadratic pulling triangulation,  
and hence $\nabla$ has a quadratic triangulation by Lemma~\ref{lemma:haase-paffenholz}.

Finally we turn to the case when $\nabla=\nabla(\Gamma_{\mathrm{V}}^*,\theta)$. 
If $\nabla$ is equivalent to the Birkhoff polytope, then by Proposition~\ref{prop:triangulation of 4-cells} it has a regular unimodular triangulation whose minimal non-faces have 3 elements. 

From now on we assume that the $4$-dimensional polytope $\nabla = \nabla(\Gamma_{\mathrm{V}}^*,\theta)$ is 
not equivalent to the Birkhoff polytope. 
We prove by induction on $\sum |\theta(v)|$ that $\nabla$ has a quadratic triangulation. 
If $\nabla$ is a cell of $(\Gamma_{\mathrm{V}}^*,\theta)$, then $\nabla\in\mathcal{P}(\Gamma_{\mathrm{V}})$, hence 
$\nabla\cong \nabla_{\mathrm{II.c}}^{(4)}$ by Proposition~\ref{prop:4-cells} (v), and has a quadratic triangulation by 
Proposition~\ref{prop:triangulation of 4-cells} (i). 
Otherwise there is an arrow $a$ in $\Gamma_{\mathrm{V}}^*$ such that points of $\nabla$ take at least $3$ distinct integer values on $a$. Without loss of generality we may assume $a = a_{1,1}$ (keeping the notation of Lemma~\ref{lemma:grobner_weights}), and note that $\theta(v_{1,1}) > 1$. If $$k = \min(\{x(a_{1,1}) |\quad x \in \nabla\}) > 0$$ then by Lemma~\ref{lemma:translated_quivers} $\nabla(\Gamma_{\mathrm{V}}^*,\theta)$ is equivalent to $\nabla(\Gamma_{\mathrm{V}}^*,\theta - \omega_{\underline{k}})$ where $\underline{k}$ is defined as $\underline{k}(a_{1,1}) = k$ and 0 on the rest of the arrows. Since $u_1$ is a source and $v_{1,1}$ is a sink in this case we also have $\theta(v_{1,1}) \ge k$ and $\theta(u_{1,1}) \le -k$ and we are done by induction as $\sum |(\theta - \omega_{\underline{k}})(v)| = \sum |\theta(v)| - 2*k$.
Hence we can assume $$\min(\{x(a_{1,1}) |\quad x \in \nabla\}) = 0,$$ 
and subdivide $\nabla$ along the hyperplane $x(a_{1,1}) = 1$. 
We obtain that 
$\nabla = \nabla_1 \cup \nabla_2$ where  
$\nabla_1=\{x\in\nabla\mid x(a_{1,1})\ge 1\}$  and 
$\nabla_2=\{x\in\nabla\mid x(a_{1,1})\le 1\}=\{x\in\nabla\mid x(b_{1,1})\ge 1-\theta(v_{1,1})\}$.
By Lemma~\ref{lemma:translated_quivers} we have 
$\nabla_1 \cong \nabla(\Gamma_{\mathrm{V}}^*,\theta')$ for \begin{align*}\theta'(u_1) = \theta(u_1) + 1\\\theta'(v_{1, 1}) = \theta(v_{1,1}) - 1\\\theta'(v) = \theta(v) \quad \forall v \notin \{u_1, v_{1,1}\},\end{align*}
and $\nabla_2 \cong \nabla(\Gamma_{\mathrm{V}}^*,\theta'')$
for \begin{align*}\theta''(w_1) = \theta(w_1) + \theta(v_{1,1}) - 1\\\theta''(v_{1, 1}) = 1\\\theta'(v) = \theta(v) \quad \forall v \notin \{w_1, v_{1,1}\}.\end{align*}
Note that by our assumption that points in $\nabla$ take at least 3 distinct integer values on $a_{1, 1}$, both $\nabla_1$ and $\nabla_2$ are nonempty. Since $u_1$ and $w_1$ are sources and $v_{1,1}$ is a sink, it follows that all of $\theta$, $\theta'$ and $\theta''$ are non-positive on $u_1$ and $w_1$ and non-negative on $v_{1,1}$. It is then easy to see that $\sum |\theta(v)| > \sum |\theta'(v)|$ and $\sum |\theta(v)| > \sum |\theta''(v)|$, hence the induction hypothesis can be applied. If none of $\nabla_1$ and $\nabla_2$ is equivalent to the Birkhoff polytope then by the induction hypothesis they have quadratic triangulations and by Lemma~\ref{lemma:haase-paffenholz}  so does $\nabla$. Otherwise note that both $\theta'$ and $\theta''$ can be written as $\theta - \omega_{\underline{k}}$ where $\underline{k}$ takes a positive value on precisely one of $a_{1,1}$ and $b_{1,1}$ and is $0$ elsewhere. 
Thus by Lemma~\ref{lemma:grobner_weights} we may assume that $\theta = \theta_{\mathrm{V.b}} + \omega_{\underline{m}} + \omega_{\underline{k}}$ where $\underline{m}(a_{i,j}) = 0$ for all $i, j$. If $\underline{k}(b_{1,1}) > 0$ then $\theta$ itself is of the form given in Lemma~\ref{lemma:grobner_weights}, contradicting the assumption that $\nabla$ is not equivalent to the Birkhoff polytope. If $k = \underline{k}(a_{1,1}) > 0$ then we have that:
\begin{align*}
\theta(u_1) = -1 - k\\  
\theta(u_2) = \theta(u_3) = -1\\ 
\theta(w_i) = -2 - \underline{m}(b_{1,i}) - \underline{m}(b_{2,i}) - \underline{m}(b_{3,i}) \quad\forall i:\quad 1 \leq i \leq 3\\
\theta(v_{1,1}) = 1 + k + \underline{m}(b_{1,1})\\
\theta(v_{i,j}) = 1 + \underline{m}(b_{i,j}) \quad\forall (i,j)\neq (1,1)\\
\end{align*} 
Take a point $x\in \nabla$. 
Since $x(a)\ge 0$ for all  $a\in {(\Gamma_{\mathrm{V}}^*)}_1$, we have $x(a_{i,j})\le -\theta(u_i)$ for all $i,j$, hence  for $(i,j)=(1,1)$  or $i\neq 1$ we have 
\begin{equation}\label{eq:mbij_1} x(b_{i,j})=\theta(v_{i,j})-x(a_{i,j})\ge \theta(v_{i,j})+\theta(u_i)={\underline{m}}(b_{i,j}). 
\end{equation} 
Using the inequality $x(b_{i,j})\le \theta(v_{i,j})$ we get for $j\in\{2,3\}$ that 
\begin{align}\label{eq:mbij_2}
x(b_{1,j})&=-\theta(w_j)-x(b_{2,j})-x(b_{3,j})
\\ \notag &\ge -\theta(w_j)-\theta(v_{2,j})-\theta(v_{3,j})={\underline{m}}(b_{1,j}).
\end{align} 
It follows that 
\begin{align}\label{eq:xa11} 
x(a_{1,1})&=-\theta(u_1)-x(a_{1,2})-x(a_{1,3})
\\ \notag &=k+1-\theta(v_{1,2})+x(b_{1,2})-\theta(v_{1,3})+x(b_{1,3})
\ge k-1.
\end{align}
Applying Lemma~\ref{lemma:translated_quivers} we conclude from \eqref{eq:mbij_1}, \eqref{eq:mbij_2}, \eqref{eq:xa11} that $\nabla \cong \nabla(\Gamma_{\mathrm{V}}^*,\alpha)$ where $\alpha$ is defined as follows:
\begin{align*}
\alpha(u_1) = -2\\
\alpha(u_2) = \alpha(u_3) = -1\\
\alpha(w_j) = -2 \quad\forall j\\
\alpha(v_{1,1}) = 2\\
\alpha(v_{i,j}) = 1 \quad\forall (i,j)\neq (1,1).\\
\end{align*}
Finally it is easy to check that we can contract $b_{1,1}$ and then $a_{1,1}$ in $(\Gamma_{\mathrm{V}}^*,\alpha)$ to obtain $(Q,\sigma)$  with $\nabla\cong \nabla(Q,\sigma)$ and the chassis $\mathcal{C}(\Gamma(Q))$ is the $5$ vertex contracted descendant of $\Gamma_{\mathrm{V}}$, which is also a contracted descendant of $\Gamma_{\mathrm{IV}}$ 
(see Figure~\ref{figure:hasse diagram}). Therefore by Proposition~\ref{prop:Gamma^*}, 
$\nabla=\nabla(\Gamma_{\mathrm{IV}}^*,\beta)$ for some weight $\beta$, and we are done by the first half of the proof. 
\end{proofof}
\begin{remark}
$\nabla(\Gamma_{\mathrm{V}}^*,\alpha)$ in the last part of the proof of Theorem~\ref{thm:cubic triangulation} is a compressed polytope that can not be obtained as a cell of a quiver with chassis $\Gamma_{\mathrm{V}}^*$, hence it did not appear on our lists in Proposition~\ref{prop:4-cells}. Note that there is no contradiction here as $(\Gamma_{\mathrm{V}}^*,\alpha)$ is not tight. In fact  $\nabla(\Gamma_{\mathrm{V}}^*,\alpha)$ is equivalent to the polytope $\nabla_{\mathrm{VI.c}}^{(4)}$ in Section~\ref{sec:remaining 4-cells} and an explicit quadratic triangulation of it will be given in Proposition~\ref{prop:triangulation of new 4-cells}.
\end{remark}


\pgfdeclarelayer{nodelayer}
\pgfdeclarelayer{edgelayer}
\pgfsetlayers{edgelayer,nodelayer,main}

\begin{figure}
\begin{tikzpicture}[scale=0.5]
	\begin{pgfonlayer}{nodelayer}
		\node [style={small_fill}] (0) at (-13, 7) {};
		\node [style={small_fill}] (1) at (-11, 7) {};
		\node [style={small_fill}] (2) at (-13, 5) {};
		\node [style={small_fill}] (3) at (-11, 5) {};
		\node [style={small_fill}] (4) at (-13, 3) {};
		\node [style={small_fill}] (5) at (-11, 3) {};
		\node [style={small_fill}] (6) at (-13, 1) {};
		\node [style={small_fill}] (7) at (-11, 1) {};
		\node [style={small_fill}] (8) at (-12, 0) {};
		\node [style={small_fill}] (9) at (-12, -1) {};
		\node [style={small_fill}] (10) at (-13, -2) {};
		\node [style={small_fill}] (11) at (-11, -2) {};
		\node [style={small_fill}] (12) at (-13, -4) {};
		\node [style={small_fill}] (13) at (-11, -4) {};
		\node [style={small_fill}] (14) at (-11, -4) {};
		\node [style={small_fill}] (15) at (-13, -6) {};
		\node [style={small_fill}] (16) at (-11, -6) {};
		\node [style={small_fill}] (17) at (-13, -8) {};
		\node [style={small_fill}] (18) at (-11, -8) {};
		\node [style={small_fill}] (19) at (-13, -10) {};
		\node [style={small_fill}] (20) at (-11, -10) {};
		\node [style={small_fill}] (21) at (-13, -12) {};
		\node [style={small_fill}] (22) at (-11, -12) {};
		\node [style={small_fill}] (23) at (-13, -14) {};
		\node [style={small_fill}] (24) at (-11, -14) {};
		\node [style={small_fill}] (25) at (-13, -16) {};
		\node [style={small_fill}] (26) at (-13, -18) {};
		\node [style={small_fill}] (27) at (-13, -20) {};
		\node [style={small_fill}] (28) at (-11, -20) {};
		\node [style={small_fill}] (29) at (-11, -18) {};
		\node [style={small_fill}] (30) at (-11, -16) {};
		\node [style={small_fill}] (31) at (-7, 4) {};
		\node [style={small_fill}] (32) at (-5, 4) {};
		\node [style={small_fill}] (33) at (-6, 3) {};
		\node [style={small_fill}] (34) at (-7, 2) {};
		\node [style={small_fill}] (35) at (-5, 2) {};
		\node [style={small_fill}] (36) at (-6, 0) {};
		\node [style={small_fill}] (37) at (-6, -1.25) {};
		\node [style={small_fill}] (38) at (-6, -2.5) {};
		\node [style={small_fill}] (39) at (-8, -3) {};
		\node [style={small_fill}] (40) at (-4, -3) {};
		\node [style={small_fill}] (46) at (-7, -5) {};
		\node [style={small_fill}] (47) at (-5, -5) {};
		\node [style={small_fill}] (48) at (-5, -7) {};
		\node [style={small_fill}] (49) at (-7, -7) {};
		\node [style={small_fill}] (50) at (-7, -9) {};
		\node [style={small_fill}] (51) at (-7, -11) {};
		\node [style={small_fill}] (52) at (-5, -11) {};
		\node [style={small_fill}] (53) at (-5, -14) {};
		\node [style={small_fill}] (54) at (-7, -14) {};
		\node [style={small_fill}] (55) at (-7, -12.5) {};
		\node [style={small_fill}] (56) at (1.25, 7.5) {};
		\node [style={small_fill}] (59) at (-0.75, 4.5) {};
		\node [style={small_fill}] (60) at (3.25, 4.5) {};
		\node [style={small_fill}] (61) at (1.25, 5.75) {};
		\node [style={small_fill}] (62) at (0, 2) {};
		\node [style={small_fill}] (63) at (3, 2) {};
		\node [style={small_fill}] (64) at (0, -1) {};
		\node [style={small_fill}] (65) at (3, -1) {};
		\node [style={small_fill}] (66) at (0, -9) {};
		\node [style={small_fill}] (67) at (3, -9) {};
		\node [style={small_fill}] (68) at (0, -12) {};
		\node [style={small_fill}] (69) at (3, -12) {};
		\node [style={small_fill}] (70) at (0, -3) {};
		\node [style={small_fill}] (71) at (3, -3) {};
		\node [style={small_fill}] (72) at (0, -6) {};
		\node [style={small_fill}] (73) at (3, -6) {};
		\node [style={small_fill}] (74) at (0, -14) {};
		\node [style={small_fill}] (75) at (3, -14) {};
		\node [style={small_fill}] (76) at (0, -17) {};
		\node [style={small_fill}] (77) at (3, -17) {};
		\node [style={small_fill}] (78) at (9.25, 2) {};
		\node [style={small_fill}] (79) at (7.25, -1) {};
		\node [style={small_fill}] (80) at (11.25, -1) {};
		\node [style={small_fill}] (81) at (9.25, -3.5) {};
		\node [style={small_fill}] (82) at (7.25, -6.5) {};
		\node [style={small_fill}] (83) at (11.25, -6.5) {};
		\node [style={small_fill}] (84) at (7.25, -12) {};
		\node [style={small_fill}] (85) at (11.25, -12) {};
		\node [style=none] (86) at (-10, 5) {};
		\node [style=none] (87) at (-7.75, 3) {};
		\node [style=none] (89) at (-8, -1.25) {};
		\node [style=none] (91) at (-10, -6) {};
		\node [style=none] (92) at (-10, -0.5) {};
		\node [style=none] (93) at (-10, -12) {};
		\node [style=none] (94) at (-8, -7) {};
		\node [style=none] (95) at (-10.25, -18) {};
		\node [style=none] (96) at (-8, -12.5) {};
		\node [style=none] (97) at (-4, -12.5) {};
		\node [style=none] (98) at (-1, -15.5) {};
		\node [style=none] (99) at (-1, -10.5) {};
		\node [style=none] (101) at (-1, -4.5) {};
		\node [style=none] (102) at (-4, -6.5) {};
		\node [style=none] (103) at (-4, -1) {};
		\node [style=none] (104) at (-1, 0.5) {};
		\node [style=none] (105) at (-1, 6) {};
		\node [style=none] (106) at (-4, 3) {};
		\node [style=none] (107) at (4, 6) {};
		\node [style=none] (108) at (4, 0.5) {};
		\node [style=none] (109) at (4, -4.5) {};
		\node [style=none] (110) at (4, -10.5) {};
		\node [style=none] (111) at (4, -15.5) {};
		\node [style=none] (112) at (6.25, 0.5) {};
		\node [style=none] (113) at (6.25, -5) {};
		\node [style=none] (115) at (9.25, -7.25) {};
		\node [style=none] (116) at (9.25, -10.75) {};
	\end{pgfonlayer}
	\begin{pgfonlayer}{edgelayer}
		\draw (0) to (1);
		\draw (0) to (3);
		\draw (0) to (5);
		\draw (2) to (1);
		\draw (2) to (3);
		\draw (2) to (5);
		\draw (4) to (5);
		\draw (4) to (1);
		\draw (4) to (3);
		\draw (6) to (8);
		\draw (6) to (7);
		\draw (7) to (8);
		\draw (7) to (11);
		\draw (8) to (9);
		\draw (9) to (10);
		\draw (10) to (11);
		\draw (9) to (11);
		\draw (6) to (10);
		\draw [in=165, out=15] (12) to (14);
		\draw [in=195, out=-15] (12) to (14);
		\draw (12) to (15);
		\draw (15) to (17);
		\draw (16) to (18);
		\draw (14) to (16);
		\draw (15) to (16);
		\draw [in=195, out=-15] (17) to (18);
		\draw [in=165, out=15] (17) to (18);
		\draw [in=195, out=-15] (19) to (20);
		\draw [in=165, out=15] (19) to (20);
		\draw (19) to (21);
		\draw (21) to (24);
		\draw (22) to (23);
		\draw (21) to (23);
		\draw (23) to (24);
		\draw (22) to (24);
		\draw (20) to (22);
		\draw [in=195, out=-15] (25) to (30);
		\draw [in=165, out=15] (25) to (30);
		\draw (25) to (26);
		\draw [in=75, out=-75] (26) to (27);
		\draw [in=105, out=-105] (26) to (27);
		\draw (27) to (28);
		\draw [in=285, out=75] (28) to (29);
		\draw (29) to (30);
		\draw [in=105, out=-105] (29) to (28);
		\draw (31) to (32);
		\draw (31) to (33);
		\draw (32) to (33);
		\draw (33) to (35);
		\draw (33) to (34);
		\draw (31) to (34);
		\draw (32) to (35);
		\draw (34) to (35);
		\draw [in=105, out=-105] (36) to (37);
		\draw [in=75, out=-75] (36) to (37);
		\draw (36) to (40);
		\draw (36) to (39);
		\draw (37) to (38);
		\draw (39) to (38);
		\draw (38) to (40);
		\draw (39) to (40);
		\draw (46) to (49);
		\draw (49) to (48);
		\draw (48) to (50);
		\draw (47) to (48);
		\draw [in=195, out=-15] (46) to (47);
		\draw [in=165, out=15] (46) to (47);
		\draw [in=105, out=-105] (49) to (50);
		\draw [in=75, out=-75] (49) to (50);
		\draw [in=195, out=-15] (51) to (52);
		\draw [in=165, out=15] (51) to (52);
		\draw [in=105, out=-105] (51) to (55);
		\draw [in=75, out=-75] (51) to (55);
		\draw (55) to (54);
		\draw (52) to (53);
		\draw [in=195, out=-15] (54) to (53);
		\draw [in=165, out=15] (54) to (53);
		\draw (56) to (60);
		\draw (56) to (59);
		\draw (59) to (60);
		\draw [in=255, out=105] (61) to (56);
		\draw [in=75, out=-75] (56) to (61);
		\draw (61) to (59);
		\draw (61) to (60);
		\draw [in=195, out=-15] (62) to (63);
		\draw [in=165, out=15] (62) to (63);
		\draw [in=105, out=-105] (62) to (64);
		\draw [in=75, out=-75] (62) to (64);
		\draw (64) to (65);
		\draw (62) to (65);
		\draw (63) to (65);
		\draw [in=195, out=-15] (66) to (67);
		\draw [in=165, out=15] (66) to (67);
		\draw [in=105, out=-105] (66) to (68);
		\draw [in=165, out=15] (68) to (69);
		\draw [in=75, out=-75] (66) to (68);
		\draw [in=195, out=-15] (68) to (69);
		\draw (67) to (69);
		\draw (70) to (72);
		\draw (70) to (73);
		\draw [in=195, out=-15] (72) to (73);
		\draw [in=165, out=15] (70) to (71);
		\draw (71) to (73);
		\draw [in=195, out=-15] (70) to (71);
		\draw [in=165, out=15] (72) to (73);
		\draw (74) to (76);
		\draw [in=195, out=-15] (76) to (77);
		\draw [in=165, out=15] (74) to (75);
		\draw (75) to (77);
		\draw [in=195, out=-15] (74) to (75);
		\draw [in=165, out=15] (76) to (77);
		\draw (76) to (77);
		\draw [in=135, out=-75] (78) to (80);
		\draw [in=75, out=-135] (78) to (79);
		\draw [in=195, out=-15] (79) to (80);
		\draw [in=45, out=-105] (78) to (79);
		\draw [in=105, out=-45] (78) to (80);
		\draw [in=165, out=15] (79) to (80);
		\draw [in=135, out=-75] (81) to (83);
		\draw [in=75, out=-135] (81) to (82);
		\draw [in=45, out=-105] (81) to (82);
		\draw [in=105, out=-45] (81) to (83);
		\draw (82) to (83);
		\draw (81) to (82);
		\draw [in=210, out=-30] (84) to (85);
		\draw [in=195, out=-15] (84) to (85);
		\draw [in=150, out=30] (84) to (85);
		\draw [in=165, out=15] (84) to (85);
		\draw (84) to (85);
		\draw [style=new edge style 2] (86.center) to (87.center);
		\draw [style=new edge style 2] (92.center) to (89.center);
		\draw [style=new edge style 2] (92.center) to (87.center);
		\draw [style=new edge style 2] (93.center) to (94.center);
		\draw [style=new edge style 2] (95.center) to (96.center);
		\draw [style=new edge style 2] (93.center) to (89.center);
		\draw [style=new edge style 2] (106.center) to (105.center);
		\draw [style=new edge style 2] (106.center) to (104.center);
		\draw [style=new edge style 2] (103.center) to (105.center);
		\draw [style=new edge style 2] (103.center) to (104.center);
		\draw [style=new edge style 2] (103.center) to (101.center);
		\draw [style=new edge style 2] (103.center) to (99.center);
		\draw [style=new edge style 2] (102.center) to (104.center);
		\draw [style=new edge style 2] (102.center) to (101.center);
		\draw [style=new edge style 2] (102.center) to (98.center);
		\draw [style=new edge style 2] (97.center) to (99.center);
		\draw [style=new edge style 2] (107.center) to (112.center);
		\draw [style=new edge style 2] (107.center) to (113.center);
		\draw [style=new edge style 2] (108.center) to (113.center);
		\draw [style=new edge style 2] (109.center) to (113.center);
		\draw [style=new edge style 2] (110.center) to (112.center);
		\draw [style=new edge style 2] (111.center) to (113.center);
		\draw [style=new edge style 2] (115.center) to (116.center);
		\draw [style=new edge style 2] (91.center) to (94.center);
		\draw [style=new edge style 2] (93.center) to (96.center);
	\end{pgfonlayer} 

\end{tikzpicture} 
    \centering
    \caption{} \label{figure:hasse diagram}	
\end{figure}


\section{Consequences for toric ideals} \label{sec:toric ideals} 

Recall that given a lattice polytope $\nabla\subset \mathbb{R}^d$ and a field $K$ 
one associates to them  the ideal 
$\mathcal{I}_{\nabla}$ in the multivariate polynomial ring $K[t_z\mid z\in \nabla\cap \mathbb{Z}^d]$ 
(the variables are labeled by the lattice points in $\nabla$) generated by all the binomials 
\[\prod_{z\in \nabla\cap \mathbb{Z}^d}t_z^{\alpha(z)}-\prod_{z\in \nabla\cap \mathbb{Z}^d}t_z^{\beta(z)},\]  
where $\alpha,\beta\in \mathbb{N}_0^{\nabla\cap\mathbb{Z}^d}$ satisfy 
$\sum_{z\in \nabla\cap \mathbb{Z}^d}\alpha(z)z=\sum_{z\in \nabla\cap \mathbb{Z}^d}\beta(z)z\in \mathbb{Z}^d$ 
and $\sum_{z\in \nabla\cap \mathbb{Z}^d}\alpha(z)=\sum_{z\in \nabla\cap \mathbb{Z}^d}\beta(z)\in \mathbb{N}_0$. 
Note that for equivalent lattice polytopes $\nabla$, $\nabla'$ we have $\mathcal{I}_{\nabla}=\mathcal{I}_{\nabla'}$. 
For sake of completeness of the picture, we recall the algebro-geometric relevance of the 
ideals $\mathcal{I}_{\nabla}$. There is a canonical construction of  a projective toric variety 
$X_{\nabla}$ for any lattice polytope $\nabla$, as the toric variety whose fan is the normal fan of $\nabla$ 
(see \cite{cox-little-schenck} for details). When $\nabla$ is normal, the variety $X_{\nabla}$ can be realized as the 
Zariski closure in the projective space $\mathbb{P}(K^{m+1})=\mathbb{P}^{m}$ of the 
subset $\{(x^{z_0}:x^{z_1}:\dots:x^{z_m})\in \mathbb{P}^m\mid x\in (K^{\times})^d\}$, 
where $z_0,\dots,z_m$ are the lattice points in $\nabla$, $K^{\times}=K\setminus\{0\}$, 
and $x^z=x_1^{z(1)}\cdots x_d^{z(d)}$ for $x\in (K^{\times})^d$ and $z\in \nabla\cap\mathbb{Z}^d$. 
Moreover, $\mathcal{I}_{\nabla}$ is the vanishing ideal of this embedded projective variety. 

In dimension at most $4$, 
Theorem~\ref{thm:cubic triangulation} implies a positive answer for Question~\ref{question:cubic Groebner}:  

\begin{corollary} \label{cor:cubic groebner}
Let $\nabla$ be a flow polytope of dimension at most $4$. If $\nabla$ is not equivalent to the Birkhoff polytope then it has a square-free initial ideal generated by quadratic elements 
(in particular, $\mathcal{I}_{\nabla}$ has a Gr\"obner basis consisting of elements of degree $2$). The $4$-dimensional  Birkhoff polytope has a square-free initial ideal generated by monomials of degree $3$. 
\end{corollary}  
\begin{proof} 
By Theorem~\ref{thm:cubic triangulation}  $\nabla$ has a regular unimodular triangulation.  
By \cite[Theorem 8.3]{sturmfels} this triangulation is the initial complex of some initial ideal of $\mathcal{I}_{\nabla}$. Moreover, this  initial ideal is generated by square-free monomials by \cite[Corollary 8.9]{sturmfels}. Therefore  the generators of the initial ideal  correspond to the minimal non-faces of the triangulation.  So our statement is an immediate consequence of Theorem~\ref{thm:cubic triangulation}.  
\end{proof} 

\begin{remark}\label{remark:bogvad} (i) The so-called B\"ogvad conjecture asserts that the toric ideal of a smooth polytope is generated by quadratic polynomials (see for example \cite{bruns}, \cite{balletti}, \cite{higashitani-ohsugi} for partial results on this conjecture). Corollary~\ref{cor:cubic groebner}    
implies that the B\"ogvad conjecture holds for smooth flow polytopes up to dimension $4$. 

(ii) The fact that the toric ideal of a flow polytope of dimension at most four is quadratically generated with the only exception of the Birkhoff polytope was shown before in our preprint \cite{domokos-joo:2}. 
\end{remark}

The logic of the proof of Corollary~\ref{cor:cubic groebner} together with \cite[Proposition 13.15 (ii)]{sturmfels} 
imply that the toric ideal of a $d$-dimensional flow polytope has an initial ideal generated by square-free monomials of degree at most $d+1$. This bound can be slightly improved thanks to the following: 

\begin{lemma}\label{lemma:at most d} 
Let $\nabla$ be a $d$-dimensional lattice polytope all of whose lattice points are vertices of $\nabla$.
Then the  minimal non-faces of a triangulation of $\nabla$ have at most $d$ elements. 
\end{lemma} 
\begin{proof} 
The proof of \cite[Proposition 13.15 (ii)]{sturmfels} shows that the elements in a minimal non-face can not be affinely dependent. In particular a minimal non-face can not have more than $d+1$ elements.   
It remains to exclude the case that $v_1,\dots, v_{d+1}$ are affinely independent and form a minimal non-face of the triangulation. 
So $P = \mathrm{Conv}(v_1,\dots, v_{d+1})$ is a simplex of dimension $d$. Let $x$ be a point in its interior. Then $x$ is contained in some simplex  $S$ of the triangulation. Clearly $S\neq P$, since  $P$ is not a face of the triangulation. Therefore $S$ has a vertex $w\notin \{v_1,\dots,v_{d+1}\}$. It follows that $w \notin P$, since all the lattice points in 
$\nabla$ are vertices of $\nabla$ by assumption, thus the convex hull of lattice points can not contain another lattice point. 
Since $S$ intersects $P$ in an interior point, $S \cap P$ contains an open sphere, hence we can pick $x' \in S \cap P$ such that $x'$ is an interior point of both $P$ and $S$ and that the line segment from $w$ to $x'$ intersects a facet $F$ of $P$ in an interior point of $F$. Recall that  $\{v_1,\dots,v_{d+1}\}$ is a minimal non-face, hence $F$ is a face of the triangulation.  Since $S$ intersects $P$ in an interior point we have that $F$ is a face (and thus a facet) of $S$. This is a contradiction, since the line segment between the vertex $w$ of the simplex $S$ and the interior point $x' \in S$ can not meet the facet $F$ of $S$ opposite to $w$.
\end{proof} 

\begin{proposition}\label{prop:general}
The toric ideal $\mathcal{I}_{\nabla}$ of any $d$-dimensional flow 
polytope $\nabla$ has an initial ideal generated by square-free monomials of degree at most $d$. 
\end{proposition}
\begin{proof} 
Taking a pulling triangulation of the cells of the regular hyperplane subdivision of $\nabla$ discussed in Section~\ref{sec:regular subdivision}  we get a regular unimodular triangulation of $\nabla$ 
by Lemma~\ref{lemma:haase-paffenholz}.  Therefore by \cite[Theorem 8.3]{sturmfels} and \cite[Corollary 8.9]{sturmfels} $\mathcal{I}_{\nabla}$  has an initial ideal generated by square-free monomials corresponding to the minimal non-faces of the triangulation. The cells of the above regular hyperplane subdivision of $\nabla$ are compressed, hence all lattice points in a cell are verices of the cell. Therefore by Lemma~\ref{lemma:at most d} and by  Lemma~\ref{lemma:haase-paffenholz} the minimal non-faces of our regular unimodular triangulation of $\nabla$ have size at most $\max\{d,2\}$.  
\end{proof}

\section{The remaining compressed $4$-dimensional flow polytopes} \label{sec:remaining 4-cells} 

In this section we complete the classification of the $4$-dimensional prime compressed flow polytopes. 
In principle this could be done in the same way as in the $3$-dimensional case in Section~\ref{sec:class-compressed-3-dim}, but significantly more graphs and weights must be analyzed. The amount of computations needed can be  drastically decreased thanks to Lemma~\ref{lemma:reduction34} below. 

\begin{lemma}\label{lemma:reduction34}
Let $v,w_1$ be valency $3$ vertices in $\Gamma\in\mathcal{L}_d$ connected by two edges. 
The third edge $c$ adjacent to $w_1$ connects it to $w_2$, and suppose that $w_2 \neq v$. 
Denote by $\Gamma'$ the graph obtained by contracting the edge $c$ in $\Gamma$ (so $\chi(\Gamma')=d)$; write $w$ for 
the vertex of $\Gamma'$ arising from $w_1,w_2$ glued together. Let $\theta'\in \mathbb{Z}^{(\Gamma')^*_0}$ be a  weight with  $\theta'(v)\in\{-1,-2\}$ and $\theta'(s)=1$ for all source vertices of $(\Gamma')^*$. Define 
$\theta\in \mathbb{Z}^{\Gamma^*_0}$ as follows: 
\begin{align*} \theta(u)&=\theta'(u)\text{ for }  u\in \Gamma^*_0\setminus \{w_1,w_2,v_c\}=(\Gamma')^*_0\setminus\{w\} 
\\ &\text{ and } 
\begin{cases} \theta(w_1)=-2, \ \theta(w_2)=\theta'(w)+1 & \text{ if } \theta'(v)=-1\\
\theta(w_1)=-1,\  \theta (w_2)=\theta'(w) &\text{ if }\theta'(v)=-2\end{cases}.
\end{align*} 
\begin{itemize} 
\item[(i)] Then $\nabla((\Gamma')^*,\theta')\cong \nabla(\Gamma^*,\theta)$. 
\item[(ii)] Moreover,  if $\theta'\in \mathcal{W}(\Gamma')$ and $\dim(\nabla((\Gamma')^*,\theta')=d$, then 
$\theta\in \mathcal{W}(\Gamma)$. 
\end{itemize}
\end{lemma}

\begin{proof} (i) Let us denote by $c_1$ and $c_2$ the arrows of $\Gamma^*$ adjacent to the sink that was put on the edge $c$ of $\Gamma$. 
The value of $x\in \mathcal{A}(\Gamma^*,\theta)$ on two arrows starting at $v$ and ending at the sinks that were put on the two edges between $v$ and $w_1$ determines 
its value on the other arrows of $\Gamma^*$ adjacent to $v$, $w_1$, or the sinks that were put on the edges of 
$\Gamma$ adjacent to $v$ or $w_1$. This is shown explicitly in the figure below in the case $\theta(v)=-1$: 

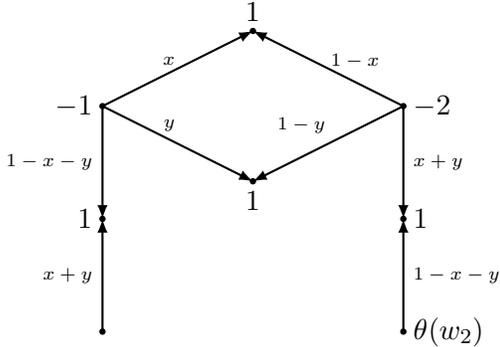
\begin{tikzpicture}[>=latex]; 
\foreach \x in {(0,0),(2,1),(4,0),(2,-1),(4,-1.5),(4,-3),(0,-1.5),(0,-3)} \filldraw \x circle (1pt); 
\draw [->, thick] (0,0) to (2,1); \draw [->, thick] (4,0) to (2,1);\draw [->, thick] (0,0) to (2,-1); \draw [->, thick] (4,0) to (2,-1);
\draw [->, thick] (0,0) to (0,-1.5); \draw [->, thick] (0,-3) to (0,-1.5); \draw [->, thick] (4,0) to (4,-1.5); \draw [->, thick] (4,-3) to (4,-1.5);
\node [left] at (0,-0.75) {\text{\tiny$1-x-y$}}; \node [left] at (0,-2.25) {\text{\tiny$x+y$}}; \node [left] at (1.1,0.6) {\text{\tiny$x$}};
\node [left] at (1.1,-0.25) {\text{\tiny$y$}}; \node [right] at (2.9,0.6) {\text{\tiny$1-x$}};
\node [left] at (3.1,-0.25) {\text{\tiny$1-y$}}; \node [right] at (4,-0.75) {\text{\tiny$x+y$}}; \node [right] at (4,-2.25) {\text{\tiny$1-x-y$}};
\node [left] at (0,0) {$-1$}; \node [above] at (2,1) {$1$}; \node [below] at (2,-1) {$1$}; \node [right] at (4,0) {$-2$}; 
\node [left] at (0,-1.5) {$1$}; \node [right] at (4,-1.5) {$1$}; \node [right] at (4,-3) {$\theta(w_2)$}; 
\end{tikzpicture}

One can read off from the above figure that both of $c_1$ and $c_2$ are contractible for $\theta$, and contracting any of them, 
the other is contractible for the resulting weight; contracting it as well we end up with the pair $((\Gamma')^*,\theta')$.  
Consequently, $\nabla(\Gamma^*,\theta)\cong \nabla((\Gamma')^*,\theta')$ 
by Proposition~\ref{prop:tightness} (iv).  
The case $\theta(v)=-2$ is handled similarly (in fact we can use the same figure as above, except that the vertex corresponding to $w_2$ will be the left bottom vertex).

(ii) If $\theta' \in \mathcal{W}(\Gamma')$, then the weight $\theta$ satisfies all requirements  from Definition~\ref{def:W(Gamma)} for a weight to belong to $\mathcal{W}(\Gamma)$ except possibly one of the required inequalities for $\theta(w_2)$. Suppose that  $\theta(w_2)\ge 0$ or $\theta(w_2)\le -\mathrm{valency}_{\Gamma}(w_2)$. 
Since $w_2$ is a source vertex connected only to sinks in $\Gamma^*$, and $\theta$ takes value $1$ on these sinks, 
we conclude that in this case $\nabla(\Gamma^*,\theta)$ is empty or has dimension strictly smaller than 
$\chi(\Gamma^*)=\chi(\Gamma)=d$. By (i) we conclude that $\dim(\nabla((\Gamma')^*,\theta')<d$ as well, so (ii) holds. 
\end{proof} 

\begin{corollary} \label{cor:refined all compressed} 
For $d\ge 2$ any prime compressed $d$-dimensional flow polytope belongs to $\mathcal{P}(\Gamma)$ 
for some $\Gamma\in \mathcal{L}_d$ containing no double edge connecting 
a valency $3$ vertex with a valency $\ge 4$ vertex. 
\end{corollary} 

\begin{proof} 
By Proposition~\ref{prop:all compressed}, any $d$-dimensional compressed prime flow polytope occurs in $\mathcal{P}(\Gamma')$ for  some $\Gamma'\in \mathcal{L}_d$.  
If $\Gamma'$ has a double edge connecting a valency $3$ vertex $v$ with a valency $\ge 4$ vertex $w$, then 
$\Gamma$ is a contracted descendant of another graph $\Gamma\in  \mathcal{L}_d$ as in Lemma~\ref{lemma:reduction34} (see the last paragraph of Section~\ref{sec:graphs} for more details). Now $\Gamma$ has  
one less instances of a double edge that connects a valency $3$ vertex with a valency $\ge 4$ vertex, and  
by Lemma~\ref{lemma:reduction34},  we have 
$\mathcal{P}(\Gamma)\subseteq \mathcal{P}(\Gamma')$. 
Successively applying this reduction we get the desired statement.  
 \end{proof} 

Inspection of  Figure~\ref{figure:hasse diagram} shows that there are nine graphs in $\mathcal{L}_d$ 
having no double edge that connects a valency $3$ vertex with a valency $\ge 4$ vertex, 
namely the graphs $\Gamma_{\mathrm{I}}, \Gamma_{\mathrm{II}}, \Gamma_{\mathrm{III}}, \Gamma_{\mathrm{IV}}, \Gamma_{\mathrm{V}}$ defined in Section~\ref{sec:4-cells} and the graphs 
$\Gamma_{\mathrm{VI}}, \Gamma_{\mathrm{VII}}, \Gamma_{\mathrm{VIII}}, \Gamma_{\mathrm{IX}}$ 
defined below. 

For these latter four graphs $\Gamma$ we can compute the cells of $(\Gamma^*,\theta)$ for all $\theta\in \mathcal{W}(\Gamma)$ 
in the same way as in Section~\ref{sec:4-cells}. 
First we list the graphs and indicate the relevant weights (we give the value of the weight on the vertices 
belonging to $\Gamma_0$; for the vertices in $\Gamma^*_0\setminus\Gamma_0$ the value is $1$): 

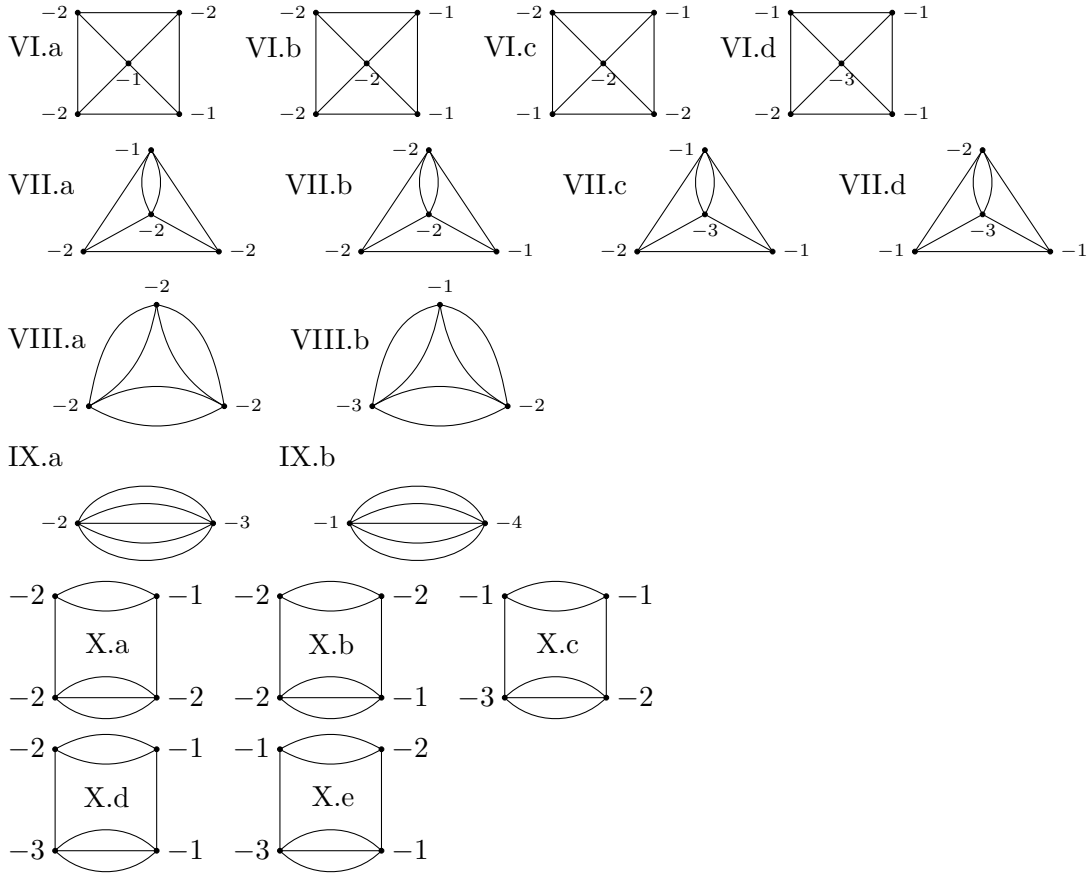
\begin{tikzpicture}[scale=0.45]  
\node at (-1.25,2) {VI.a};
\foreach \x in {(0,0),(3,0),(1.5,1.5),(0,3),(3,3)} \filldraw \x circle (2pt); 
\draw (0,0)--(0,3)--(3,3)--(3,0)--(0,0); \draw (0,0)--(1.5,1.5)--(3,0);\draw (0,3)--(1.5,1.5)--(3,3);
\node [left] at (0,0) {\text{\tiny$-2$}}; 
\node [left] at (0,3) {\text{\tiny$-2$}}; 
\node[right] at (3,0)  {\text{\tiny$-1$}};
\node[right] at (3,3)  {\text{\tiny$-2$}};
\node[below] at (1.5,1.5)  {\text{\tiny$-1$}};
\end{tikzpicture}
\begin{tikzpicture}[scale=0.45]  
\node at (-1.25,2) {VI.b};
\foreach \x in {(0,0),(3,0),(1.5,1.5),(0,3),(3,3)} \filldraw \x circle (2pt); 
\draw (0,0)--(0,3)--(3,3)--(3,0)--(0,0); \draw (0,0)--(1.5,1.5)--(3,0);\draw (0,3)--(1.5,1.5)--(3,3);
\node [left] at (0,0) {\text{\tiny$-2$}}; 
\node [left] at (0,3) {\text{\tiny$-2$}}; 
\node[right] at (3,0)  {\text{\tiny$-1$}};
\node[right] at (3,3)  {\text{\tiny$-1$}};
\node[below] at (1.5,1.5)  {\text{\tiny$-2$}};
\end{tikzpicture}
\begin{tikzpicture}[scale=0.45]  
\node at (-1.25,2) {VI.c};
\foreach \x in {(0,0),(3,0),(1.5,1.5),(0,3),(3,3)} \filldraw \x circle (2pt); 
\draw (0,0)--(0,3)--(3,3)--(3,0)--(0,0); \draw (0,0)--(1.5,1.5)--(3,0);\draw (0,3)--(1.5,1.5)--(3,3);
\node [left] at (0,0) {\text{\tiny$-1$}}; 
\node [left] at (0,3) {\text{\tiny$-2$}}; 
\node[right] at (3,0)  {\text{\tiny$-2$}};
\node[right] at (3,3)  {\text{\tiny$-1$}};
\node[below] at (1.5,1.5)  {\text{\tiny$-2$}};
\end{tikzpicture}
\begin{tikzpicture}[scale=0.45]  
\node at (-1.25,2) {VI.d};
\foreach \x in {(0,0),(3,0),(1.5,1.5),(0,3),(3,3)} \filldraw \x circle (2pt); 
\draw (0,0)--(0,3)--(3,3)--(3,0)--(0,0); \draw (0,0)--(1.5,1.5)--(3,0);\draw (0,3)--(1.5,1.5)--(3,3);
\node [left] at (0,0) {\text{\tiny$-2$}}; 
\node [left] at (0,3) {\text{\tiny$-1$}}; 
\node[right] at (3,0)  {\text{\tiny$-1$}};
\node[right] at (3,3)  {\text{\tiny$-1$}};
\node[below] at (1.5,1.5)  {\text{\tiny$-3$}};
\end{tikzpicture} 

\begin{tikzpicture}[scale=0.45]  
\node at (-1.25,2) {VII.a};
\foreach \x in {(0,0),(4,0),(2,1.1),(2,3)} \filldraw \x circle (2pt); 
\draw (0,0)--(2,3)--(4,0)--(2,1.1)--(0,0)--(4,0); 
\draw (2,1.1) to [out=120,in=-120] (2,3); \draw (2,1.1) to [out=60,in=-60] (2,3);
\node [left] at (0,0) {\text{\tiny$-2$}}; 
\node[right] at (4,0)  {\text{\tiny$-2$}};
\node[left] at (2,3)  {\text{\tiny$-1$}};
\node[below] at (2,1.1)  {\text{\tiny$-2$}};
\end{tikzpicture}
\begin{tikzpicture}[scale=0.45]  
\node at (-1.25,2) {VII.b};
\foreach \x in {(0,0),(4,0),(2,1.1),(2,3)} \filldraw \x circle (2pt); 
\draw (0,0)--(2,3)--(4,0)--(2,1.1)--(0,0)--(4,0); 
\draw (2,1.1) to [out=120,in=-120] (2,3); \draw (2,1.1) to [out=60,in=-60] (2,3);
\node [left] at (0,0) {\text{\tiny$-2$}}; 
\node[right] at (4,0)  {\text{\tiny$-1$}};
\node[left] at (2,3)  {\text{\tiny$-2$}};
\node[below] at (2,1.1)  {\text{\tiny$-2$}};
\end{tikzpicture}
\begin{tikzpicture}[scale=0.45]  
\node at (-1.25,2) {VII.c};
\foreach \x in {(0,0),(4,0),(2,1.1),(2,3)} \filldraw \x circle (2pt); 
\draw (0,0)--(2,3)--(4,0)--(2,1.1)--(0,0)--(4,0); 
\draw (2,1.1) to [out=120,in=-120] (2,3); \draw (2,1.1) to [out=60,in=-60] (2,3);
\node [left] at (0,0) {\text{\tiny$-2$}}; 
\node[right] at (4,0)  {\text{\tiny$-1$}};
\node[left] at (2,3)  {\text{\tiny$-1$}};
\node[below] at (2,1.1)  {\text{\tiny$-3$}};
\end{tikzpicture}
\begin{tikzpicture}[scale=0.45]  
\node at (-1.25,2) {VII.d};
\foreach \x in {(0,0),(4,0),(2,1.1),(2,3)} \filldraw \x circle (2pt); 
\draw (0,0)--(2,3)--(4,0)--(2,1.1)--(0,0)--(4,0); 
\draw (2,1.1) to [out=120,in=-120] (2,3); \draw (2,1.1) to [out=60,in=-60] (2,3);
\node [left] at (0,0) {\text{\tiny$-1$}}; 
\node[right] at (4,0)  {\text{\tiny$-1$}};
\node[left] at (2,3)  {\text{\tiny$-2$}};
\node[below] at (2,1.1)  {\text{\tiny$-3$}};
\end{tikzpicture}

\begin{tikzpicture}[scale=0.45]  
\node at (-1.25,2) {VIII.a};
\foreach \x in {(0,0),(4,0),(2,3)} \filldraw \x circle (2pt); 
\draw (0,0) to [out=30,in=150] (4,0); \draw (0,0) to [out=-30,in=-150] (4,0); 
\draw (0,0) to [out=80,in=-170] (2,3); \draw (0,0) to [out=30,in=-100] (2,3); 
\draw (4,0) to [out=150,in=-80] (2,3); \draw (4,0) to [out=100,in=-10] (2,3); 
\node [left] at (0,0) {\text{\tiny$-2$}}; 
\node[right] at (4,0)  {\text{\tiny$-2$}};
\node[above] at (2,3)  {\text{\tiny$-2$}};
\end{tikzpicture}
\begin{tikzpicture}[scale=0.45]  
\node at (-1.25,2) {VIII.b};
\foreach \x in {(0,0),(4,0),(2,3)} \filldraw \x circle (2pt); 
\draw (0,0) to [out=30,in=150] (4,0); \draw (0,0) to [out=-30,in=-150] (4,0); 
\draw (0,0) to [out=80,in=-170] (2,3); \draw (0,0) to [out=30,in=-100] (2,3); 
\draw (4,0) to [out=150,in=-80] (2,3); \draw (4,0) to [out=100,in=-10] (2,3); 
\node [left] at (0,0) {\text{\tiny$-3$}}; 
\node[right] at (4,0)  {\text{\tiny$-2$}};
\node[above] at (2,3)  {\text{\tiny$-1$}};
\end{tikzpicture}

\begin{tikzpicture}[scale=0.45]  
\node at (-1.25,2) {IX.a};
\foreach \x in {(0,0),(4,0)} \filldraw \x circle (2pt); 
\draw (0,0)--(4,0); 
\draw (0,0) to [out=70,in=110] (4,0); \draw (0,0) to [out=30,in=150] (4,0); 
\draw (0,0) to [out=-70,in=-110] (4,0); \draw (0,0) to [out=-30,in=-150] (4,0); 
\node [left] at (0,0) {\text{\tiny$-2$}}; 
\node[right] at (4,0)  {\text{\tiny$-3$}};
\end{tikzpicture}
\begin{tikzpicture}[scale=0.45]  
\node at (-1.25,2) {IX.b};
\foreach \x in {(0,0),(4,0)} \filldraw \x circle (2pt); 
\draw (0,0)--(4,0); 
\draw (0,0) to [out=70,in=110] (4,0); \draw (0,0) to [out=30,in=150] (4,0); 
\draw (0,0) to [out=-70,in=-110] (4,0); \draw (0,0) to [out=-30,in=-150] (4,0); 
\node [left] at (0,0) {\text{\tiny$-1$}}; 
\node[right] at (4,0)  {\text{\tiny$-4$}};
\end{tikzpicture}

\begin{tikzpicture}[scale=0.45]
\foreach \x in {(0,0),(0,3),(3,0),(3,3)} \filldraw \x circle (2pt); 
\draw (0,0)--(0,3); \draw (3,0)--(3,3);  
\node [left] at (2.5,1.6) {X.a};  
\node [left] at (0,0) {$-2$};
\node [left] at (0,3) {$-2$};
\node [right] at (3,3) {$-1$};
\node [right] at (3,0) {$-2$};
\draw (0,0)--(3,0);
\draw  [-]  (0,0) to [out=45,in=135] (3,0);
\draw  [-]  (0,0) to [out=-45,in=-135] (3,0);
\draw  [-]  (0,3) to [out=30,in=150] (3,3);
\draw  [-]  (0,3) to [out=-30,in=-150] (3,3);
\end{tikzpicture}
\begin{tikzpicture}[scale=0.45]
\foreach \x in {(0,0),(0,3),(3,0),(3,3)} \filldraw \x circle (2pt); 
\draw (0,0)--(0,3); \draw (3,0)--(3,3);  
\node [left] at (2.5,1.6) {X.b};  
\node [left] at (0,0) {$-2$};
\node [left] at (0,3) {$-2$};
\node [right] at (3,3) {$-2$};
\node [right] at (3,0) {$-1$};
\draw (0,0)--(3,0);
\draw  [-]  (0,0) to [out=45,in=135] (3,0);
\draw  [-]  (0,0) to [out=-45,in=-135] (3,0);
\draw  [-]  (0,3) to [out=30,in=150] (3,3);
\draw  [-]  (0,3) to [out=-30,in=-150] (3,3);
\end{tikzpicture}
\begin{tikzpicture}[scale=0.45]
\foreach \x in {(0,0),(0,3),(3,0),(3,3)} \filldraw \x circle (2pt); 
\draw (0,0)--(0,3); \draw (3,0)--(3,3);  
\node [left] at (2.5,1.6) {X.c};  
\node [left] at (0,0) {$-3$};
\node [left] at (0,3) {$-1$};
\node [right] at (3,3) {$-1$};
\node [right] at (3,0) {$-2$};
\draw (0,0)--(3,0);
\draw  [-]  (0,0) to [out=45,in=135] (3,0);
\draw  [-]  (0,0) to [out=-45,in=-135] (3,0);
\draw  [-]  (0,3) to [out=30,in=150] (3,3);
\draw  [-]  (0,3) to [out=-30,in=-150] (3,3);
\end{tikzpicture}

\begin{tikzpicture}[scale=0.45]
\foreach \x in {(0,0),(0,3),(3,0),(3,3)} \filldraw \x circle (2pt); 
\draw (0,0)--(0,3); \draw (3,0)--(3,3);  
\node [left] at (2.5,1.6) {X.d};  
\node [left] at (0,0) {$-3$};
\node [left] at (0,3) {$-2$};
\node [right] at (3,3) {$-1$};
\node [right] at (3,0) {$-1$};
\draw (0,0)--(3,0);
\draw  [-]  (0,0) to [out=45,in=135] (3,0);
\draw  [-]  (0,0) to [out=-45,in=-135] (3,0);
\draw  [-]  (0,3) to [out=30,in=150] (3,3);
\draw  [-]  (0,3) to [out=-30,in=-150] (3,3);
\end{tikzpicture}
\begin{tikzpicture}[scale=0.45]
\foreach \x in {(0,0),(0,3),(3,0),(3,3)} \filldraw \x circle (2pt); 
\draw (0,0)--(0,3); \draw (3,0)--(3,3);  
\node [left] at (2.5,1.6) {X.e};  
\node [left] at (0,0) {$-3$};
\node [left] at (0,3) {$-1$};
\node [right] at (3,3) {$-2$};
\node [right] at (3,0) {$-1$};
\draw (0,0)--(3,0);
\draw  [-]  (0,0) to [out=45,in=135] (3,0);
\draw  [-]  (0,0) to [out=-45,in=-135] (3,0);
\draw  [-]  (0,3) to [out=30,in=150] (3,3);
\draw  [-]  (0,3) to [out=-30,in=-150] (3,3);
\end{tikzpicture}

We obtain the following $4$-dimensional cells: 

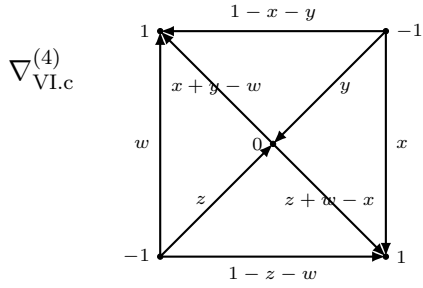
\begin{tikzpicture}[>=latex]
\foreach \x in {(0,0),(1.5,1.5),(3,0),(0,3),(3,3)} \filldraw \x circle (1pt); 
\node [left] at (-1,2.5) {$\nabla_{\mathrm{VI.c}}^{(4)}$};  
\draw  [->, thick]  (0,0) to  (3,0); \node [below] at (1.5,0) {\text{\tiny$1-z-w$}};
\draw  [->, thick]  (0,0) to (0,3); \node [left] at (0,1.5) {\text{\tiny$w$}}; 
\draw [->, thick] (0,0) to (1.5,1.5); \node [left] at (0.75,0.75) {\text{\tiny$z$}}; 
\draw  [->, thick]  (1.5,1.5) to (3,0); \node at (2.25,0.75) {\text{\tiny$z+w-x$}};
\draw  [->, thick]  (1.5,1.5) to (0,3); \node at (0.75,2.25)   {\text{\tiny$x+y-w$}}; 
\draw  [->, thick]  (3,3) to (0,3); \node [above] at (1.5,3) {\text{\tiny$1-x-y$}};
\draw  [->, thick]  (3,3) to (1.5,1.5); \node [right] at (2.25,2.25) {\text{\tiny$y$}};
\draw  [->, thick]  (3,3) to (3,0); \node [right] at (3,1.5) {\text{\tiny$x$}};
\node [left] at (0,0) {\text{\tiny{$-1$}}}; \node [left] at (0,3) {\text{\tiny{$1$}}}; \node [right] at (3,0) {\text{\tiny{$1$}}};
\node [right] at (3,3) {\text{\tiny{$-1$}}}; \node [left] at (1.5,1.5) {\text{\tiny{$0$}}}; 
\node [right] at (5,1.5) {$\begin{array}{c|c|c|c|c|c|c}
v_0 & v_1 & v_2 & v_3 & v_4 & v_5 & v_6\\
\hline 
0&0&1&1&0&0&0\\ 0&0&0&0&1&1&1\\ 0&1&1&0&0&1&0 \\ 0&0&0&1&0&0&1
\end{array}$}; 
\end{tikzpicture}

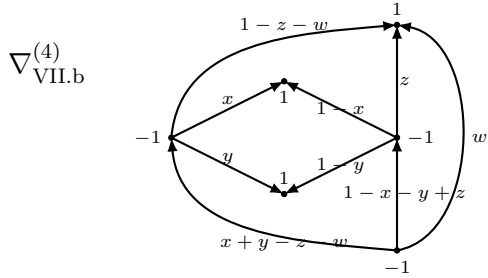
\begin{tikzpicture}[>=latex]
\foreach \x in {(0,0),(3,0),(3,1.5),(3,-1.5),(1.5,0.75),(1.5,-0.75)} \filldraw \x circle (1pt); 
\node [left] at (-1,1) {$\nabla_{\mathrm{VII.b}}^{(4)}$};  
\draw  [->, thick]  (0,0) to  [out= 85,in=-175] (3,1.5); \node at (1.5,1.5) {\text{\tiny$1-z-w$}};
\draw  [->, thick]  (3,-1.5) to [out=175,in=-85] (0,0); \node at (1.5,-1.4) {\text{\tiny$x+y-z-w$}}; 
\draw [->, thick] (3,-1.5) to (3,0); \node at (3.1,-0.75) {\text{\tiny$1-x-y+z$}}; 
\draw  [->, thick]  (3,0) to (3,1.5); \node  at (3.1,0.75) {\text{\tiny$z$}};
\draw  [->, thick]  (3,-1.5) to [out=10,in=-10] (3,1.5); \node at (4.1,0)   {\text{\tiny$w$}}; 
\draw  [->, thick]  (0,0) to (1.5,0.75); \node  at (0.75,0.5) {\text{\tiny$x$}};
\draw  [->, thick]  (0,0) to (1.5,-0.75); \node  at (0.75,-0.275) {\text{\tiny$y$}};
\draw  [->, thick]  (3,0) to (1.5,-0.75); \node  at (2.25,-0.375) {\text{\tiny$1-y$}};
\draw  [->, thick]  (3,0) to (1.5,0.75); \node at (2.25,0.375) {\text{\tiny$1-x$}};
\node [left] at (0,0) {\text{\tiny{$-1$}}}; \node [above] at (3,1.5) {\text{\tiny{$1$}}}; 
\node [below] at (3,-1.5) {\text{\tiny{$-1$}}};
\node [below] at (1.5,0.75) {\text{\tiny{$1$}}}; \node [above] at (1.5,-0.75) {\text{\tiny{$1$}}}; 
\node [right] at (3,0) {\text{\tiny{$-1$}}}; 
\node [right] at (5,0.5) {$\begin{array}{c|c|c|c|c|c|c|c}
v_0 & v_1 & v_2 & v_3 & v_4 & v_5 & v_6& v_7\\
\hline 
0&1&0&1&0&1&1&0\\ 0&0&1&0&1&1&0&1\\ 0&0&0&1&1&1&0&0 \\ 0&0&0&0&0&0&1&1
\end{array}$}; 
\end{tikzpicture}

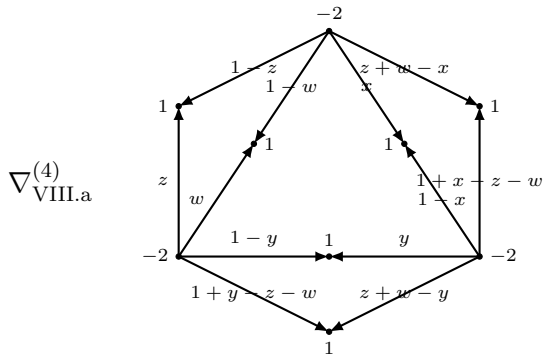
\begin{tikzpicture}[>=latex]
\foreach \x in {(0,0),(4,0),(2,3),(0,2),(1,1.5),(3,1.5),(4,2),(2,0),(2,-1)} \filldraw \x circle (1pt); 
\node [left] at (-1,1) {$\nabla_{\mathrm{VIII.a}}^{(4)}$};  
\draw  [->, thick]  (0,0) to   (2,0); \node [above] at (1,0) {\text{\tiny$1-y$}};
\draw  [->, thick]  (0,0) to  (2,-1); \node at (1,-0.5) {\text{\tiny$1+y-z-w$}}; 
\draw [->, thick] (4,0) to (2,0); \node [above] at (3,0) {\text{\tiny$y$}}; 
\draw  [->, thick]  (4,0) to (2,-1); \node  at (3,-0.5) {\text{\tiny$z+w-y$}};
\draw  [->, thick]  (0,0) to  (1,1.5); \node [right] at (0,0.75)   {\text{\tiny$w$}}; 
\draw  [->, thick]  (0,0) to (0,2); \node  [left] at (0,1) {\text{\tiny$z$}};
\draw  [->, thick]  (2,3) to (1,1.5); \node  at (1.5,2.25) {\text{\tiny$1-w$}};
\draw  [->, thick]  (2,3) to (0,2); \node  at (1,2.5) {\text{\tiny$1-z$}};
\draw  [->, thick]  (2,3) to (3,1.5); \node at (2.5,2.25) {\text{\tiny$x$}};
\draw  [->, thick]  (2,3) to (4,2); \node at (3,2.5) {\text{\tiny$z+w-x$}};
\draw  [->, thick]  (4,0) to (3,1.5);\node at (3.5,0.75) {\text{\tiny$1-x$}};
\draw  [->, thick]  (4,0) to (4,2); \node at (4,1) {\text{\tiny$1+x-z-w$}};
\node [left] at (0,0) {\text{\tiny{$-2$}}}; \node [above] at (2,3) {\text{\tiny{$-2$}}}; 
\node [right] at (4,0) {\text{\tiny{$-2$}}};
\node [right] at (1,1.5) {\text{\tiny{$1$}}}; \node [left] at (0,2) {\text{\tiny{$1$}}}; 
\node [above] at (2,0) {\text{\tiny{$1$}}}; \node [below] at (2,-1) {\text{\tiny{$1$}}}; 
\node [left] at (3,1.5) {\text{\tiny{$1$}}}; \node [right] at (4,2) {\text{\tiny{$1$}}}; 
\node [right] at (5,0.5) {$\begin{array}{c|c|c|c|c|c|c|c|c|c}
v_0 & v_1 & v_2 & v_3 & v_4 & v_5 & v_6& v_7&v_8&v_9\\
\hline 
0&0&1&0&0&1&0&1&1&1\\ 0&0&0&1&0&0&1&1&1&1\\ 0&1&1&1&0&0&0&1&1&0 \\ 0&0&0&0&1&1&1&1&0&1
\end{array}$}; 
\end{tikzpicture}

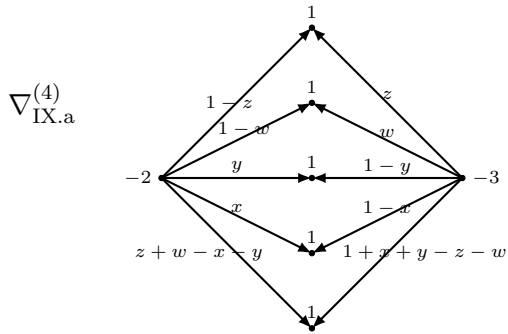
\begin{tikzpicture}[>=latex]
\foreach \x in {(0,0),(4,0),(2,0),(2,1),(2,2),(2,-1),(2,-2)} \filldraw \x circle (1pt); 
\node [left] at (-1,1) {$\nabla_{\mathrm{IX.a}}^{(4)}$};  
\draw  [->, thick]  (0,0) to   (2,0); \node at (1,0.15) {\text{\tiny$y$}};
\draw  [->, thick]  (0,0) to  (2,-1); \node at (1,-0.4) {\text{\tiny$x$}}; 
\draw [->, thick] (0,0) to (2,-2); \node at (0.5,-1) {\text{\tiny$z+w-x-y$}}; 
\draw  [->, thick]  (0,0) to (2,1); \node  at (1.1,0.65) {\text{\tiny$1-w$}};
\draw  [->, thick]  (0,0) to  (2,2); \node at (0.9,1)   {\text{\tiny$1-z$}}; 
\draw  [->, thick]  (4,0) to (2,0); \node at (3,0.15) {\text{\tiny$1-y$}};
\draw  [->, thick]  (4,0) to (2,1); \node  at (3,0.6) {\text{\tiny$w$}};
\draw  [->, thick]  (4,0) to (2,2); \node  at (3,1.1) {\text{\tiny$z$}};
\draw  [->, thick]  (4,0) to (2,-1);\node at (3,-0.4) {\text{\tiny$1-x$}};
\draw  [->, thick]  (4,0) to (2,-2); \node at (3.5,-1) {\text{\tiny$1+x+y-z-w$}};
\node [left] at (0,0) {\text{\tiny{$-2$}}}; \node [right] at (4,0) {\text{\tiny{$-3$}}}; 
\node [above] at (2,0) {\text{\tiny{$1$}}}; \node [above] at (2,1) {\text{\tiny{$1$}}}; \node [above] at (2,2) {\text{\tiny{$1$}}};
\node [above] at (2,-1) {\text{\tiny{$1$}}}; \node [above] at (2,-2) {\text{\tiny{$1$}}};
\node [right] at (5,0.5) {$\begin{array}{c|c|c|c|c|c|c|c|c|c}
v_0 & v_1 & v_2 & v_3 & v_4 & v_5 & v_6& v_7&v_8&v_9\\
\hline 
0&0&1&0&0&1&0&1&0&1\\ 0&0&0&1&0&0&1&0&1&1\\ 0&0&0&0&1&1&1&1&1&1 \\ 0&1&1&1&0&0&0&1&1&1
\end{array}$}; 
\end{tikzpicture}

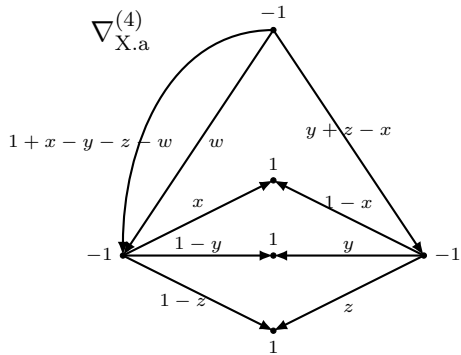
\begin{tikzpicture}[>=latex]
\foreach \x in {(0,0),(4,0),(2,3),(2,0),(2,-1),(2,1)} \filldraw \x circle (1pt); 
\node [left] at (0.5,3) {$\nabla_{\mathrm{X.a}}^{(4)}$};  
\draw  [->, thick]  (0,0) to  (2,-1); \node at (0.8,-0.6) {\text{\tiny$1-z$}}; 
\draw  [->, thick]  (0,0) to   (2,0); \node [above] at (1,-0.1) {\text{\tiny$1-y$}};
\draw  [->, thick]  (0,0) to  (2,1); \node at (1,0.7) {\text{\tiny$x$}}; 
\draw [->, thick] (4,0) to (2,0); \node [above] at (3,-0.1) {\text{\tiny$y$}}; 
\draw  [->, thick]  (4,0) to (2,-1); \node  at (3,-0.7) {\text{\tiny$z$}};
\draw  [->, thick]  (4,0) to (2,1); \node  at (3,0.7) {\text{\tiny$1-x$}};
\draw  [->, thick]  (2,3) to (4,0); \node  at (3,1.65) {\text{\tiny$y+z-x$}};
\draw  [->, thick]  (2,3) to  (0,0); \node [right] at (1,1.5)   {\text{\tiny$w$}}; 
\draw  [->, thick]  (2,3) to [out=180,in=90]  (0,0); \node  [left] at (0.8,1.5) {\text{\tiny$1+x-y-z-w$}};
\node [left] at (0,0) {\text{\tiny{$-1$}}}; \node [above] at (2,3) {\text{\tiny{$-1$}}}; 
\node [right] at (4,0) {\text{\tiny{$-1$}}};
\node [above] at (2,0) {\text{\tiny{$1$}}}; \node [below] at (2,-1) {\text{\tiny{$1$}}}; 
\node [above] at (2,1) {\text{\tiny{$1$}}};
\node [right] at (4.5,1.5) {$\begin{array}{c|c|c|c|c|c|c|c|c}
v_0 & v_1 & v_2 & v_3 & v_4 & v_5 & v_6& v_7&v_8\\
\hline 
0&0&0&0&1&1&1&1&1\\ 0&1&0&0&1&1&0&0&1\\ 0&0&1&0&0&0&1&1&1 \\ 0&0&0&1&0&1&0&1&0
\end{array}$}; 
\end{tikzpicture}

In fact we get 
\begin{itemize} 
\item[] $\nabla_{\mathrm{II.c}}^{(4)}$ from 
$\mathrm{VI.a}$, $\mathrm{VI.d}$, 
$\mathrm{VII.c}$, $\mathrm{IX.b}$, $\mathrm{X.d}$; \item[] $\nabla_{\mathrm{II.f}}^{(4)}$ from 
$\mathrm{X.e}$; 
\item[] $\nabla_{\mathrm{I.d}}^{(4)}$ from 
$\mathrm{VI.b}$, $\mathrm{VII.a}$, 
$\mathrm{VII.d}$, $\mathrm{VIII.b}$;
\item[] $\nabla_{\mathrm{VI.c}}^{(4)}$ from 
$\mathrm{VI.c}$;
\item[] $\nabla_{\mathrm{VII.b}}^{(4)}$ from 
$\mathrm{VII.b}$; 
\item[] $\nabla_{\mathrm{VIII.a}}^{(4)}$ from 
$\mathrm{VIII.a}$; 
\item[] $\nabla_{\mathrm{IX.a}}^{(4)}$ from 
$\mathrm{IX.a}$
\item[] $\nabla_{\mathrm{X.a}}^{(4)}$ from 
$\mathrm{X.a}$.  
\end{itemize} 
We obtain at most $3$-dimensional polytopes from 
$\mathrm{X.b}$ and $\mathrm{X.c}$. 
In particular, we have the following: 

\begin{proposition}
\begin{itemize}
\item[(i)] $\mathcal{P}(\Gamma_{\mathrm{VI}})=\{\nabla_{\mathrm{II.c}}^{(4)},\ \nabla_{\mathrm{I.d}}^{(4)}, \ \nabla_{\mathrm{VI.c}}^{(4)}\}$ 
\item[(ii)] $\mathcal{P}(\Gamma_{\mathrm{VII}})=\{\nabla_{\mathrm{II.c}}^{(4)},\ \nabla_{\mathrm{I.d}}^{(4)}, \ \nabla_{\mathrm{VII.b}}^{(4)}\}$ 
\item[(iii)]  $\mathcal{P}(\Gamma_{\mathrm{VIII}})=\{\nabla_{\mathrm{I.d}}^{(4)}, \ \nabla_{\mathrm{VIII.a}}^{(4)}\}$ 
\item[(iv)] $\mathcal{P}(\Gamma_{\mathrm{IX}})=\{\nabla_{\mathrm{II.c}}^{(4)}, \ \nabla_{\mathrm{IX.a}}^{(4)}\}$ 
\item[(iv)] $\mathcal{P}(\Gamma_{\mathrm{X}})=
\{\nabla_{\mathrm{II.c}}^{(4)},\nabla_{\mathrm{II.f}}^{(4)}, \ \nabla_{\mathrm{X.a}}^{(4)}\}$ 
\end{itemize}
\end{proposition}

As in Proposition~\ref{prop:triangulation of 4-cells}, 
one can easily construct pulling triangulations of the four compressed flow polytopes appearing in this section. 
In Proposition~\ref{prop:triangulation of new 4-cells} below we use the notation  introduced in the proof of 
 Proposition~\ref{prop:triangulation of 4-cells}.

\begin{proposition}\label{prop:triangulation of new 4-cells} 
\begin{itemize}
\item[(i)] Pulling at $v_5$ yields the triangulation of $\nabla_{\mathrm{VI.c}}^{(4)}$ with $4$-dimensional cells 
\[\langle 0,1,2,3,5\rangle,\  \langle 0,3,4,5,6 \rangle 
,\  \langle 0,1,3,5,6 \rangle ,\  \langle 0,2,3,4,5 \rangle \] 
and minimal  non-faces 
\[\{v_1,v_4\},\  \{v_2,v_6\}. \]  
\item[(ii)] Pulling at $v_0$, $v_4$ and then at $v_7$ yields the triangulation of $\nabla_{\mathrm{VII.b}}^{(4)}$ 
with $4$-dimensional cells 
\begin{align*}\langle 0, 1, 3, 5, 6\rangle,\ \langle 0, 3, 4, 5, 6\rangle,\ \langle 0, 4, 5, 6, 7\rangle,
\\ \langle 0, 1, 5, 6, 7\rangle, \ \langle 0, 1, 2, 5, 7\rangle,
\ \langle 0, 2, 4, 5, 7\rangle\end{align*} 
and  minimal  non-faces are 
\[\{v_2,v_6\}, \ \{v_3,v_7\}, \ \{v_1,v_4\}, \ \{v_2,v_3\}.\] 
\item[(iii)] Pulling at $v_0$, $v_2$ and then at $v_6$ yields the triangulation of $\nabla_{\mathrm{VIII.a}}^{(4)}$ with 
$4$-dimensional cells 
\begin{align*}\langle 0,2,7,8,9 \rangle,\  \langle 0,2,5,7,9 \rangle ,\  
\langle 0,2,3,7,8 \rangle ,\  \langle 0,1,2,3,7 \rangle,  
\\   \langle 0,2,4,5,7 \rangle ,\  \langle 0,1,2,4,7 \rangle ,\   \langle 0,6,7,8,9 \rangle ,\   \langle 0,3,6,7,8 \rangle,  
\\  \langle 0,5,6,7,9 \rangle ,\   \langle 0,4,5,6,7 \rangle ,\   \langle 0,1,3,6,7 \rangle ,\  
   \langle 0,1,4,6,7 \rangle 
\end{align*}
and minimal  non-faces 
\begin{align*}\{v_1,v_5\},\ \{v_1,v_8\},\ \{v_1,v_9\},\ \{v_2,v_6\}, 
\{v_3,v_4\},
\\ \{v_3,v_5\},\ \{v_3,v_9\},\ \{v_4,v_8\},\ \{v_4,v_9\},\ \{v_5,v_8\}.
\end{align*}   
\item[(iv)] Pulling at $v_0$, $v_1$ and then at $v_6$ yields the triangulation of $\nabla_{\mathrm{IX.a}}^{(4)}$ 
with $4$-dimensional cells 
\begin{align*}\nabla_{\mathrm{IX.a}}^{(4)}=\langle 0, 2, 5, 7, 9 \rangle ,\  \langle 0, 3, 6, 8, 9 \rangle ,\ 
\langle 0, 1, 4, 7, 8 \rangle ,\  \langle 0, 1, 2, 7, 9 \rangle,  
\\   \langle 0, 1, 7, 8, 9  \rangle ,\  
 \langle 0, 1, 2, 3, 9 \rangle ,\   \langle 0, 1, 3, 8, 9\rangle ,\   \langle 0, 5, 6, 7, 9 \rangle,   
 \\   \langle 0, 6, 7, 8, 9 \rangle ,\   \langle 0, 4, 5, 6, 7 \rangle ,\   \langle 0, 4, 6, 7, 8 \rangle   
 \end{align*}  
and minimal  non-faces 
\begin{align*}\{v_4,v_9\},\ \{v_3,v_7\},\ \{v_1,v_6\},\ \{v_3,v_4\},\  
\{v_3,v_5\},\ \{v_5,v_8\},\\ 
\{v_2,v_8\},\ \{v_2,v_6\},\ \{v_1,v_5\},\ \{v_2,v_4\}.\end{align*}  
\item[(v)] Pulling at $v_0$, $v_6$ and then at $v_3$ yields the triangulation of $\nabla_{\mathrm{X.a}}^{(4)}$ 
with $4$-dimensional cells 
\begin{align*}\nabla_{\mathrm{X.a}}^{(4)}=\langle 0, 1, 4, 5, 8 \rangle ,\  \langle 0, 5, 6, 7, 8 \rangle ,\ 
\langle 0, 4, 5, 6, 8\rangle ,  \langle 0, 2, 6, 7, 8 \rangle,  \\
   \langle 0, 3, 5, 7, 8  \rangle ,\  
 \langle 0, 1, 3, 5, 8 \rangle, 
 \langle 0, 2, 3, 7, 8  \rangle ,\  
 \langle 0, 1, 2, 3, 8 \rangle
 \end{align*}  
and minimal  non-faces 
\begin{align*}\{v_2,v_5\},\ \{v_3,v_4\},\ \{v_3,v_6\},\ \{v_1,v_7\},\  
\{v_4,v_7\},\ \{v_1,v_6\},\\ 
\{v_2,v_4\}.\end{align*}  
\end{itemize}
\end{proposition}

We record the complete list of 
$4$-dimensional compressed prime flow polytopes 
that was obtained in Section~\ref{sec:4-cells} 
and in the present section: 

\begin{proposition}\label{prop:all 4-dim compressed} 
Up to equivalence there are $11$ prime compressed 
$4$-dimensional flow polytopes, namely 
\[\nabla_{\mathrm{II.c}}^{(4)}, 
\nabla_{\mathrm{I.d}}^{(4)},
\nabla_{\mathrm{II.f}}^{(4)},
\nabla_{\mathrm{III.d}}^{(4)},
\nabla_{\mathrm{I.a}}^{(4)},
\nabla_{\mathrm{V.b}}^{(4)},
\nabla_{\mathrm{VI.c}}^{(4)},
\nabla_{\mathrm{VII.b}}^{(4)},
\nabla_{\mathrm{VIII.a}}^{(4)},
\nabla_{\mathrm{IX.a}}^{(4)},
\nabla_{\mathrm{X.a}}^{(4)}.\] 
\end{proposition} 

\begin{example} \label{example:4 dim compressed non-flow} 
We give an example of a compressed $0-1$ polytope in $\mathbb{R}^4$ that is not a flow polytope (compare with 
Remark~\ref{remark:all compressed are flow}).  
Consider the following seven $0-1$ vectors in $\mathbb{R}^4$: 
\[\begin{array}{c|c|c|c|c|c|c}
v_0 & v_1 & v_2 & v_3 & v_4 & v_5 & v_6\\
\hline 
0&1&1&0&1&0&0
\\ 0&1&0&1&1&1&0
\\ 0&0&1&1&1&0&1 
\\ 0&0&0&1&0&0&0
\end{array}\]
Their convex hull $\nabla$ is a cone with apex $v_3$ over the facet 
$\mathrm{Conv}(v_0,v_1,v_2,v_4,v_5,v_6)$. 
It has $9$  facets, in the notation of Proposition~\ref{prop:triangulation of 4-cells} they are the following: 
\begin{align*}\langle 0,1,2,4,5,6\rangle,\ \langle 0,3,5,6 \rangle, \ \langle 1,3,4,5\rangle, \ 
\langle 2,3,4,6\rangle, \ \langle 0,1,2,3\rangle,  \\ \ \langle 3,4,5,6\rangle, 
\ \langle 1,2,3,4\rangle, \ \langle 0,2,3,6\rangle, \ \langle 0,1,3,5 \rangle 
\end{align*} 
The corresponding facet-inequalities 
defining $\nabla$ are 
\begin{align*} 
w\ge0, \quad x\ge 0, \quad y\le 1, \quad 
z\le 1, \quad x-y-z+2w\le 0, \quad \\ x-y-z+w\ge -1, \quad 
x+w\le 1, \quad w\le y,\quad w\le z 
\end{align*} 
The polytope $\nabla$ is compressed, but is not equivalent to a flow polytope. Indeed, it is $4$-dimensional and  prime  (since the number of its vertices is prime), and among the $11$ compressed prime $4$-dimensional flow polytopes there is none having $7$ vertices and $9$ facets. 
\end{example} 


\section{Ehrhart polynomials}\label{sec:ehrhart}  

Recall that the \emph{Ehrhart polynomial} of a lattice polytope $\nabla$ is the polynomial  $L_{\nabla}$ with rational 
coefficients such that for any positive integer $n$, the value $L_{\nabla}(n)$ is the number of lattice points in the polytope 
$n\nabla$. 
It is known that the Ehrhart polynomial of a lattice polytope is determined by the $f$-vector of an unimodular triangulation and vice versa (see \cite[Theorem 9.3.25]{lorea-rambau-santos}). Using a somewhat different logic we compute the Ehrhart polynomials   of the $3$ and $4$-dimensional prime  compressed flow polytopes from  the triangulations constructed in the previous sections. 

\begin{proposition}\label{prop:ehrhart} 
The Ehrhart polynomials of the $3$ and $4$-dimensional prime compressed flow polytopes are given 
in the table below: 
\[
\begin{array}{c|c}
 \nabla & L_{\nabla}(n) \\ \hline 
\nabla_{\mathrm{I.a}}^{(3)} & \binom{n+3}{3}=\frac{1}{6}(n^3+6n^2+11n+6) \\
\nabla_{\mathrm{II.c}}^{(3)} & \binom{n+3}{3}+\binom{n+2}{3}=\frac{1}{6}(2n^3+9n^2+13n+6) \\
\nabla_{\mathrm{I.b}}^{(3)} & \binom{n+3}{3}+2\binom{n+2}{3}+\binom{n+1}{3}=
\frac{1}{6}(4n^3+12n^2+14n+6) \\
\nabla_{\mathrm{II.c}}^{(4)} & \binom{n+4}{4}=\frac{1}{24}(n^4 +10n^3 +35n^2 +50n +24 \\
\nabla_{\mathrm{I.d}}^{(4)} & \binom{n+4}{4}+\binom{n+3}{4}=\frac{1}{24}(2n^4+16n^3+46n^2+56n+24) \\
\nabla_{\mathrm{II.f}}^{(4)} & \binom{n+4}{4}+2\binom{n+3}{4}=\frac{1}{24}(3n^4+22n^3+57n^2+62n+24)\\
\nabla_{\mathrm{III.d}}^{(4)} & \binom{n+4}{4}+3\binom{n+3}{4}+\binom{n+2}{4}=
\frac{1}{24}(5n^4+30n^3+67n^2+66n+24) \\
\nabla_{\mathrm{I.a}}^{(4)} & \binom{n+4}{4}+4\binom{n+3}{4}+\binom{n+2}{4}
=\frac{1}{24}(6n^4+36n^3+78n^2+72n+24)\\
\nabla_{\mathrm{V.b}}^{(4)} & \binom{n+4}{4}+\binom{n+3}{4}+\binom{n+2}{4}=
\frac{1}{24}(3n^4+18n^3+45n^2+54n+24) \\
\nabla_{\mathrm{VI.c}}^{(4)} & \binom{n+4}{4}+2\binom{n+3}{4}+\binom{n+2}{4}
=\frac{1}{24}(4n^4 +24n^3 +56n^2 +60n +24) \\
\nabla_{\mathrm{VII.b}}^{(4)} & \binom{n+4}{4}+3\binom{n+3}{4}+2\binom{n+2}{4}=
\frac{1}{24}(6n^4 +32n^3 +66n^2 +64n +24) \\
\nabla_{\mathrm{VIII.a}}^{(4)} & \binom{n+4}{4}+5\binom{n+3}{4}+5\binom{n+2}{4}+\binom{n+1}{4}
=\frac{1}{24}(12n^4 +48n^3 +84n^2 +72n +24)\\
\nabla_{\mathrm{IX.a}}^{(4)} & \binom{n+4}{4}+5\binom{n+3}{4}+5\binom{n+2}{4}
=\frac{1}{24}(11n^4 +50n^3 +85n^2 +70n +24)
\\
\nabla_{\mathrm{X.a}}^{(4)} & \binom{n+4}{4}+4\binom{n+3}{4}+
3\binom{n+2}{4}
=\frac{1}{24}(8n^4+40n^3+76n^2+68n+24)
\end{array} \]
\end{proposition} 

\begin{proof}
We use the notation of Section~\ref{sec:toric ideals}. The value $L_{\nabla}(n)$ equals the dimension over $K$ of 
the degree $n$ homogeneous component of the factor algebra $K[t_z\mid z\in \nabla\cap \mathbb{Z}^d]/\mathcal{I}_{\nabla}$ 
(here the polynomial algebra is endowed with the standard grading; that is, the degree $1$ homogeneous component 
is spanned by the variables $t_z$). This dimension equals the number of degree $n$ standard monomials in 
$K[t_z\mid z\in \nabla\cap \mathbb{Z}^d]$ with respect to a Gr\"obner basis of $\mathcal{I}_{\nabla}$. 
Now for a $3$ or $4$-dimensional prime compressed flow polytope $\nabla$ consider the Gr\"obner basis whose initial monomials 
correspond to the minimal non-faces for the triangulation of $\nabla$ given in Proposition~\ref{prop:triangulation of 4-cells} or Proposition~\ref{prop:triangulation of new 4-cells}. Since this triangulation is unimodular, our initial ideal is 
generated by square-free monomials (see \cite[Corollary 8.9]{sturmfels}), and the corresponding initial complex 
is the triangulation of $\nabla$ we started with. It follows that the standard monomials are exactly the monomials 
for which there exists a maximal cell $C$ in the triangulation such that the support of the monomial belongs to 
$\langle t_z\mid z\in C\rangle$ (the subsemigroup of $K[t_z\mid z\in \nabla\cap\mathbb{Z}^d]$ generated by 
$\{t_z\mid z\in C\}$). Suppose that $\nabla=C_1\cup\dots\cup C_m$ is the given triangulation.  
Below we present for $\nabla$ monomials $u_1,\dots,u_m$ such that the set of standard monomials 
for the chosen Gr\"obner basis of $\mathcal{I}_{\nabla}$ is $\bigsqcup_{i=1}^mw_i\langle t_z\mid z\in C_i\rangle$ 
(disjoint union). 
From this presentation of the set of standard monomials one can easily read off the number of standard 
monomials in each degree, thereby verifying the table in the statement. 
\begin{align*} \nabla_{\mathrm{I.a}}^{(3)}: & \quad \langle t_0,t_1,t_2,t_3\rangle \\
\nabla_{\mathrm{II.c}}^{(3)}: & \quad \langle t_0,t_2,t_3,t_4\rangle \sqcup t_1\langle t_0,t_1,t_3,t_4\rangle \\
\nabla_{\mathrm{I.b}}^{(3)}: & \quad \langle t_0,t_1,t_2,t_5\rangle \sqcup t_3\langle t_0,t_1,t_3,t_5\rangle 
\sqcup t_4\langle t_0,t_2,t_4,t_5\rangle\sqcup t_3t_4\langle t_0,t_3,t_4,t_5\rangle \\ 
\nabla_{\mathrm{II.c}}^{(4)}: & \quad \langle t_0,t_1,t_2,t_3,t_4 \rangle  \\
\nabla_{\mathrm{I.d}}^{(4)}: & \quad \langle t_0,t_1,t_2,t_3,t_5\rangle \sqcup t_4\langle t_0,t_1,t_2,t_4,t_5 \rangle \\
\nabla_{\mathrm{II.c}}^{(4)}: & \quad \langle t_0,t_1,t_2,t_3,t_5\rangle \sqcup t_6\langle t_0,t_2,t_3,t_5,t_6\rangle \sqcup
t_4\langle t_0,t_2,t_3,t_4,t_6\rangle \\
\nabla_{\mathrm{III.d}}^{(4)}: & \quad \langle t_0,t_1,t_2,t_3,t_5\rangle \sqcup t_7\langle t_0,t_1,t_3,t_5,t_7\rangle \sqcup
t_4\langle t_1,t_2,t_3,t_4,t_5\rangle \\ &\quad   \sqcup  t_6\langle t_0,t_3,t_5,t_6,t_7\rangle \sqcup 
t_4t_7\langle t_1,t_3,t_4,t_5,t_7\rangle \\
\nabla_{\mathrm{I.a}}^{(4)}: & \quad \langle t_0,t_4,t_5,t_6,t_7\rangle \sqcup t_8\langle t_0,t_4,t_6,t_7,t_8\rangle \sqcup
t_1\langle t_0,t_1,t_4,t_5,t_6\rangle \\ & \quad  \sqcup t_3\langle t_0,t_3,t_4,t_5,t_7\rangle\sqcup 
t_2\langle t_0,t_1,t_2,t_4,t_6\rangle\sqcup t_1t_3\langle t_0,t_1,t_3,t_4,t_5\rangle \\ 
\nabla_{\mathrm{V.b}}^{(4)}: & \quad \langle t_0,t_2,t_3,t_4,t_5\rangle \sqcup t_1\langle t_0,t_1,t_2,t_3,t_5\rangle \sqcup
t_1t_4\langle t_0,t_1,t_2,t_4,t_5\rangle  \\
\nabla_{\mathrm{VI.c}}^{(4)}: & \quad \langle t_0,t_1,t_2,t_3,t_5\rangle \sqcup t_6\langle t_0,t_1,t_3,t_5,t_6\rangle \sqcup
t_4\langle t_0,t_3,t_4,t_5,t_6\rangle \sqcup  t_2t_4\langle t_0,t_2,t_3,t_4,t_5\rangle\\
\nabla_{\mathrm{VII.b}}^{(4)}: & \quad \langle t_0,t_1,t_3,t_5,t_6\rangle \sqcup t_4\langle t_0,t_3,t_4,t_5,t_6\rangle \sqcup
t_7\langle t_0,t_4,t_5,t_6,t_7\rangle \\ & \quad  \sqcup  t_1t_7\langle t_0,t_1,t_5,t_6,t_7\rangle 
\sqcup t_2\langle t_0,t_1,t_2,t_5,t_7\rangle \sqcup t_2t_4\langle t_0,t_2,t_4,t_5,t_7\rangle\\
\nabla_{\mathrm{VIII.a}}^{(4)}: & \quad \langle t_0,t_2,t_7,t_8,t_9\rangle \sqcup t_5\langle t_0,t_2,t_5,t_7,t_9\rangle \sqcup
t_3\langle t_0,t_2,t_3,t_7,t_8\rangle \\ & \quad  \sqcup  t_6\langle t_0,t_6,t_7,t_8,t_9\rangle 
\sqcup t_1\langle t_0,t_1,t_2,t_3,t_7\rangle \sqcup t_4\langle t_0,t_1,t_2,t_4,t_7\rangle \\ &\quad 
\sqcup 
t_5t_6\langle t_0,t_5,t_6,t_7,t_9\rangle \sqcup t_3t_6\langle t_0,t_3,t_6,t_7,t_8\rangle \sqcup
t_4t_5\langle t_0,t_2,t_4,t_5,t_7\rangle \\ & \quad \sqcup  t_1t_6\langle t_0,t_1,t_3,t_6,t_7\rangle 
\sqcup t_4t_6\langle t_0,t_1,t_4,t_6,t_7\rangle \sqcup t_4t_5t_6\langle t_0,t_4,t_5,t_6,t_7\rangle \\
\nabla_{\mathrm{IX.a}}^{(4)}: & \quad \langle t_0,t_2,t_5,t_7,t_9\rangle \sqcup t_1\langle t_0,t_1,t_2,t_7,t_9\rangle \sqcup
t_8\langle t_0,t_1,t_7,t_8,t_9\rangle \\ &\quad  \sqcup  t_3\langle t_0,t_1,t_2,t_3,t_9\rangle 
\sqcup t_6\langle t_0,t_6,t_7,t_8,t_9\rangle \sqcup t_4\langle t_0,t_1,t_4,t_7,t_8\rangle \\ &\quad 
\sqcup 
t_5t_6\langle t_0,t_5,t_6,t_7,t_9\rangle \sqcup t_4t_6\langle t_0,t_4,t_6,t_7,t_8\rangle 
\sqcup t_4t_5\langle t_0,t_4,t_5,t_6,t_7\rangle \\ &\quad
\sqcup t_3t_8\langle t_0,t_1,t_3,t_8,t_9\rangle
\sqcup t_3t_6\langle t_0,t_3,t_6,t_8,t_9\rangle
\\
\nabla_{\mathrm{X.a}}^{(4)}: & \quad \langle t_0,t_1,t_4,t_5,t_8\rangle \sqcup t_6\langle t_0,t_4,t_5,t_6,t_8\rangle \sqcup
t_3\langle t_0,t_1,t_3,t_5,t_8\rangle \\ &\quad  \sqcup  t_2\langle t_0,t_1,t_2,t_3,t_8\rangle 
\sqcup t_7\langle t_0,t_5,t_6,t_7,t_8\rangle \sqcup t_3t_7\langle t_0,t_3,t_5,t_7,t_8\rangle \\ &\quad 
\sqcup 
t_2t_7\langle t_0,t_2,t_3,t_7,t_8\rangle \sqcup t_2t_6\langle t_0,t_2,t_6,t_7,t_8\rangle 
\end{align*} 
\end{proof}



\begin{thebibliography}{mmm}

\bibitem{altmann-nill-schwentner-wiercinska} K. Altmann,  B. Nill, S.  Schwentner and I. Wiercinska, 
Flow polytopes and the graph of reflexive polytopes, 
Discrete Math. 309 (2009), no. 16, 4992-4999. 

\bibitem{altmann-straten} K. Altmann and D.  van Straten, 
Smoothing of quiver varieties,  Manuscripta Math. 129 (2009), no. 2, 211-230. 

\bibitem{baldonisilva-deloera-vergne} W. Baldoni-Silva, J. A. De Loera, M. Vergne, Counting integer flows in networks,  Found. Comput.
Math. 4 (2004), 277-314. 

\bibitem{baldoni-vergne} W. Baldoni and M. Vergne, Kostant patition functions and flow polytopes, 
Transform. Groups 13 (2008), 447-469. 

\bibitem{balletti} G. Balletti, Enumeration of lattice polytopes by their volume. Discrete Comput. Geom. 65 (2021), 
1087-1122. 


\bibitem{beck-pixton} M. Beck and D. Pixton, The Ehrhart polynomial of the Birkhoff polytope, Discrete Comput. Geom. 30 (2003), 623-637. 

\bibitem{bruns} W. Bruns, The quest for counterexamples in toric geometry,  Proc. CAAG 2010, Ramanujan Math. Soc. Lect. Notes Series No. 17 (2013), 1-17.

\bibitem{bruns-gubeladze} W. Bruns, J. Gubeladze: Polytopes, Rings, and
K-theory, Springer-Verlag, New York (2009).


\bibitem{cox-little-schenck} D. Cox, J. Little, and H. Schenck, Toric Varieties, Amer. Math. Soc., 2010.


\bibitem{diaconis-eriksson} P. Diaconis and N. Eriksson,  Markov bases for noncommutative Fourier analysis of ranked data, J. Symbolic Comput. 41 (2006), 182-195.

\bibitem{diaconis-sturmfels} P. Diaconis and B. Sturmfels, Algebraic algorithms for sampling from conditional distributions, 
The Annals of Statistics 26 (1998), 363-397. 

\bibitem{domokos-joo}
M. Domokos and D. Jo\'o, On the equations and classification of toric quiver varieties, 
Proceedings of the Royal Society of Edinburgh: Section A Mathematics 146 (2016),  265-295.  

\bibitem{domokos-joo:2} M. Domokos and D. Jo\'o, 
Toric quiver cells, preprint, arXiv:1609.03618.  

\bibitem{haase-paffenholz} C. Haase and A. Paffenholz, Quadratic Gr\"obner bases for smooth $3 \times 3$ transportation polytopes, J. Algebr. Comb. 30 (2009), 477-489.



\bibitem{haase-etal} C. Haase, A. Paffenholz, L. C. Piechnik and F. Santos, Existence of unimodular triangulations - positive results, arXiv: 1405.1687, Memoirs of the American Mathematical Society, to appear.  

\bibitem{higashitani-ohsugi} Akihiro Higashitani and Hidefumi Ohsugi, 
Toric ideals of Minkowski sums of unit 
simplices, Algebraic Combinatorics 3 (2020), 831-837.   


\bibitem{hille} L. Hille, Tilting line bundles and moduli of thin sincere representations of quivers, An. St. Univ. Ovidius Constanza 4 (1996), 76-82. 

\bibitem{hille:canada} L. Hille, Toric quiver varieties, pp. 311-325, Canad. Math. Soc. Conf. Proc. 24, Amer. Math. Soc., Providence, RI, 1998. 


\bibitem{lorea-rambau-santos} J. A. De Lorea, J. Rambau, F. Santos, Triangulations: Structures for algorithms and applications, Springer, 2010.


\bibitem{meszaros-morales} K. M\'esz\'aros and A. H. Morales, Volumes and Ehrhart polynomials of flow polytopes, Matematische Zeitschrift 293 (2019), 1369-1401. 


\bibitem{ohsugi-hibi} Hidefumi Ohsugi and Takayuki Hibi, Convex polytopes all of whose reverse lexicographic initial ideals are squarefree, Proc. Am. Math. Soc. 129 (2001), 
2541-2546.

\bibitem{ohsugi-hibi_2010} Hidefumi Ohsugi and Takayuki Hibi, 
Toric rings and ideals of nested configurations, 
J. Commut. Algebra 2 (2010), no. 2, 187-208. 



\bibitem{schrijver} A. Schrijver, Combinatorial Optimization -- Polyhedra and Efficiency. Number 24A in Algorithms and Combinatorics, Springer, Berlin, 2003.  



\bibitem{sturmfels} B. Sturmfels, Gr\"obner Bases and Convex Polytopes, University Lecture Series 8, AMS, Providence, Rhode Island, 1996. 


\bibitem{sullivant} S. Sullivant, Compressed polytopes and statistical disclosure limitation, Tohoku Math. J. (2)  58 (2006), 433-445.

\bibitem{yamaguchi-ogawa-takemura} Takashi Yamaguchi, Mitsunori Ogawa and Akimichi Takemura, 
Markov degree of the Birkhoff model, J. Alg. Comb. 40 (2014), 293-311. 



\end{thebibliography}
\end{document}